\documentclass{amsart}

\newtheorem{theorem}{Theorem}[section]
\newtheorem{lemma}[theorem]{Lemma}

\theoremstyle{definition}
\newtheorem{definition}[theorem]{Definition}

\theoremstyle{remark}
\newtheorem{remark}[theorem]{Remark}

\numberwithin{equation}{section}

\usepackage{verbatim}
\usepackage{pgf}
\usepackage{tikz}

\begin{document}

\title[perturbed geodesics and Yang-Mills connections]{Perturbed geodesics on the moduli space of flat connections and Yang-Mills theory}

\author{R\'emi Janner}
\address{R\'emi Janner, Department of Mathematics, HG F28.2, ETH, R\"amistrasse 101, 8092 Zurich, Switzerland}
\email{remi.janner@math.ethz.ch}
\thanks{The author want to express his gratitude to Dietmar Salamon for suggesting this problem to him and for valuable discussions.  For partial financial support we are most grateful to the ETH Zurich (grant TH-0106-1).}



\date{\today.}



\begin{abstract}
If we consider the moduli space of flat connections of a non trivial principal $\textbf{SO}(3)$-bundle over a surface, then we can define a map from the set of perturbed closed geodesics, below a given energy level, into families of perturbed Yang-Mills connections depending on a parameter $\varepsilon$. In this paper we show that this map is a bijection and maps perturbed geodesics into perturbed Yang-Mills connections with the same Morse index.
\end{abstract}

 \maketitle

\tableofcontents

\section{Introduction}

The moduli space of flat connections for a principal bundle over a surface $\Sigma$ with genus $g$ is an infinite dimensional analogue of a symplectic reduction and was investigated for the first time in 1983 by Atiyah and Bott (cf. \cite{MR702806}) who showed that, on this particular moduli space, one can define an almost complex structure induced by the Hodge-*-operator acting on the 1-forms over $\Sigma$ and hence induced by its conformal structure; with the almost complex structure and the inner product on the 1-forms one can also obtain a symplectic form. Furthermore, if we choose a principal non trivial $\textrm{\bf SO}(3)$-bundle, then the moduli space $\mathcal M^g(P)$, defined as the quotient between the space of the flat connections $\mathcal A_0(P)\subset \mathcal A(P)$ and the identity component of the gauge group $\mathcal G_0(P)$, is a smooth compact symplectic manifold of dimension $6g-6$ (cf. \cite{MR1297130}). In the nineties some aspects of the topology of $\mathcal M^g(P)$ were investigated by Dostoglou and Salamon in \cite{MR1283871}, where they proved an isomorphism between the symplectic and the instanton Floer homology related to this moduli space, and in the work of Hong (cf. \cite{MR1715156}). Hong took an oriented compact manifold $B$ with a Riemannian metric $g_B$ and a harmonic map $\phi:B\to \mathcal M^g(P)$ and he showed that if the Jacobi operator of $\phi$ is invertible, then there exist a constant $\varepsilon_0$ and, for $0<\varepsilon<\varepsilon_0$, a family $A^\varepsilon$ of Yang-Mills connections of the principal $\textbf{SO}(3)$-bundles $P\times B\to \Sigma\times B$, where the base manifold has a partial rescaled metric $g_\Sigma\oplus \frac 1{\varepsilon^2}g_B$, which converges to the connection that generates $\phi$. In this paper we choose $B=S^1$ and a slightly different rescaling of the metric and we extend the results of Hong; more precisely the setting is the following one.\\

On the one hand, we consider the loop space on $\mathcal M^g(P)$ and its elements can be seen as connections $A(t)+\Psi(t) dt$ on a the manifold $\Sigma \times S^1$, where $A(t)\in \mathcal A_0(P)$ and $\Psi(t)$ is a $0$-form in $\Omega^0(\Sigma,\mathfrak g_P)$, satisfying the condition $d_A^*\left(\partial_tA-d_A\Psi\right)=0$. The 1-form $\partial_tA-d_A\Psi$ corresponds to the speed vector of our loop and thus the perturbed energy functional is  
\begin{equation}\label{intro:EH}
E^H(A)=\frac 12 \int_0^1\left(\left\|\partial_tA-d_A\Psi\right\|^2_{L^2(\Sigma)}-H_t(A)\right) dt
\end{equation}
where $H_t:\mathcal A(P)\to \mathbb R$ is a generic equivariant Hamiltonian map which is introduced in order to obtain an invertible second variational form. On the other hand, we can take the 3-manifold $\Sigma \times S^1$ with the metric $\varepsilon^2 g_\Sigma \oplus g_{S^1}$ for a positive parameter $\varepsilon$ and consider the principal $\textrm{\bf SO}(3)$-bundle $P\times S^1\to \Sigma\times S^1$. In this case, for a connection $\Xi=A+\Psi \,dt\in \mathcal A(P\times S^1)$, where $A(t)\in\mathcal A(P)$, $\Psi(t)\in\Omega^0(\Sigma,\mathfrak g_P)$ the curvature is $F_{\Xi}=F_{A}-(\partial_{t}A-d_{A}\Psi)\wedge dt$ and thus the perturbed Yang-Mills functional can be written as
 \begin{equation}\label{intro:YME}
 \mathcal{YM}^{\varepsilon,H}(\Xi)
  =\frac12\int_{0}^1 \left(\frac 1 {\varepsilon^2}\|F_{A}\|_{L^2(\Sigma)}^2
  +\|\partial_{t}A-d_{A}\Psi\|_{L^2(\Sigma)}^2-H_{t}(A)\right) dt.
\end{equation}
Then, by a contraction argument one can define a map between the perturbed geodesics below an energy level $b$, denoted by $\mathrm{Crit}^b_{E^{H}}$, and the set of the perturbed Yang-Mills connections $\mathrm{Crit}^b_{\mathcal{YM}^{\varepsilon,H}}$ with energy less than $b$ provided that the parameter $\varepsilon$ is small enough. Furthermore, this map can also be defined uniquely, it is surjective and maps perturbed geodesics to perturbed Yang-Mills connections with the same Morse index. Summarizing, in this paper, we prove the following theorem.

\begin{theorem}\label{thm:mainthm}
We assume that the Jacobi operators of all the perturbed geodesics are invertible and we choose a regular value $b$ of the energy $E^H$ and $p\geq 2$. Then there are two positive constants $\varepsilon_0$ and $c$ such that the following holds. For every $\varepsilon\in(0,\varepsilon_0)$ there is a unique gauge equivariant map
\begin{equation*}
\mathcal T^{\varepsilon,b}:\mathrm{Crit}^b_{E^H}
\to \mathrm{Crit}^b_{\mathcal{YM}^{\varepsilon,H}}
\end{equation*}
satisfying, for $\Xi^0 \in \mathrm{Crit}^b_{E^H}$, 
\begin{equation}\label{intro:A:a1}
d_{\Xi^0 }^{*_\varepsilon}\left(\mathcal T^{\varepsilon,b }(\Xi^0)-\Xi^0\right)=0,\quad\left\| \mathcal T^{\varepsilon,b}(\Xi^0)-\Xi^0
\right\|_{\Xi^0,2,p,\varepsilon}
\leq c \varepsilon^2.
\end{equation}
Furthermore, this map is bijective and $\textrm{index}_{E^H}(\Xi^0)=\textrm{index}_{\mathcal {YM}^{\varepsilon,H}}(\mathcal T^{\varepsilon,b}(\Xi^0))$.
\end{theorem}

\textbf{The result of Hong.} Hong could assume that the harmonic map $\phi$ has an invertible Jacobi operator because, even for an unperturbed energy functional, you can reach this condition for example for a 2-dimensional manifold $B$ and eventually slightly perturbing the metric $g_B$. For $B=S^1$ the Jacobi operator of a geodesic is never invertible and for this reason we need to introduce a perturbation in our functional $E^H$ as we will discuss in the section \ref{chapter:intro}. Another important point worth to be remarked is the different choice of the rescaling. On the one side, both choices give the same equations for the Yang-Mills connections, for $B=S^1$, if we do not consider the Hamiltonian perturbation, and hence his methods work also in our case; we can therefore say that Hong proved the existence of the map $\mathcal T^{\varepsilon,b}$. However, he did not prove its uniqueness and its surjectivity. On the other side, the different choice of the metric gives two different Yang-Mills energy functionals; in fact, using the metric $g_\Sigma\oplus\frac 1{\varepsilon^2}g_{S^1}$ one obtains the Yang-Mills energy functional $\varepsilon\mathcal{YM}^{\varepsilon,\frac 1\varepsilon H}$ instead of $ \mathcal{YM}^{\varepsilon,H}$ and the properties of $ \mathcal{YM}^{\varepsilon,H}$ will play a major role in the proof of the surjectivity of $\mathcal T^{\varepsilon,b}$ and in particular, to obtain the a priori estimates for the curvature of the perturbed Yang-Mills connections.\\

\textbf{Outline.} The second section is of preliminary nature; in fact, first, we briefly introduce the moduli space $\mathcal M^g(P):=\mathcal A_0(P)/\mathcal G_0(P)$ of flat connections of a non-trivial principal $\mathrm{\bf SO}(3)$-bundle $P$ over a surface $(\Sigma,\,g_\Sigma)$ of genus $g$. Then, on the one hand, we discuss the equations of the perturbed closed geodesics on $\mathcal M^g(P)$ and on the other hand, we introduce for a given $\varepsilon>0$ the equations for the perturbed Yang-Mills connections of the principal $\mathrm{\bf SO}(3)$-bundle $P\times S^1\to \Sigma\times S^1$ where the metric on $\Sigma$ is rescaled by a factor $\varepsilon^2$. Next, we also define the norm which will play a fundamental role in the proof of the theorem \ref{thm:mainthm}. In the successive two sections we compute the linear (section \ref{chapter:ellest}) and the quadratic (section \ref{chapter:qe}) estimates and in section \ref{section:construction}, we define the injective map $\mathcal T^{\varepsilon,b}$ and furthermore, we prove that this map is unique under the condition (\ref{intro:A:a1}).
In the next section, we show some a priori estimates (section \ref{chapter:aprioriestimates}) that we need to prove the surjectivity of the map $\mathcal T^{\varepsilon,b}$ (section \ref{chapter:surjectivity}). We prove the surjectivity of the map $\mathcal T^{\varepsilon,b}$ indirectly: We assume that there is a sequence of perturbed Yang-Mills connections $\Xi^{\varepsilon_\nu}$, $\varepsilon_\nu\to 0$, which is not in the image of $\mathcal T^{\varepsilon_\nu,b}$, and we show that this sequence has a subsequence  which converges to a geodesic $\Xi^0$; then using the uniqueness property of $\mathcal T^{\varepsilon,b}$ this subsequence turn out to be in the image of $\mathcal T^{\varepsilon_\nu,b}(\Xi^0)$ yielding a contradiction. In the last section, we conclude the proof of the theorem \ref{thm:mainthm} proving that $\mathcal T^{\varepsilon,b}$ maps perturbed geodesics to perturbed Yang-Mills connections with the same index (theorem \ref{thm:indexym}); in fact the theorem \ref{thm:mainthm} follows directly from the definition \ref{thm:defT} of the map $\mathcal T^{\varepsilon,b}$, its surjectivity (theorem \ref{thm:surjj}) and the index theorem \ref{thm:indexym}.\\

\begin{remark}
We denote by $\mathcal L^b\mathcal M^g(P)\subset \mathcal{LM}^g(P)$ and by $\mathcal A^{\varepsilon,b}(P\times S^1)$ respectively the subsets where $E^H\leq b$ and $\mathcal {YM}^{\varepsilon,H}\leq b$. Since we have a bijection between the critical points of the two functionals, we can also expect an isomorphism between the Morse homology, defined with the $L^2$-flows, of the bounded loop space $\mathcal L^b\mathcal M^g(P)$ and that of the moduli space $\mathcal A^{\varepsilon,b}(P\times S^1)/\mathcal G_0(P\times S^1)$, as it is explained in \cite{remyj3}:
\begin{theorem}
We assume that the energy functional $E^H$ is Morse-Smale. For every regular value $b>0$ of $E^H$ there is a positive constant $\varepsilon_0$ such that, for $0<\varepsilon<\varepsilon_0$, the inclusion $\mathcal L^b\mathcal M^g(P)\to \mathcal A^{\varepsilon, b}\left(P\times S^1\right)/\mathcal G_0\left(P\times S^1\right)$ induces an isomorphism
$$HM_*\left(\mathcal L^b\mathcal M^g(P),E^H,\mathbb Z_2\right)\cong HM_*\left(\mathcal A^{\varepsilon,b}\left(P\times S^1\right)/\mathcal G_0\left(P\times S^1\right),\mathcal{YM}^{\varepsilon,b},\mathbb Z_2\right).$$  
\end{theorem}
\end{remark}

\begin{remark}The manifold $\mathcal M^g(P)=\mathcal A_0(P)/\mathcal G_0(P)$ can be also interpreted as a symplectic quotient defined with the moment map $\mu:\mathcal A(P)\to \Omega^0(\mathfrak g_P)$, $\mu(A)=*F_A$ and thus we can also investigate the finite dimensions analogue of the correspondence stated in the theorem \ref{thm:mainthm}. For this purpose we choose a finite dimensional symplectic manifold $X$ and a Lie group $G$ acting free on it; we assume in addition that a Hamiltonian action is generated by an equivariant moment map $\mu: X\to \mathfrak g$, where $\mathfrak g$ denotes the Lie algebra of $G$, with regular value $0$ and that the compatible almost complex structure $J$ on $X$ is $G$-invariant. Furthermore, we choose a time dependent and $G$ invariant potential $V_t:X\to \mathbb R$. On the one side, we can study the perturbed geodesics on the symplectic quotient $\mathcal M:=\mu^{-1}(0)/G$, that we assume compact, and hence the critical points of
\begin{equation}
\mathcal E^{\mu,V}(x,\xi):=\frac12\int^1_0\left(|\dot x+L_x\xi |^2-V_t(x)\right)dt
\end{equation}
for $(x,\xi)\in \mathcal L(\mu^{-1}(0))\times \mathcal L(\mathfrak g)$ and where $L_{x(t)}\xi(t)\in T_{x(t)}X$ denotes the fundamental vector field generated by $\xi(t)\in \mathfrak g$ and evaluated at $x(t)$. On the other side, we choose on the loop space of $X\times\mathfrak g$ the twisted energy functional 
\begin{equation}
\mathcal E^{\mu,V,\varepsilon }(x,\xi):=\frac12\int^1_0\left(\frac1{\varepsilon^2 }|\mu(x) |^2+|\dot x+L_x\xi |^2-V_t(x)\right)\,dt
\end{equation}
for $(x,\xi)\in \mathcal L(X)\times\mathcal L( \mathfrak g)$. This last energy functional is the analogue of the perturbed Yang-Mills energy functional $\mathcal {YM}^{\varepsilon,H}$. Also for the finite dimensional case, we can prove a bijection between the critical loops below a given energy level and  for $\varepsilon$ small enough (cf. \cite{remyj}).
\end{remark}

\section{Preliminaries}\label{sec:pre}

In the next sections we briefly explain the setting for our results, if the reader is interested in more details we refer to \cite{remyj2}. In order to introduce the moduli space of flat connections for a non-trivial principal 
 $\textbf{SO}(3)$-bundle over a surface $\Sigma$, we first explain some facts about a principal $G$-bundle $\pi:P\to \Sigma$ where $G$ is a compact Lie group with Lie algebra $\mathfrak g$ and $P$ and $\Sigma$ are smooth manifolds. The action of $G$ on $P$ defines a vertical space
 $$V:=\left\{\left.\left.\left(p,p\xi:=\frac d{dt}\right|_{t=0}p\exp(t\xi)\right)\right| p\in P, \, \xi\in \mathfrak g\right\}
 \subset TP$$
 in the tangent bundle and hence a choice of a {connection}, i.e. an equivariant function $A:TP\to \mathfrak g$
 which satisfies
$$\begin{array}{rll}i) & A(p,\,p\xi)=\xi & \forall p\in P,\,\forall \xi\in\mathfrak g,\\
ii) & A(pg,\,vg)=g^{-1}A(p,v)g & \forall p\in P,\forall v\in T_pP,\end{array}$$
could also be seen as a choice of an equivariant horizontal distribution $H\subset TP$
 which corresponds to the kernel of $A$
and at each point $p\in P$ induces the short exact sequence
$$0\longrightarrow H_{p}=\ker A(p,.)\stackrel{\iota}{\longrightarrow} T_{p}P 
{\longrightarrow}V_{p} \longrightarrow0,$$ where $\iota$ is the inclusion of $H_{p}$ in $T_{p}P$ and $V_p$ the restriction of $V$ at the point $p$. In addition, since $V_{p}=\ker(d\pi(p))$ and $T_{p}P=H_{p}\oplus V_{p}$, $d\pi(p)$ induces an isomorphism between 
$H_{p}$ and $T_{\pi(p)}\Sigma$, hence the horizontal distribution is isomorph to the pullback $\pi^*T\Sigma$ and this observation implies 
that a vectorfield $X$ on $\Sigma$ has a unique horizontal lift $\tilde X\subset H$ on $P$ such that $\tilde X(p)\in H_{p}$ and 
$d_{p}\pi(\tilde X(p))=X(\pi(p))$. 
The set of all the connections of a principal bundle is denoted by $\mathcal A(P)$ and it is an affine space; in fact, for every connection $A_{0}\in \mathcal A(P)$, $\mathcal A(P)=A_{0}+\Omega^1_{\textrm{Ad},H}(P,\mathfrak g)$ where $\Omega^1_{\textrm{Ad},H}(P,\mathfrak g)$ denotes the set of all equivariant functions $\alpha:TP\to\mathfrak g$ such that $V\subset \ker\alpha$, i.e. $\alpha$ is 
horizontal. Similarly, $\Omega^k_{Ad,H}(P,\mathfrak g)$ is 
the space of equivariant and horizontal $k$-forms, i.e for an $\omega\in\Omega^k_{Ad,H}(P,\mathfrak g)$ we have
\begin{equation*}
\begin{split}
\omega(pg;v_{1}g,v_{2}g,...,v_{k}g)=&g^{-1}\omega(p;v_{1},v_{2},...,v_{k})g,\\
\omega(p;v_{1},...,v_{k})=&0, \mbox{ if $v_{i}=p\xi$ for an $i\in\{1,...,k\}$},
\end{split}
\end{equation*}
where $p\in P,\, g\in G, \, \xi\in\mathfrak g, \,v_{i}\in T_{p}P,\,1\leq i\leq k$. 
Therefore, the equivariant and horizontal $k$-forms 
$\Omega^k_{\textrm{Ad},H}(P,\mathfrak g)$ correspond to the $k$-forms over $\Sigma$ with values in the {adjoint bundle}, i.e. 
$\Omega^k_{\textrm{Ad},H}(P,\mathfrak g)\cong\Omega^k(\Sigma,\mathfrak g_{P})$, where $\mathfrak g_P:=P\times_{\mathrm{Ad}}\mathfrak g$ is the associated bundle defined by the equivalence classes $[pg,\xi]\equiv [p,\mathrm{Ad}_g\xi]\equiv [p, g\xi g^{-1}]$.\\

The Lie group $\mathcal{G}(P)$ of equivariant smooth maps $u:P\to G$ is called the {gauge group}
of $P$, i.e. 
$$\mathcal{G}(P):=\{u\in C^\infty(P,G)\mid u(pg)=g^{-1}u(p)g,\,\forall p\in P, \,\forall g\in G \}.$$ 
Since $G$ acts on $P$, every element of the gauge group induces a {gauge transformation}
 of the bundle $P$, i.e. $ \tilde u: P\to P;\, p\mapsto pu(p)$ which is a $G$-{bundle isomorphism}. A gauge transformation $u$ acts on the space of connections by
$$u^*A=u^{-1}Au+u^{-1}du$$
for $A\in \mathcal A(P)$ and hence we can consider $u$ as a change of coordinates. Furthermore, since the Lie algebra of $\mathcal G(P)$ is the space of the equivariant, horizontal 0-forms over $P$, i.e $\Omega^0(\Sigma,\mathfrak g_P)$, in order to compute the infinitesimal gauge transformation on a connection $A$, we choose  an element $\phi$ of the Lie algebra $\Omega^0(\Sigma,\mathfrak g_P)$ and we set $u_t=\exp(t\phi)=1+t\phi+O(t^2)$, then
\begin{equation}\label{gauge1}
\frac{d}{dt}\Big|_{t=0}(u_t^*A)=-\frac{d}{dt}\Big|_{t=0}(u_t^{-1}Au_t+u_t^{-1}du_t)
=-[A,\phi]-d\phi=-d_A\phi.
\end{equation}
In fact, choosing a connection $A\in\mathcal A(P)$, we can define the \textbf{covariant derivative} {Covariant derivative}
$$d_A:\Omega^0(\Sigma, \mathfrak g_P)\to \Omega^1(\Sigma, \mathfrak g_P);\, \phi\mapsto d_A\phi=d\phi+[A,\phi]$$
and the {exterior derivative}
$$d_A:\Omega^k(\Sigma, \mathfrak g_P)\to \Omega^{k+1}(\Sigma, \mathfrak g_P); \omega\mapsto d_A\omega=d\omega+[A\wedge\omega]$$
where $[\omega_{1}\wedge\omega_{2}]:=\omega_{1}\wedge\omega_{2}-(-1)^{lk}\omega_{2}\wedge \omega_{1}$ 
denotes the super Lie bracket operator for $\omega_{1}\in \Omega^l(\Sigma,\mathfrak g_{P})$ and $\omega_{2}\in \Omega^k(\Sigma,\mathfrak g_{P})$. The Hodge operator acts not only on $\Omega^k(\Sigma)$, but on $\Omega^k(\Sigma,\mathfrak g_P)$, too; 
in fact, since\footnote{$\Gamma\big(\wedge^kT^*\Sigma\otimes\mathfrak g_P\big)$ denotes the sections of the bundle $\wedge^kT^*\Sigma\otimes\mathfrak g_P\to \Sigma$.} $\Omega^k(\Sigma,\mathfrak g_P)=\Gamma\big(\wedge^kT^*\Sigma\otimes\mathfrak g_P\big)$,
for all $\omega\in\Omega^k(\Sigma)$, and all $\xi\in\Omega^0(\Sigma,\mathfrak g_P)$, we define $*(\omega\otimes\xi):=*\omega\otimes\xi$. 
Therefore, using two inner products, one on $\Omega^k(\Sigma)$ defined using the Hodge operator and an invariant inner product on the Lie algebra on $\Omega^0(\Sigma,\mathfrak g_P)$, 
we have an inner product on the $k$-forms $\Omega^k(\Sigma,\mathfrak g_P)$
\begin{equation}\label{pro3}
\langle a,b\rangle=\int_\Sigma\langle a\wedge *b\rangle\quad \forall a,b\in\Omega^k(\Sigma,\mathfrak g_P);
\end{equation}
for two vectorfields $X$, $Y$ on $\Sigma$, $\langle a\wedge b\rangle(X,Y)=\langle a(X), b(Y)\rangle-\langle a(Y), b(X)\rangle$. Using this inner product we can define the adjoint operator
$$d_A^*:\Omega^{k+1}(\Sigma,\mathfrak g_P)\to\Omega^{k}(\Sigma,\mathfrak g_P)$$
of the exterior derivative $d_A$, $A\in \mathcal A(P)$.
\\

For any connection $A\in\mathcal A(P)$, the two form $F_A:=dA+\frac{1}{2}[A\wedge A]\in\Omega^2(\Sigma,\mathfrak g_P)$
 is called \textbf{curvature} of $A$ and the gauge group acts by 
$F_{u^*A}=u^{-1}F_{A}u$ for every $u\in\mathcal G(P)$. With this last definition it is possible to introduce the set of flat connections
$$\mathcal A_0(P):=\{A\in \mathcal A(P)\mid F_{A}=0\}$$
and for an $A\in\mathcal A_0(P)$, since $d_A\circ d_A=0$, the cohomology groups
$$H^k_A(\Sigma,\mathfrak g_P ):=\ker d_A/\textrm{im } d_A\Big|_{\Omega^k(\Sigma,\mathfrak g_P)}=\ker d_A\cap \ker d_A^*\Big|_{\Omega^k(\Sigma,\mathfrak g_P)}$$ 
are well defined for any $k\in\mathbb N$. Moreover, we have the orthogonal splitting
\begin{equation}\label{split}
\Omega^k(\Sigma,\mathfrak g_P)=d_A\Omega^{k-1}(\Sigma,\mathfrak g_P)\oplus 
H^k_A(\Sigma,\mathfrak g_P )\oplus d_A^*\Omega^{k+1}(\Sigma,\mathfrak g_P)
\end{equation}
and we denote the canonical projection in to the harmonic forms by $\pi_{A}$, i.e.
$$\pi_{A}: \Omega^k(\Sigma,\mathfrak g_{P})\to H_{A}^k(\Sigma,\mathfrak g_{P}).$$

\section{The moduli space $\mathcal M^g(P)$}\label{chapter:intro}

For the following, we choose a compact oriented Riemann 
surface $\Sigma$ of genus $g\geq 1$ and a non-trivial principal 
 $\textbf{SO}(3)$-bundle $\pi:P\to\Sigma$; next, we define the even gauge group $\mathcal G_{0}(P)$ as the unit component of $\mathcal G(P)$ and for more details we refer to \cite{MR1297130}. Finally we can introduce the moduli space $$\mathcal M^g(P):=\mathcal A_0(P)/\mathcal G_0(P)$$
which is a compact smooth manifold of dimension $6g-6$ and if $g\geq 2$, then it is also connected and simply connected; these results were proved Dostoglou and Salamon (cf. \cite{MR1297130}) using the works of Newstead (cf. \cite{MR0232015}). 

\begin{remark}
If $g=2$, then the moduli space $\mathcal M^2(P)$ can be seen as an intersection of quadrics in $P_5$ (cf. \cite{MR0237500}) .
\end{remark}

\begin{remark}
Since an element $u\in \mathcal G_0(P)$, which is an element of the isotropy 
group\footnote{An $u\in\mathcal G(P)$ is an element of the isotropy group of a connection $A$ if and only if $u^*A=A$.}, maps $P$ to the 
identity, the operator
 $d_{A}:\Omega^0(\Sigma,\mathfrak g_{P})\to \Omega^1(\Sigma,\mathfrak g_{P})$ is injective. Moreover,
 $d_{A}^*d_{A}:\Omega^0(\Sigma,\mathfrak g_{P})\to \Omega^0(\Sigma,\mathfrak g_{P})$ is invertible, because 
 the fact that $d_{A}$ is injective implies that $d_{A}^*$ is surjective by the decomposition of 
 $\Omega^0(\Sigma,\mathfrak g_{p})$, see equation (\ref{split}) and in addition $\textrm{im } d_{A}^*=\textrm{im } d_{A}^*d_{A}$ 
 by the decomposition of 
 $\Omega^1(\Sigma,\mathfrak g_{p})$.
\end{remark}

The infinitesimal gauge transformation for $\Psi\in\Omega^0(\Sigma,\mathfrak g_P)$ acts on a connection by $$\mathcal A(P)\to T\mathcal A(P);\,
A\mapsto -d_A\Psi$$
and thus, the tangent space at 
 $[A]\in \mathcal M^g(P)$, $A\in\mathcal A_{0}(P)$, can be identified with the first homological group 
 $H_{A}^1(\Sigma, \mathfrak g_{P})$, in fact by (\ref{gauge1}) and by the orthogonal splitting $\ker d_A=\textrm{im } d_A\oplus H_A^1(\Sigma,\mathfrak g_P)$, we have
 \begin{equation}\label{TM}
T_{A}\mathcal A_{0}(P)/\textrm{im } d_{A}=\ker d_{A}/\textrm{im } d_{A} =H_{A}^1(\Sigma,\mathfrak g_{P})
\end{equation}
because the tangent space $T_{A}\mathcal A_0(P)$ corresponds to the kernel of $d_{A}$. Hence if we choose a conformal structure on $\Sigma$, then we have a complex structure on $\mathcal M^g(P)$ which 
is not, but the Hodge-*- operator acting on $H_{A}^1(\Sigma,\mathfrak g_{P})$. We refer to \cite{MR1297130} and \cite{MR1698616} for more details.\\

Moreover, since the tangent space of $[A]\in \mathcal M^g(P)$, for every $A\in\mathcal A_0(P)$, can be identified with $H_A^1(\Sigma,\mathfrak g_P)$, we have a symplectic form $\omega_A(a,\,b)=\int_\Sigma\langle a\wedge b\rangle$, for $a,b\in H_A^1(\Sigma,\mathfrak g_P)$, and a complex structure defined by the Hodge-$*$-operator. Since the symplectic 2-form does not depend on the base connection $A$, it is constant and thus, closed. Hence, $\mathcal M^g(P)$ is a K\"ahler manifold; this symplectic approach of the space of connections was introduced by Atiyah and Bott in \cite{MR702806}. We conclude this section with the following result (cf. \cite{remyj2} for the computations).

\begin{lemma}\label{lemma:regauge}
We choose two flat connections $A'$, $A'' \in \mathcal A_0(P)$, then
\begin{equation*}
\min_{u\in \mathcal G(P)}\|A'-u^*A''\|_{L^2(\Sigma)}\leq d([A'],[A''])
\end{equation*}
where $d(\cdot,\cdot)$ denotes the distance between $[A']$ and $[A'']$ on the smooth compact manifold $\mathcal M^g(P)$. 
\end{lemma}

\section{Perturbed geodesics on $\mathcal M^g(P)$}

The idea is to find a loop $A\subset C^\infty(\mathbb R/\mathbb Z,\mathcal A_{0})$ such that the projection 
$\Pi(A)$ on $\mathcal M^g(P)$ is a geodesic, where $\Pi: \mathcal A_{0}(P)\to \mathcal M^g(P);\,B\mapsto [B]$, and since $\partial_{t}A\in T_A\mathcal A_{0}=H_{A}^1(\Sigma,\mathfrak g_{P})\oplus \textrm{im } d_{A}$ and 
$d\Pi(A)\partial_tA\in T_{\Pi(A)}\mathcal M^g(P)$ which corresponds to $H_{A}^1(\Sigma,\mathfrak g_{P})$, 
$$0=d_{A}^*(\partial_tA-d_{A}\Psi)=d_{A}(\partial_{t}A-d_{A}\Psi)$$
for a $\Psi$ such that 
$\Psi(t)\in\Omega^0(\Sigma,\mathfrak{g}_P)$ for all $t\in S^1$. Hence, 
since $d_{A}^*d_{A}$ is invertible, $\Psi$ is uniquely determined and
 $$\pi_A(\partial_tA)=\partial_{t}A-d_{A}(d_{A}^*d_{A})^{-1}d_{A}^*\partial_{t}A
 =\partial_tA-d_{A}\Psi.$$
 The unperturbed energy of our curve is, therefore,
\begin{equation}\label{unperturbed}
E(A)=\frac12\int_0^1|d\Pi(A)\partial_tA|^2dt=\frac12\int_{0}^1|\partial_tA-d_A\Psi|^2dt.
\end{equation}

If we consider a time dependent Hamiltonian map
$$\bar H:\mathbb R/\mathbb Z\times\mathcal A_0(P)\to \mathbb R;\,(t,A)\mapsto \bar H_t(A)$$
which is invariant under $\mathcal G_{0}(P)$ and constructed using the holonomy (see \cite{MR1297130}); then we can perturbe the energy functional subtracting from $E$ the integral of $\bar H_{t}(.)$, i.e.
\begin{equation}\label{perturbation1}
E^{\bar H}(A)=\frac12\int_{0}^1|\partial_tA-d_A\Psi|^2dt-\int_{0}^1 \bar H_{t}(A) dt.
\end{equation}
The equivariance of $\bar H_{t}(\cdot)$ means that we indroduce a perturbation on the energy functional on the loop space 
of the smooth manifold $\mathcal M^g(P)$. Weber (\cite{MR1930985}) using the Thom-Smale transversality 
proved that the set
\begin{align*}
\nu_{reg}:=\{&H \in C^\infty(S^1\times\mathcal M^g(P),\mathbb R)\mid \\
&\mbox{ The Jacobi operator for $E^{\bar H}$ is bijective for all critical loops}\}
\end{align*}
is open and dense in $C^\infty(S^1\times \mathcal M^g(P),\mathbb R)$ endoved with the compact-open topology and $\nu_{reg}$ is residual. 
Therefore we can choose $\bar H_t$ near the zero function as we like such that the Jacobi operator of $E^{\bar H}$ for all the perturbed geodesics is invertible and from now our perturbation is choosen with this property. Furthermore, in the same paper Weber proved that below a given energy level we have only finite perturbed geodesics.\\

Next, we define a perturbation $H_t:\mathcal A(P)\to \mathbb R$, where $H_t(A)=\bar H(A)$ for every $A\in \mathcal A_0(P)$. A first approach is to pick a gauge invariant holonomy perturbation on $\mathcal A(P)$ since every Hamiltonian $H_t$ can be constructed in this way (cf. \cite{MR1297130}); since $H_t$ is constant along $\mathcal G(P)^*A$ for a given connection $A\in \mathcal A(P)$ and $T_{A}(\mathcal G(P)^*A)=\textrm{im } d_{A}$,
\begin{equation}\label{daX0}
 d_{A}X_t(A)=0.
\end{equation}

Another possibility is the following. We pick a smooth map $\rho:[0,\infty)\to [0,1]$ with the property that $\rho(x)=0$ if $x\geq \delta_0^2$ and $\rho(x)=1$ if $x\leq \left(\frac {2\delta_0}3\right)^2$ for a $\delta_0$ which satisfies the conditions of the lemmas \ref{lemma76dt94} and \ref{lemma82dt94} for $p=2$ and $q=4$. Then we define $H_t(A)=0$ for every $A$ with $\|F_A\|_{L^2}\geq \delta_0$ and
$$H_t(A):=\rho\left(\|F_A\|_{L^2}^2\right)\bar H_t\left(A+*d_{A}\eta(A)\right)$$
otherwise, where $\eta(A)$ is the unique $0$-form given by the theorem \ref{lemma82dt94} for the connection $A$. In this case, if $A$ is flat then $H_t(A+*sd_A\eta)$ is constant for every $0$-form $\eta\in \Omega^0(\Sigma,\mathfrak g_P)$ and every $s\in (-\varepsilon,\varepsilon)$ with $\varepsilon$ sufficiently small and we can conclude that $d_{A}*X_t(A)=0$. In both cases, the time-dependent Hamiltonian vector field $X_{t}:\mathcal A(P)\to \Omega^1(\mathfrak g_{P})$ is defined such that, for any 1-form $\alpha$ and any connection $A$, $dH_{t}(A)\alpha=\int_{\Sigma}\langle X_t(A)\wedge \alpha\rangle$. 


\begin{theorem}\label{geodesics}
A closed curve $A$, $A(t)\in\mathcal A_{0}$ for all $t\in S^1\cong\mathbb R/\mathbb Z$, descends to  a  perturbed geodesic if and only if there are $\Psi(t)\in\Omega^0(\Sigma,\mathfrak{g}_P)$ and $\omega(t)\in\Omega^2(\Sigma,\mathfrak g_P)$ such that
\begin{equation}\label{eq:thm:geod:eq1}
-\nabla_t(\partial_t A-d_A \Psi)-*X_t(A)-d_A^* \omega=0,
\end{equation}
\begin{equation}\label{eq:thm:geod:eq2}
d_A^*(\partial_t A-d_A \Psi)=0,
\end{equation}
where $\nabla_t:=\partial_t+[\Psi,.]$. If this holds, $\omega$ is the unique solution of
\begin{equation}\label{eq:thm:geod:dasdsgf}
d_Ad_A^*\omega=[(\partial_t A-d_A \Psi) \wedge(\partial_t A-d_A \Psi )]-d_A*X_t(A).
\end{equation}
\end{theorem}

\begin{proof}See \cite{MR1715156} or \cite{remyj2}.
\end{proof} 

\begin{remark} 
We defined the moduli space of flat connections $\mathcal M^g(P)$ by taking the quotient $\mathcal A_0(P)/\mathcal G_0(P)$ where $\mathcal G_0(P)$ is the even gauge group and thus a geodesic $\gamma(t)\subset\mathcal M^g(P)$ lifts to a closed path in $\mathcal A_0(P)$ which is unique modulo
$$\mathcal G_0\left(P\times S^1\right):=\{g\in \mathcal G(P\times S^1)\mid g(t)\in \mathcal G_0(P)\quad\forall t\in S^1\}.$$ 
The group $\mathcal G_0\left(P\times S^1\right)$ acts clearly also on the connections $\mathcal A(P\times S^1)$ of a principal bundle $P\times S^1\to \Sigma \times S^1$.
\end{remark}

We can therefore characterise the perturbed geodesics using the map
\begin{equation}\label{f0}
\mathcal F^0(A,\Psi):=\left(\begin{matrix}-\nabla_t(\partial_t A-d_A \Psi)-*X_{t}(A)\\
-d_A^*(\partial_t A-d_A \Psi)\wedge dt\end{matrix}\right)
=\left(\begin{matrix}\mathcal F_1^0(A,\Psi)\\
\mathcal F_2^0(A,\Psi)\end{matrix}\right)\end{equation}
defined for two loops $A(t)\in \mathcal A(P)$ and $\Psi(t)\in\Omega^0(\Sigma,\mathfrak g_P)$. In fact, a closed curve $A$, $A(t)\in\mathcal A_{0}(P)$ for all $t\in S^1\cong\mathbb R/\mathbb Z$, 
descends to  a perturbed geodesic if and only if $\mathcal F^0(A,\Psi)\in \textrm{im } d_A^*\times \{0\}$. Next, we denote the set of perturbed geodesics below a energy level $b$ by
\begin{equation*}
\begin{split}
\mathrm{Crit}^b_{E^H}:=\big\{A+\Psi dt\in \mathcal L(\mathcal A_0(P)\otimes &\Omega^0(\Sigma,\mathfrak g_P)\wedge dt)|\,E^H(A)\leq b, (\ref{eq:thm:geod:eq1}),(\ref{eq:thm:geod:eq2})\big\}.
\end{split}
\end{equation*}
The Jacobi operator of a loop $A\subset \mathcal A_{0}$,
 which descends to a perturbed geodesic on $\mathcal M^g(P)$, is given by (cf. \cite{MR1715156} or \cite{remyj2})
\begin{equation}\label{jacoperator1}
\begin{split}
\mathcal D^0(A)(\alpha,\psi)=&\pi_{A}\left(2[\psi,(\partial_tA-d_{A}\Psi)]+d*X_t(A)\alpha+\nabla_t\nabla_t\alpha\right)\\
&+\pi_{A}\left(*\left[\alpha\wedge*d_{A}(d_{A}^*d_{A})^{-1}
\left(\nabla_t(\partial_tA-d_{A}\Psi)+*X_t(A)\right)\right]
\right)
\end{split}
\end{equation}
where $\alpha(t)\in H^1_{A(t)}(\Sigma,\mathfrak g_P)$, $\Psi$ is defined uniquely by 
\begin{equation}
d_A^*(\partial_tA-d_A\Psi)=0
\end{equation}
and 
$\psi(t)\in\Omega^0(\Sigma,\mathfrak g_P)$ by
\begin{equation}\label{jacoperator2}
-2*[\alpha\wedge*(\partial_tA-d_{A}\Psi)]-d_{A}^*d_{A}\psi=0.
\end{equation}

\section{Perturbed Yang-Mills connections}\label{chapter:YM}

Now, we choose a Riemann metric $g_\Sigma$ on the surface $\Sigma$ and we consider the manifold 
$\Sigma\times S^1$ with the partial rescaled metric $(\varepsilon^2g_\Sigma\oplus g_{S^1})$ for a given 
$\varepsilon\in]0,1]$; furthermore, we denote 
by $\pi^\varepsilon:P\times S^1\to \Sigma\times S^1$ the principal $\textrm{\bf SO}(3)$-bundle over $\Sigma\times S^1$
 and we assume that the restriction $P\times \{s\}\to \Sigma\times \{s\}$ is non-trivial. If we choose a connection $\Xi=A+\Psi \,dt\in \mathcal A(P\times S^1)$ where 
 $A(t)\in\mathcal A(P)$, $\Psi(t)\in\Omega^0(\Sigma,\mathfrak g_P)$ for all $t\in S^1$, then the $L^2$-norm induced by the metric $(\varepsilon^2g_\Sigma\oplus g_{S^1})$ of the curvature 
$F_{\Xi}=F_{A}-(\partial_{t}A-d_{A}\Psi)\wedge dt$ is given by

 \begin{equation*}
\|F_{\Xi}\|_{L^2}^2
   =\int_{0}^1 \left(\frac 1 {\varepsilon^2}\|F_{A}\|_{L^2(\Sigma)}^2
  +\|\partial_{t}A-d_{A}\Psi\|_{L^2(\Sigma)}^2\right)\,dt;
 \end{equation*}
 if we add the same perturbation as in (\ref{perturbation1}), we can define the perturbed Yang-Mills functional
 \begin{equation}
 \mathcal{YM}^{\varepsilon,H}(\Xi)
  :=\frac12\int_{0}^1 \left(\frac 1 {\varepsilon^2}\|F_{A}\|_{L^2(\Sigma)}^2
  +\|\partial_{t}A-d_{A}\Psi\|_{L^2(\Sigma)}^2\right)\,dt-\int_0^1 H_{t}(A)\,dt.
\end{equation}

A critical connection $\Xi^\varepsilon=A^\varepsilon+\Psi^\varepsilon dt\in\mathcal A(P\times S^1)$ of $\mathcal {YM}^{\varepsilon,H}$ is called a { perturbed Yang-Mills connection} and  has to satisfy the equation $d_{\Xi^\varepsilon }^{*_\varepsilon} F_{\Xi^\varepsilon}-*X_t(A)=0$ that is equivalent to the two conditions
 \begin{equation}\label{epsiloneq1}
 \frac1{\varepsilon^2}d_{A^\varepsilon}^*F_{A^\varepsilon}-
  \nabla_t(\partial_t A^\varepsilon-d_{A^\varepsilon}\Psi^\varepsilon)
 -*X_{t}(A^\varepsilon)=0,
  \end{equation}
 \begin{equation}\label{epsiloneq2}
 -\frac 1{\varepsilon^2 }d_{A^\varepsilon}^*(\partial_tA^\varepsilon-d_{A}\Psi^\varepsilon)=0.
\end{equation}
In the following, if we write a perturbed Yang-Mills connection as $\Xi^\varepsilon=A^\varepsilon+\Psi^\varepsilon dt$ 
with apex $\varepsilon$, then we mean that $\Xi^\varepsilon$ is a critical point of the functional $\mathcal{YM}^{\varepsilon,H}$ and we denote the set of perturbed Yang-Mills connections below an energy level $b$ by
\begin{equation*}
\begin{split}
\mathrm{Crit}^b_{\mathcal{YM}^{\varepsilon,H}}:=
\big\{\Xi^\varepsilon\in& \mathcal A(P\times S^1)|\,\mathcal {YM}^{\varepsilon,H }\left(\Xi^\varepsilon\right)\leq b, (\ref{epsiloneq1}),(\ref{epsiloneq2})\big\}.
\end{split}
\end{equation*}
If we fix a connection $\Xi^0=A^0+\Psi^0 dt$, then we can define an 
$\varepsilon$-dependent map $\mathcal F^\varepsilon$, for $\varepsilon>0$, by $\mathcal F^\varepsilon(A,\Psi)
=\mathcal F^\varepsilon_1(A,\Psi)+\mathcal F^\varepsilon_2(A,\Psi)$ and
\begin{equation}\label{eq:fe1}
\begin{split}
\mathcal F^\varepsilon_1(A,\Psi)=&\frac1{\varepsilon^2}d_{A}^*F_{A}-
  \nabla_t(\partial_t A-d_{A}\Psi)
 -*X_{t}(A)\\
&+\frac 1{\varepsilon^2}d_Ad_A^*(A-A^0)-d_A\nabla_t(\Psi-\Psi^0),
\end{split}
\end{equation}
\begin{equation}\label{eq:fe2}
\mathcal F^\varepsilon_2(A,\Psi)=\left(-\frac 1{\varepsilon^2}d_A^*(\partial_t A-d_A \Psi)
+\frac1{\varepsilon^2}\nabla_td_A^*(A-A^0)-\nabla^2_t(\Psi-\Psi^0)\right)\wedge dt;
\end{equation}
then the zeros of $\mathcal F^\varepsilon$ are perturbed $\varepsilon$-Yang-Mills connections and they satisfy the local gauge condition $$d_{\Xi^0}^{*_\varepsilon}\left(\Xi-\Xi^0\right)=\frac 1{\varepsilon^2}d_{A^0}^*(A-A^0)-\nabla_t^{\Psi^0}(\Psi-\Psi^0)=0$$
respect to the reference connection $A^0+\Psi^0 dt$ by the following remark already considered by Hong (cf. \cite{MR1715156}). 

\begin{remark} $\Xi^\varepsilon=A^\varepsilon+\Psi^\varepsilon dt$ is a perturbed Yang-Mills
 connection on $P\times S^1$ and satisfies the gauge condition $d_{\Xi^\varepsilon}^{*_\varepsilon}(\alpha^\varepsilon+\psi^\varepsilon dt)=0$ with $\alpha^\varepsilon+\psi^\varepsilon dt:=\Xi^\varepsilon-\Xi^0$ if and 
 only if 
\begin{equation}\label{yymm}
d_{\Xi^\varepsilon}d_{\Xi^\varepsilon}^{*_\varepsilon}(\alpha^\varepsilon+\psi^\varepsilon dt)
+d_{\Xi^\varepsilon}^{*_\varepsilon}F_{\Xi^\varepsilon}-*X_{t}(A^\varepsilon)=0.
\end{equation}
One can see this deriving \ref{yymm} by $d_{\Xi^\varepsilon}^{*_\varepsilon}$.
\end{remark}

\begin{remark}
If we choose $\frac 32< p<\infty$ and $b>0$, then for every perturbed Yang-Mills connection $\Xi^\varepsilon=A^\varepsilon+\Psi^\varepsilon dt\in \mathcal A^{1,p}(P\times S^1)$, there exists a gauge transformation $u\in \mathcal G_0^{2,p}(P\times S^1)$ such that $u^*\Xi^\varepsilon$ is smooth. A proof of this statement for weak Yang-Mills connections can be found in \cite{MR2030823} (cf. theorem 9.4) and the proof holds also for perturbed Yang-Mills connections.
\end{remark}

If we linearise the equations (\ref{epsiloneq1}) and (\ref{epsiloneq2}) we obtain the two components of the Jacobi operator 
$$\mathrm{Jac}^{\varepsilon,H}(\Xi^\varepsilon):
\Omega^1(\Sigma\times S^1,\mathfrak g_P)
\to \Omega^1(\Sigma\times S^1,\mathfrak g_P)
$$ 
of a perturbed Yang-Mills connection:
\begin{equation}\label{jacym3}
\begin{split}
\mathrm{Jac}^{\varepsilon,H}&(A^\varepsilon+\Psi^\varepsilon dt)(\alpha,\psi)
=\frac1{\varepsilon^2}d_{A^\varepsilon}^*d_{A^\varepsilon}\alpha
+\frac1{\varepsilon^2}*[\alpha\wedge*F_{A^\varepsilon}]\\
&-d*X_t(A^\varepsilon)\alpha-\nabla_{t}\nabla_{t}\alpha
+d_{A^\varepsilon}\nabla_{t} \psi-2[\psi,(\partial_tA^\varepsilon-d_{A^\varepsilon}\Psi^\varepsilon)]\\
&+\left(\frac 1{\varepsilon^2}*[\alpha\wedge*(\partial_tA^\varepsilon-d_{A^\varepsilon}\Psi^\varepsilon)]
-\frac 1{\varepsilon^2}\nabla_{t}d_{A^\varepsilon}^*\alpha
+\frac 1{\varepsilon^2}d_{A^\varepsilon}^*d_{A^\varepsilon}\psi\right) dt,
\end{split}
\end{equation}
for any $\alpha(t)\in \Omega^1 (\Sigma,\mathfrak g_P)$ and $\psi(t)\in\Omega^0(\Sigma,\mathfrak g_P)$.
In 1982, Atiyah and Bott (cf. \cite{MR702806}) showed that the Jacobi operator of a Yang-Mills connection $\Xi^\varepsilon=A^\varepsilon+\Psi^\varepsilon dt$ is Fredholm of index 0; for the perturbed case we have the same result. First, we recall that the gauge group acts on the 1-forms adding 
the image of $d_{\Xi^\varepsilon}$ and hence $\alpha+\psi\, dt $ is an element of 
$\Omega^1(\Sigma\times S^1,\mathfrak g_P)/\mathcal G_\Sigma(P\times S^1)$ if and only if 
\begin{equation*}
0=\langle  \alpha+\psi \,dt , d_{\Xi^\varepsilon}
\phi\rangle=\langle  d_{\Xi^\varepsilon}^{*_\varepsilon}( \alpha+\psi\, dt) ,
\phi\rangle
\end{equation*}
for every $\phi\in \Omega^0(\Sigma\times S^1,\mathfrak g_P)$ and consequently, if and only if  $\alpha+\psi\, dt \in \ker d_{\Xi^\varepsilon}^{*_\varepsilon}$. Therefore,  under the condition $d_{\Xi^\varepsilon}^{*_\varepsilon}( \alpha+\psi\, dt)=0$ we have that
\begin{equation}\label{fredholmlemmaeq2}
\mathrm{Jac}^{\varepsilon,H}(\Xi^\varepsilon)(\alpha+\psi\,dt)
=\mathrm{Jac}^{\varepsilon,H}(\Xi^\varepsilon)(\alpha+\psi\,dt)
+d_{\Xi^\varepsilon}d_{\Xi^\varepsilon}^*( \alpha+\psi\, dt)
\end{equation}
%
%
which can be written as
\begin{equation}\label{fredholmlemmaeq23}
\left(d_{\Xi^\varepsilon}^{*_\varepsilon}d_{\Xi^\varepsilon}
+d_{\Xi^\varepsilon}d_{\Xi^\varepsilon}^{*_\varepsilon}\right)(\alpha+\psi\,dt)+*[(\alpha+\psi\,dt)\wedge 
*F_{\Xi^\varepsilon}]- d*X_t(A^\varepsilon)(\alpha+\psi\,dt)
\end{equation}
where the first term is the Laplace operator of $\alpha+\psi\,dt$ and the second one is of order zero and thus, we have a selfadjoint elliptic operator and therefore, a Fredholm operator with index 0. In addition, this can allow us to work with (\ref{fredholmlemmaeq23}) instead of using the Jacobi operator and (\ref{fredholmlemmaeq23}) can be written as the operator $\mathcal D^\varepsilon(A+\Psi dt):=\mathcal D^\varepsilon_1(A+\Psi dt)+\mathcal D^\varepsilon_2(A+\Psi dt)\,dt$ given by
\begin{equation}\label{jacoperator3}
\begin{split}
\mathcal D^\varepsilon_1(A+\Psi dt)(\alpha,\psi):
=&\frac 1{\varepsilon^2}\left(d_{A}^*d_{A}\alpha
+d_{A}d_{A}^*\alpha+*[\alpha\wedge*F_{A}]\right)
-d*X_t(A)\alpha\\
&-\nabla_{t}\nabla_{t}\alpha
-2[\psi,(\partial_tA-d_{A}\Psi)]\\
\mathcal D^\varepsilon_2(A+\Psi dt)(\alpha,\psi):
=&\frac1{\varepsilon^2}\left(2*[\alpha\wedge*(\partial_tA-d_{A}\Psi)]
+d_{A}^*d_{A}\psi\right)-\nabla_{t}\nabla_{t}\psi.
\end{split}
\end{equation}
Moreover, the operator $\mathcal D^\varepsilon$ is almost the linearisation of 
$\mathcal F^\varepsilon$; to be precise $\mathcal D^\varepsilon$ does not contain the derivatives of $d_A$, $d_A^*$ and $\nabla_t$ of the last two terms in both components  
(\ref{eq:fe1}) and (\ref{eq:fe2}), because these can be treated like quadratic terms as we will see in the lemma \ref{lemma:estimate:c}.  If the reference connection $A+\Psi \,dt$ is clear from the context, then we will write the operators without indicating it.

\section{Norms}

If we fix a connection $\Xi_{0}=A_{0}+\Psi_{0} dt\,\in\mathcal A(\Sigma\times S^1)$, then we can define a norm on its tangential space 
and since $\mathcal A(\Sigma\times S^1)$ is 
an affine space, we can use it as a metric on $\mathcal A(\Sigma\times S^1)$.
Let $\xi(t)=\alpha(t)+\psi (t)\wedge dt$ such that $\alpha(t)\in \Omega^1(\Sigma,\mathfrak g_{P})$ and $\psi(t)\in \Omega^0(\Sigma,\mathfrak g_{P})$ or $\alpha(t)\in \Omega^2(\Sigma,\mathfrak g_{P})$ and $\psi(t)\in \Omega^1(\Sigma,\mathfrak g_{P})$. Then we define the following norms
\begin{equation*}
\|\xi\|_{0,p,\varepsilon}^p:=\int_{0}^1\left(\|\alpha\|_{L^p(\Sigma)}^p
+\varepsilon^{p}\|\psi\|_{L^p(\Sigma)}^p\right)dt,
\end{equation*}
\begin{equation*}
\|\xi\|_{\infty,\varepsilon}:=\|\alpha\|_{L^\infty(\Sigma\times S^1)}
+\varepsilon\|\psi\|_{L^\infty(\Sigma\times S^1)}
\end{equation*}
and
\begin{equation*}
\begin{split}
\|\xi\|_{\Xi_0,1,p,\varepsilon}^p:=&\int_{0}^1\left(\|\alpha\|_{L^p(\Sigma)}^p
+\|d_{A_{0}}\alpha\|_{L^p(\Sigma)}^p+
\|d_{A_{0}}^*\alpha\|_{L^p(\Sigma)}^p
+\varepsilon^{p}\|\nabla_{t}\alpha\|_{L^p(\Sigma)}^p\right)dt\\
&+\int_{0}^1\varepsilon^p\left(\|\psi\|_{L^p(\Sigma)}^p+\|d_{A_{0}}\psi\|_{L^p(\Sigma)}^p
+\varepsilon^{p}\|\nabla_{t}\psi\|_{L^p(\Sigma)}^p\right) dt,
\end{split}
\end{equation*}
\begin{equation*}
\|\psi\|_{0,p,\varepsilon}^p:=\int_{0}^1\|\psi\|_{L^p(\Sigma)}^p\,dt.
\end{equation*}
Inductively,
\begin{equation*}
\begin{split}
\|\xi\|_{\Xi_0,k+1,p,\varepsilon}^p:=&\int_{0}^1\left(\|\alpha\|_{\Xi_0,k,p,\varepsilon}^p
+\|d_{A_{0}}\alpha\|_{\Xi_0,k,p,\varepsilon}^p+\|d_{A_{0}}^*\alpha\|_{\Xi_0,k,p,\varepsilon}^p\right)dt\\
+&\int_{0}^1\left(\varepsilon^p\|\nabla_{t}\alpha\|_{\Xi_0,k,p,\varepsilon}^p+\|\psi \,dt\|_{\Xi_0,k,p,\varepsilon}^p\right)dt\\
+&\int_{0}^1\left(\|d_{A_{0}}\psi \wedge dt\|_{\Xi_0,k,p,\varepsilon}^p
+\varepsilon^{p}\|\nabla_{t}\psi\, dt\|_{\Xi_0,k,2,\varepsilon}^p
\right) dt.
\end{split}\end{equation*}
Also in this case, if the reference connection is clear from the context we write the norms without mentioning it.
\begin{remark}
For $i=1,2$, we can define by $W^{k,p}(\Sigma\times S^1,\Lambda^iT^*(\Sigma\times S^1)\otimes\mathfrak g_{P\times S^1 })$ 
the Sobolev space of the sections of $\Lambda^iT^*(\Sigma\times S^1)\otimes\mathfrak g_{P\times S^1}\to \Sigma\times S^1$ 
as the completion of\footnote{Let $E\to M$ a vector bundle, then $\Gamma{E}$ denotes the space of section of the bundle.}
$$\Gamma(\Lambda^iT^*(\Sigma\times S^1)\otimes\mathfrak g_{P})=\Omega^i(\Sigma\times S^1,\mathfrak g_{P\times S^1})$$ 
respect to the norm $\|\cdot\|_{\Xi_0,k,p,1}$. Furthermore, we can define the Sobolev space of the connections on $P\times S^1$ 
as\footnote{For more information, see appendix B of \cite{MR2030823}.} 
$$\mathcal A^{k,p}(P\times S^1)=\Xi_{0}+ W^{k,p},$$
where $W^{k,p}=W^{k,p}(\Sigma\times S^1,T^*(\Sigma\times S^1)\otimes\mathfrak g_{P\times S^1})$, $\Xi_0\in\mathcal A(P\times S^1)$.
\end{remark}
\begin{remark}
The Sobolev space of gauge transformation $\mathcal G^{2,p}_0(P\times S^1)$ is the completion of $\mathcal G_0(P\times S^1)$ with respect the Sobolev $ W^{1,p}$-norm on 1-forms, i.e. $g\in \mathcal G^{2,p}_0(P\times S^1)$ if $g^{-1}d_{\Sigma\times S^1}g\in  W^{1,p}$ and hence $g:\mathcal A^{1,p}(P\times S^1)\to \mathcal A^{1,p}(P\times S^1)$.
\end{remark}
\begin{remark}
The gauge condition $d_{\Xi^\varepsilon}^*(\Xi^\varepsilon-\Xi_0)=0$ assures us that if the perturbed Yang-Mills connection $\Xi^\varepsilon$ is an element of $\mathcal A^{1,2}(P\times S^1)$, then, for any $k\geq2$, there is an $u\in\mathcal G_0^{2,p}(P\times S^1)$ such that $u^*\Xi^\varepsilon\in\mathcal A^{k,2}(P)$ (cf. \cite{MR2030823}, Chapter 9).
\end{remark}
We now choose a reference connection $\Xi_0$ and analogously as for the lemma 4.1 in \cite{MR1283871}, if we define $\bar \xi=\bar\alpha+\bar\psi\,dt$ where 
$\bar\alpha(t)=\alpha(\varepsilon s)$ and $\bar\psi(t)=\varepsilon \psi(\varepsilon s)$, $0\leq t\leq \varepsilon^{-1}$, then $\|\xi\|_{k,p,\varepsilon}=\varepsilon^{\frac1p}\|\bar \xi\|_{W^{k,p}}$. In addition, all the Sobolev inequalities 
hold as follows by the Sobolev embedding theorem (cf. theorem B2 in \cite{MR2030823}).
\begin{theorem}[Sobolev estimates]\label{lemma:sobolev}
We choose $1\leq p,q<\infty$ and $l\leq k$. Then there is a constant $c_s$ such that for every 
$\xi\in  W^{k,p}(\Sigma\times S^1,\Lambda^iT^*(\Sigma\times S^1)\otimes\mathfrak g_{P\times S^1})$, $i=1,2$, and any reference connection $\Xi_0$:
\begin{enumerate}
\item If $l-\frac 3q\leq k-\frac3p$, then
\begin{equation}\label{eq:sobolev1}
\|\xi\|_{\Xi_0,l,q,\varepsilon}
\leq c_s\varepsilon^{1/q-1/p}\,\|\xi\|_{\Xi_0,k,p,\varepsilon}.
\end{equation}
\item If $0< k-\frac3p$, then
\begin{equation}\label{eq:sobolev2}
\|\xi\|_{\Xi_0,\infty,\varepsilon}
\leq c_s\varepsilon^{-1/p}\,\|\xi\|_{\Xi_0,k,p,\varepsilon}.
\end{equation}
\end{enumerate}
\end{theorem}
%

\section{Elliptic estimates}\label{chapter:ellest}


The aim of this chapter is to estimate (theorem \ref{lemma:evaluate2}) the $\|\cdot\|_{2,p,\varepsilon}$-norm of a 1-form $\xi=\alpha+\psi\, dt$ using the $L^p$-norm of the operator $\mathcal D^\varepsilon(\Xi)$ when $\Xi=A+\Psi dt$ represents a perturbed closed geodesic on $\mathcal M^g(P)$. We recall that we assume the Jacobi operator to be invertible for every perturbed geodesic. Hong in \cite{MR1715156} proved a weaker estimate which, in our setting, can be identified with
\begin{equation*}
\begin{split}
\|\alpha+\psi \,dt&-\pi_A(\alpha)\|_{1,2,\varepsilon}+\varepsilon \|\pi_A(\alpha)\|_{1,2,\varepsilon }\\
\leq & c\varepsilon^2 \|\mathcal D^\varepsilon(\Xi+\alpha_0^\varepsilon)(\alpha,\psi)\|_{0,2,\varepsilon}+c\varepsilon \|\pi_A\mathcal D^\varepsilon(\Xi+\alpha_0^\varepsilon)(\alpha,\psi)\|_{L^2}
\end{split}
\end{equation*}
where $\alpha_0^\varepsilon\in \textrm{im } d_A^*$ is the unique solution of
$$d_{A}^*d_A\alpha_0^\varepsilon=\varepsilon^2\nabla_t(\partial_tA-d_A\Psi)+*X_t(A^0);$$
in addition, he estended the last estimate to
\begin{equation*}
\|\alpha+\psi \,dt\|_{k,2,\varepsilon}\leq c \|\mathcal D^\varepsilon(\Xi+\alpha_0^\varepsilon)(\alpha,\psi)\|_{k-1,2,\varepsilon}
\end{equation*}
and with this inequality he proved the existence of a map from the perturbed geodesics $\mathrm{Crit}_{E^H}^b$ to the perturbed Yang-Mills connections $\mathrm{Crit}_{\mathcal{YM}^{\varepsilon,H } }^b$, but he did not show its uniqueness and its surjectivity. With the last two estimates is not possible to obtain the uniqueness statement of the theorem \ref{thm:localuniqueness} even for $p=2$ and, as we have already discussed, the surjectivity could not be established using his rescaling of the metric, in particular because you can not expect that the norms of the curvature $\partial_tA-d_A\Psi$ have a uniforme bound for all the Yang-Mills connections below a given energy level.\\

For this chapter we choose a regular value $b$ of the energy $E^H$, we fix a perturbed closed geodesic $\Xi=A+\Psi\,dt\in \mathrm{Crit}^b_{E^H}$ and 
we define every operator and every norm using this connection. Since the perturbed geodesic $\Xi$ is smooth, there is a positive constant $c_0$ which bounds the $L^\infty$-norm of the velocity and its derivatives, in particular
\begin{equation}\label{crit:boundlin}
\left\|\partial_tA-d_A\Psi\right\|_{L^\infty}+\left\|\nabla_t\left(\partial_tA-d_A\Psi\right)\right\|_{L^\infty}\leq c_0.
\end{equation}

In general, we denote a constant, which is needed to fulfill an estimate, by $c$; it can therefore indicate different constants also in a single computation.

\begin{theorem}\label{lemma:evaluate}
We choose a constant $p\geq2$. If $p=2$ we set $j=0$ otherwise $j=1$. There exist two constants $\varepsilon_0>0$ and $c>0$ 
such that
\begin{equation}\label{eq:lemma:evaluate1}
\|\xi\|_{2,p,\varepsilon}
\leq c\left( \varepsilon\|\mathcal D^\varepsilon(\xi)
\|_{0,p,\varepsilon}+\|\pi_A(\alpha)\|_{L^p}\right),
\end{equation}
\begin{equation}\label{eq:lemma:evaluate2}
\|(1-\pi_A)\xi\|_{2,p,\varepsilon}
\leq\, c \left( \varepsilon^2\|\mathcal D^\varepsilon(\xi)
\|_{0,p,\varepsilon}+\varepsilon\|\pi_A(\alpha)\|_{L^p}+j\varepsilon^2\|\nabla_t^2\pi_A(\alpha)\|_{L^p}\right),
\end{equation}
\begin{equation}\label{eq:lemma:evaluate3}
\begin{split}
\|\alpha-\pi_A(\alpha)\|_{2,p,\varepsilon}
\leq& c \varepsilon^2\left( \|\mathcal D_1^\varepsilon(\xi)
\|_{L^p}+\varepsilon^2\|\mathcal D_2^\varepsilon(\xi)
\|_{L^p}\right)\\
&+c \varepsilon^2\left(\|\pi_A(\alpha)\|_{L^p}+ \|\nabla_t^{j+1}\pi_A(\alpha)\|_{L^p}\right),
\end{split}
\end{equation}
for every $\xi=\alpha+\psi\,dt\in {\mathrm W}^{2,p}$ and 
$0<\varepsilon<\varepsilon_0$.
\end{theorem}
We want also to remark that the estimates for $p=2$ are enough to prove the bijection between the critical connections, but for the identification between the flows between the critical points, which is discussed in \cite{remyj3}, we need the theorem also for $p>2$.
We recall that by the lemma \ref{geodesics} for perturbed geodesic $\Xi=A+\Psi dt$ we can associate a two form $\omega$ defined as the unique solution of
$$d_Ad_A^*\omega=\left[\left(\partial_tA-d_A\Psi\right)\wedge\left(\partial_tA-d_A\Psi\right)\right]$$
which is equivalent to
 $$\omega=d_A(d_A^*d_A)^{-1}(\nabla_t(\partial_tA-d_A\Psi)+*X_t(A)).$$
\begin{theorem}\label{lemma:evaluate2}
We choose $p\geq 2$ and we assume that there is a constant $c_0$ such that
\begin{equation}\label{crit:linesteg1}
|\langle\mathcal D^0\left(\bar\alpha\right), \bar\alpha \rangle|\geq c_0\left( \|\bar\alpha\|_{L^2}+\|\nabla_t\bar\alpha\|_{L^2}\right)^2
\end{equation}
for every $\bar\alpha \in \mathrm W^{2,p}$. 
Then there are two constants $c>0$ and $\varepsilon_0>0$ such that
\begin{equation}\label{eq:lemma:evaluate2222}
\begin{split}
\|\pi_A(\alpha)\|_{L^p}&+\|\nabla_t\pi_A(\alpha)\|_{L^p}+\|\nabla_t^2\pi_A(\alpha)\|_{L^p}\\
\leq&c\big(\varepsilon\|\mathcal D^\varepsilon\left(\alpha,\psi\right) dt\|_{0,p,\varepsilon}+\|\pi_A\left(\mathcal D^\varepsilon_1\left(\alpha,\psi\right)+*[\alpha\wedge*\omega]\right)\|_{L^p}
\big),
\end{split}
\end{equation}
\begin{equation}\label{eq:lemma:evaluate222}
\|\alpha+\psi\, dt\|_{2,p,\varepsilon}
\leq c\big( \varepsilon \,\|\mathcal D^\varepsilon(\alpha,\psi)\|_{0,p,\varepsilon}
+\|\pi_A\left(\mathcal D^\varepsilon_1(\alpha,\psi)
+*[\alpha\wedge *\omega]\right)\|_{L^p}\big),
\end{equation}
\begin{equation}\label{eq:lemma:evaluate22}
\begin{split}
\|\alpha+\psi \,dt&-\pi_A(\alpha)\|_{2,p,\varepsilon}\\
\leq& c \left( \varepsilon^2 \|\mathcal D^\varepsilon(\alpha,\psi)\|_{0,p,\varepsilon}+\varepsilon \|\pi_A\left(\mathcal D^\varepsilon_1(\alpha,\psi)
+*[\alpha\wedge *\omega]\right)\|_{L^p}\right),
\end{split}
\end{equation}
\begin{equation}\label{eq:lemma:evaluate2277}
\begin{split}
\|\alpha-\pi_A(\alpha)\|_{2,p,\varepsilon}
\leq& c \varepsilon^2  \|\mathcal D_1^\varepsilon(\alpha,\psi)\|_{L^p}+c\varepsilon^4\|\mathcal D_2^\varepsilon(\alpha,\psi)\|_{L^p}\\
&+c\varepsilon^2\|\pi_A\left(\mathcal D^\varepsilon_1(\alpha,\psi)
+*[\alpha\wedge *\omega]\right)\|_{L^p}
\end{split}
\end{equation}
for every $\alpha+\psi \,dt \in  {W}^{2,p}$ and 
$0<\varepsilon< \varepsilon_0$.
\end{theorem}

\begin{remark}
 The condition (\ref{crit:linesteg1}) is always satisfied whenever the Jacobi operator $\mathcal D^0$ is invertible because there is a positive constant $c$ sucht that $\left\|\bar\alpha\right\|_{L^2}^2
\leq c \left\langle \mathcal D^0(\bar\alpha),\bar\alpha\right\rangle_{L^2}$ and
\begin{align*}
\left\|\nabla_t\bar\alpha\right\|_{L^2}^2=& \left|\left\langle \pi_A\left( \nabla_t\nabla_t\bar\alpha\right),\bar\alpha\right\rangle_{L^2}\right|
\leq c \left|\left\langle\mathcal D^0(\bar\alpha),\bar \alpha\right\rangle_{L^2}\right|+c\|\bar\alpha\|_{L^2}^2
\end{align*}

where the last estimate follows from the definition of $\mathcal D^0(\bar\alpha)$ and (\ref{crit:boundlin}).
\end{remark}

We first prove the theorem \ref{lemma:evaluate2} using the theorem \ref{lemma:evaluate} which will be discussed later.

\begin{proof}[Proof of theorem \ref{lemma:evaluate2}]
In order to prove the theorem we start with the estimates proved by Hong (cf. \cite{MR1715156}) and discussed in \cite{remyj2}
\begin{equation}\label{eq:lemma:nonono}
\begin{split}
 \|\pi_A(\alpha) \|_{L^2}&+\|\nabla_t\pi_A(\alpha)\|_{L^2}\\
\leq&c \|\pi_A\big(\mathcal D^{\varepsilon_\nu}_1(\Xi^0)(\alpha,\psi)+*[\alpha\wedge *\omega]\big)\|_{L^2}+c \|(1-\pi_A)(\alpha)\|_{L^2}\\
&+c\|\nabla_t(1-\pi_A)(\alpha)\|_{L^2}+c\varepsilon^2 \|\nabla_t\psi\|_{L^2}+\varepsilon^2\|\psi\|_{L^2}\\
&+c\varepsilon_\nu^2\|\mathcal{D}_2^{\varepsilon_\nu}(\Xi^0)(\alpha,\psi)\|_{L^2},
\end{split}
\end{equation}

\begin{equation}\label{eq:pia}
\begin{split}
\|\pi_A(\alpha)\|_{L^2}&+\|\nabla_t\pi_A(\alpha)\|_{L^2}\\
\leq &c \left(\varepsilon\|\mathcal D^\varepsilon\left(\alpha,\psi\right)\|_{0,2,\varepsilon}+\|\pi_A\left(\mathcal D^\varepsilon_1\left(\alpha,\psi\right)+*[\alpha\wedge*\omega]\right)\|_{L^2}\right)
\end{split}
\end{equation}
and in addition for $q\geq2$ we have that
\begin{equation}\label{crit:eq:pia2}
\begin{split}
\left\|\nabla_t\nabla_t\pi_A(\alpha)\right\|_{L^q}
\leq& c\left\|(d_A+d_A^*)\nabla_t\nabla_t\pi_A(\alpha)\right\|_{L^q}+\left\|\pi_A\nabla_t\nabla_t\pi_A(\alpha)\right\|_{L^q}\\
\leq& c\left\|\nabla_t\pi_A(\alpha)\right\|_{L^q}+c\left\|\pi_A(\alpha)\right\|_{L^q}\\
&+\left\|\pi_A\left(\mathcal D_1^\varepsilon(\alpha+\psi dt)-*[\alpha,*\omega]\right)\right\|_{L^q}\\
&+c\|\alpha\|_{L^q}+c\|\psi\|_{L^q}+c\|\nabla_t(1-\pi_A)\alpha\|_{L^q}\\
\leq &c \left(\varepsilon\|\mathcal D^\varepsilon\left(\alpha,\psi\right)\|_{0,2,\varepsilon}+\|\pi_A\left(\mathcal D^\varepsilon_1\left(\alpha,\psi\right)+*[\alpha\wedge*\omega]\right)\|_{L^2}\right)
\end{split}
\end{equation}
where the second inequality follows from the commutation formulas, the definition of $\mathcal D_1^\varepsilon$ and the triangular inequality and the third by the theorem \ref{lemma:evaluate} and (\ref{eq:pia}). Finally, in the case $p=2$, the theorem  \ref{lemma:evaluate2} follows from the theorem \ref{lemma:evaluate} and from the inequalities (\ref{eq:pia}) and (\ref{crit:eq:pia2}) for $q=2$. For $2<p<6$ we use the Sobolev's theorem \ref{lemma:sobolev} for $\varepsilon=1$:
\begin{equation}\label{eq:lemma:evaluate222ss2}
\begin{split}
\|\pi_A(\alpha)\|_{L^p}&+\|\nabla_t\pi_A(\alpha)\|_{L^p}\\
\leq &c\left(\|\pi_A(\alpha)\|_{L^2}+\|\nabla_t\pi_A(\alpha)\|_{L^2}+\|\nabla_t^2\pi_A(\alpha)\|_{L^2}\right)\\
\leq&c\big(\varepsilon\|\mathcal D^\varepsilon\left(\alpha,\psi\right) dt\|_{0,2,\varepsilon}+\|\pi_A\left(\mathcal D^\varepsilon_1\left(\alpha,\psi\right)+*[\alpha\wedge*\omega]\right)\|_{L^2}
\big),\\
\leq&c\big(\varepsilon\|\mathcal D^\varepsilon\left(\alpha,\psi\right) dt\|_{0,p,\varepsilon}+\|\pi_A\left(\mathcal D^\varepsilon_1\left(\alpha,\psi\right)+*[\alpha\wedge*\omega]\right)\|_{L^p}
\big),
\end{split}
\end{equation}
where the third step follows from the H\"older identity. (\ref{crit:eq:pia2}), (\ref{eq:lemma:evaluate222ss2}) and the theorem \ref{lemma:evaluate} yield now to the estimates (\ref{eq:lemma:evaluate222}), (\ref{eq:lemma:evaluate22}) and (\ref{eq:lemma:evaluate2277}). The estimate (\ref{eq:lemma:evaluate2222}) follows then from (\ref{crit:eq:pia2}) with $q=p$,  (\ref{eq:lemma:evaluate222}) and (\ref{eq:lemma:evaluate22}). In order to prove the estimates for $p\geq6$ we proceed in the same way.
By the Sobolev's theorem \ref{lemma:sobolev} for $\varepsilon=1$ and the H\"older inequality: 
\begin{equation}\label{eq:lemma:evaluate222ss22}
\begin{split}
\|\pi_A(\alpha)\|_{L^p}&+\|\nabla_t\pi_A(\alpha)\|_{L^p}\\
\leq &c\left(\|\pi_A(\alpha)\|_{L^3}+\|\nabla_t\pi_A(\alpha)\|_{L^3}+\|\nabla_t^2\pi_A(\alpha)\|_{L^3}\right)\\
\leq&c\big(\varepsilon\|\mathcal D^\varepsilon\left(\alpha,\psi\right) dt\|_{0,3,\varepsilon}+\|\pi_A\left(\mathcal D^\varepsilon_1\left(\alpha,\psi\right)+*[\alpha\wedge*\omega]\right)\|_{L^3}
\big),\\
\leq&c\big(\varepsilon\|\mathcal D^\varepsilon\left(\alpha,\psi\right) dt\|_{0,p,\varepsilon}+\|\pi_A\left(\mathcal D^\varepsilon_1\left(\alpha,\psi\right)+*[\alpha\wedge*\omega]\right)\|_{L^p}
\big).
\end{split}
\end{equation}
The estimates (\ref{eq:lemma:evaluate222}), (\ref{eq:lemma:evaluate22}) and (\ref{eq:lemma:evaluate2277}) are a consequense of (\ref{eq:lemma:evaluate222ss22}) and the theorem (\ref{lemma:evaluate}); (\ref{eq:lemma:evaluate2222}) follows then from (\ref{crit:eq:pia2}) with $q=p$,  (\ref{eq:lemma:evaluate222}) and (\ref{eq:lemma:evaluate22}).
\end{proof}

\begin{lemma}
We have the following two commutation formulas:
\begin{equation}\label{commform}
[d_A,\nabla_t]=-[(\partial_t A-d_A\Psi)\wedge\cdot\,];
\end{equation}
\begin{equation}\label{commform2}
[d_A^*,\nabla_t]=*[(\partial_t A-d_A\Psi)\wedge*\,\cdot\,].
\end{equation}
\end{lemma}

\begin{proof} The lemma follows from the definitions of the operators using the Jacoby identity for the super Lie bracket operator.\end{proof}

In the following pages we prepare the proof of the theorem \ref{lemma:evaluate} and in order to do this we start showing the next result.

\begin{theorem}\label{lemma:evaluate0}
For $1< p<\infty$ there exist two constants $\varepsilon_0>0$ and $c>0$  such that
\begin{equation}\label{crit:linest.sak2}
\|\psi \|_{2,p,\varepsilon}
\leq c\left( \|(d_A^*d_A-\varepsilon^2\nabla_t\nabla_t)\psi\|_{L^p}+\|\psi\|_{1,p,\varepsilon}\right)
\end{equation}
\begin{equation}\label{crit:linest.sak}
\|\alpha \|_{2,p,\varepsilon}
\leq c\left( \|(d_Ad_A^*+d_A^*d_A-\varepsilon^2\nabla_t\nabla_t)\alpha\|_{L^p}+\|\alpha\|_{1,p,\varepsilon}\right)
\end{equation}
for every $1$-form $\alpha\in \mathrm W^{2,p}$ and every $0$-form $\psi \in \mathrm W^{2,p}$, $0<\varepsilon<\varepsilon_0$.
\end{theorem}

\begin{proof}
We prove the theorem in four steps and in the first three we work in local coordinates and hence we consider the following setting. We choose a metric $g=g_{\mathbb R^2}\oplus dt^2$ on $U\times \mathbb R\subset \mathbb R^2\times \mathbb R$ with $U$ open and contained in a compact set, a constant connection $\Xi_c=A_c+\Psi_c dt\in \Omega^1(U\times \mathbb R, \mathfrak g)$ of the trivial bundle $U\times \mathbb R\times \mathrm{\bf SO}(3)\to U\times \mathbb R$ which satisfies $F_{A_c}=0$ and a positive constant $c_0$. Furthermore we pick a connection $\tilde\Xi=\tilde A+\tilde\Psi dt\in \Omega^1(U\times \mathbb R, \mathfrak g)$ which satisfies
\begin{equation}\label{crit:lienest:sssa}
\begin{split}
\|(\tilde A-A)+(\tilde \Psi-\Psi)\,dt\|_{\infty,\varepsilon}
+\|d_A^*(\tilde A-A)\|_{L^\infty}\leq c_0,\\
 \|d_A(\tilde A-A)+d_A(\tilde \Psi-\Psi)\,dt\|_{\infty,\varepsilon}
\leq c_0,\\
\varepsilon \|\nabla_t(\tilde A-A)+\nabla_t(\tilde \Psi-\Psi)\,dt\|_{\infty,\varepsilon}\leq c_0.
\end{split}
\end{equation}

{\bf Step 1.} For $1<p<\infty$ there exists a constant $c$, such that
\begin{equation}\label{eq:lemma:estuM1}
\|\psi\|_{W^{2,p}}\leq c\left(\|d^*d\psi\|_{L^{p}}+\|\psi\|_{W^{1,p}}\right)
\end{equation}
\begin{equation}\label{eq:lemma:estuM2}
\|\alpha\|_{W^{2,p}}\leq c\left(\|(d^*d+dd^*)\alpha\|_{L^{p}}+\|\alpha\|_{W^{1,p}}\right)
\end{equation}
holds for every $0$-form $\psi\in W^{2,p}_c(U\times \mathbb R,\mathfrak g)$ and every $1$-form $\alpha \in W^{2,p}_c(U\times \mathbb R,T^*(U\times \mathbb R)\times\mathfrak g)$ with compact support in $U\times \mathbb R$.

\begin{proof}[Proof of step 1]
The first step follows directly from the Calderon-Zygmund inequality, i.e. 
$$\|u\|_{W^{2,p}}\leq c\left(\|\Delta_{g}u\|_{L^{p}}+\|u\|_{W^{1,p}}\right)$$
for every $u\in W^{2,p}_c(U\times \mathbb R)$ with compact support in $U\times \mathbb R$. We refer to the chapter 2 and 3 of \cite{MR2030823} for the details.
\end{proof}

{\bf Step 2.} For $1<p<\infty$ there exists a constant $c$, such that
\begin{equation}\label{eq:lemma:estuM3}
\|\psi\|_{\Xi_c,2,p,\varepsilon}\leq c\left(\left\|d_{A_c}^*d_{A_c}\psi-\varepsilon^2\nabla_t^{\Psi_c}\nabla_t^{\Psi_c}\psi\right\|_{L^{p}}+\|\psi\|_{\Xi_c,1,p,\varepsilon}\right)
\end{equation}
\begin{equation}\label{eq:lemma:estuM32}
\|\alpha\|_{\Xi_c,2,p,\varepsilon}\leq c\left(\left\|\left(d_{A_c}^*d_{A_c}+d_{A_c}d_{A_c}^*-\varepsilon^2\nabla_t^{\Psi_c}\nabla_t^{\Psi_c}\right)\alpha\right\|_{L^{p}}+\|\alpha\|_{\Xi_c,1,p,\varepsilon}\right)
\end{equation}
holds for every $0$-form $\psi\in W^{2,p}_c(U\times \mathbb R,\mathfrak g)$ and every $1$-form $\alpha \in W^{2,p}_c(U\times \mathbb R,T^*(U\times \mathbb R)\times\mathfrak g)$ with compact support in $U\times \mathbb R$.

\begin{proof}[Proof of step 2]
First, since the norms $\|\cdot\|_{W^{i,p}}$ and $\|\cdot\|_{\Xi_c,i,p,1}$ are equivalent
\begin{equation*}
\begin{split}
\|\psi\|_{\Xi_c,2,p,1}\leq& \|\psi\|_{W^{2,p}}+c\|\Xi_c\|_{C^1}\|\psi\|_{W^{1,p}}\\
\leq &c\left(\|(d^*d)\psi\|_{L^{p}}+c\|\psi\|_{\Xi_c,1,p,1}+c\|\Xi_c\|_{L^\infty}\|\psi\|_{L^p}\right)\\
\leq & c\left(\left\|d_{A_c}^*d_{A_c}\psi-\nabla_t^{\Psi_c}\nabla_t^{\Psi_c}\psi\right\|_{L^{p}}
+(1+\|\Xi_c\|_{C^1} )\|\psi\|_{\Xi_c,1,p,1}\right)\\
\leq & c\left(\left\|d_{A_c}^*d_{A_c}\psi-\nabla_t^{\Psi_c}\nabla_t^{\Psi_c}\psi\right\|_{L^{p}}
+\|\psi\|_{\Xi_c,1,p,1}\right)
\end{split}
\end{equation*}
and analogously
\begin{equation*}
\|\alpha\|_{\Xi_c,2,p,1}\leq  c\left(\left\|(d_{A_c}^*d_{A_c}+d_{A_c}d_{A_c}^*-\nabla_t^{\Psi_c}\nabla_t^{\Psi_c})\alpha\right\|_{L^{p}}+\|\alpha\|_{\Xi_c,1,p,1}\right).
\end{equation*}
Next, we define a $0$-form $\bar\psi:=\psi(x,\varepsilon t)$, a $1$-form $\bar\alpha:=\alpha(x,\varepsilon t)$ and the connection $\bar A(x,t)+\bar\Psi(x,t) dt= A(x,\varepsilon t)+\varepsilon \Psi(x,\varepsilon t) dt$, then
\begin{equation*}
\begin{split}
\|\psi\|_{\Xi_c,2,p,\varepsilon}=& \varepsilon^{\frac1p} \|\bar\psi\|_{\bar A+\bar \Psi dt,2,p,1}\\
\leq & c\varepsilon^{\frac1p}\left(\left\|d_{\bar A}^*d_{\bar A}\bar \psi-\nabla_t^{\bar \Psi}\nabla_t^{\bar \Psi}\bar\psi\right\|_{L^{p}}
+\|\psi\|_{\bar A+\bar \Psi dt,1,p,1}\right)\\
= & c\left(\left\|d_{A_c}^*d_{A_c}\psi-\varepsilon^2\nabla_t^{\Psi_c}\nabla_t^{\Psi_c}\psi\right\|_{L^{p}}
+\|\psi\|_{\Xi_c, 1,p,\varepsilon}\right)
\end{split}
\end{equation*}
and analogously,
\begin{equation*}
\begin{split}
\|\alpha\|_{\Xi_c,2,p,\varepsilon}=& \varepsilon^{\frac1p} \|\bar\alpha\|_{\bar A+\bar \Psi dt,2,p,1}\\
\leq & c\varepsilon^{\frac1p}\left(\left\|(d_{\bar A}^*d_{\bar A}+d_{\bar A}d_{\bar A}^*-\nabla_t^{\bar \Psi}\nabla_t^{\bar \Psi})\bar\alpha\right\|_{L^{p}}
+\|\alpha\|_{\bar A+\bar \Psi dt,1,p,1}\right)\\
= & c\left(\left\|(d_{A_c}^*d_{A_c}+d_{A_c}d_{A_c}^*-\varepsilon^2\nabla_t^{\Psi_c}\nabla_t^{\Psi_c})\alpha\right\|_{L^{p}}
+\|\alpha\|_{\Xi_c, 1,p,\varepsilon}\right).
\end{split}
\end{equation*}
\end{proof}

{\bf Step 3.} For $1<p<\infty$ there exists a constant $c$, such that
\begin{equation}\label{eq:lemma:estuM3s3}
\|\psi\|_{\tilde \Xi, 2,p,\varepsilon}\leq c\left(\left\|d_{\tilde A}^*d_{\tilde A}\psi-\varepsilon^2\nabla_t^{\tilde\Psi}\nabla_t^{\tilde\Psi}\psi\right\|_{L^{p}}+\|\psi\|_{\tilde \Xi, 1,p,\varepsilon}\right)
\end{equation}
\begin{equation}\label{eq:lemma:estuM3s32}
\|\alpha\|_{\tilde \Xi, 2,p,\varepsilon}\leq c\left(\left\|\left(d_{\tilde A}^*d_{\tilde A}+d_{\tilde A}d_{\tilde A}^*-\varepsilon^2\nabla_t^{\tilde\Psi}\nabla_t^{\tilde\Psi}\right)\alpha\right\|_{L^{p}}+\|\alpha\|_{\tilde \Xi, 1,p,\varepsilon}\right)
\end{equation}
holds for every $0$-form $\psi\in W^{2,p}_c(U\times \mathbb R,\mathfrak g)$ and every $1$-form $\alpha \in W^{2,p}_c(U\times \mathbb R,T^*(U\times \mathbb R)\times\mathfrak g)$ with compact support in $U\times \mathbb R$.

\begin{proof}[Proof of step 3] The third step follows from the second step and the assumption (\ref{crit:lienest:sssa}).
\end{proof}

{\bf Step 4.} We prove the theorem.

\begin{proof}[Proof of step 4]
We choose a finite atlas $\{V_i,\varphi_i:V_i\to \Sigma\times S^1\}_{i\in I}$ of our 3-manifold $\Sigma\times S^1$. 
Furthermore, we fix a partition of the unity 
$\{\rho_i\}_{i\in I}\subset C^\infty(\Sigma\times S^1,[0,1])$, $\sum_{i\in I}\rho_i(x)=1$ 
for every $x\in \Sigma\times S^1$ and $\mathrm{supp}(\rho_i)\subset \varphi_i(V_i)$ for any $i\in I$. Furthermore, we denote by $\Xi_i=A_i+\Psi_i dt\in \Omega(V_i, \mathfrak g)$ the local representations of the connection $A+\Psi dt$ on $V_i$ and by $\alpha_i$ the local representations of $\alpha$. We choose the atlas in order that each $\Xi_i$ satisfies the condition (\ref{crit:lienest:sssa}) for constant connections $\Xi_i^c$. Then by the last step
\begin{equation*}
\begin{split}
\|(\rho_i\circ\varphi_i) \alpha_i\|_{\Xi_i,2,p,\varepsilon}
\leq & c(\Xi_i)
\left\|\left(d_{A_i}d^{*}_{A_i}
+d^{}_{A_i}d_{A_i}-\nabla_t^{\Psi_i}\nabla_t^{\Psi_i}\right)((\rho_i\circ\varphi_i) \alpha_i)\right\|_{L^p(U_i)}\\
&+ c(\Xi_i)\|(\rho_i\circ\varphi_i) \alpha_i\|_{\Xi_i,1,p,\varepsilon},
\end{split}
\end{equation*}
If we sum up all the estimates we obtain
\begin{equation*}
\begin{split}
\|\alpha&\|_{A+\Psi dt,2,p,\varepsilon}
\leq c\|\alpha\|_{1,p,\varepsilon}+\sum_{i\in I}\|(\rho_i\circ\varphi_i) \alpha_i\|_{\Xi_i, 2,p,\varepsilon}\\
\leq& \sum_{i\in I}  c(\Xi_i)
\left\|\left(d_{A_i}d^{*}_{A_i}
+d^{*}_{A_i}d_{A_i}-\varepsilon^2\nabla_t^{\Psi_i}\nabla_t^{\Psi_i}\right)((\rho_i\circ\varphi_i)\alpha_i)\right\|_{L^p(U_i)}\\
&+\sum_{i\in I}  c(\Xi_i)\|(\rho_i\circ\varphi_i)\alpha_i\|_{\Xi_i,1,p,\varepsilon} +c\|\alpha\|_{1,p,\varepsilon}\\
\leq & c(\Xi)\Big(
\left\|(d_{A}d^{*}_{A}
+d^{*}_{A}d_{A}-\varepsilon^2\nabla_t\nabla_t)\alpha\right\|_{L^p}+ \|\alpha\|_{A+\Psi dt,1,p,\varepsilon}\Big).
\end{split}
\end{equation*}
In the same way we can prove (\ref{crit:linest.sak2}).
\end{proof}
\end{proof}

The next lemma allows us to estimate the non-harmonic part of a 1-form using its harmonic term and the elliptic operator $d_Ad_A^*+d_A^*d_A-\varepsilon^2\nabla_t^2$.

\begin{lemma} \label{flow:lemma:lplppia11}
There are two positive constants $c$ and $\varepsilon_0$ such that the following holds. For any $i$-form $\xi\in W^{2,p}$, $i=0,1$ and $0<\varepsilon<\varepsilon_0$
\begin{equation}
\int_{S^1} \|\xi\|^p_{L^2(\Sigma)}  \,dt\leq c \int_{S^1}\|-\varepsilon^2\nabla_t^2\xi+\Delta_A\xi \|^p_{L^2(\Sigma)}dt+c\int_{S^1}\|\pi_A(\xi)\|^p_{L^2(\Sigma)}dt.
\end{equation}
where $\Delta_A=d_Ad_A^*+d_A^*d_A$.
\end{lemma}

\begin{proof}
In this proof we denote the norm $\|\cdot\|_{L^2(\Sigma) }$ by $\|\cdot\|$. If we consider only the Laplace part of the operator, we obtain that
\begin{equation*}
\begin{split}
\int_{S^1}\|\xi\|^{p-2}\langle\xi,-\varepsilon^2\partial^2_t\xi+\Delta_A\xi\rangle dt=&\int_{S^1}\|\xi\|^{p-2}\left(\varepsilon^2\|\partial_t\xi\|^2+\|d_A\xi\|^2+\|d_A^*\xi\|^2\right)dt\\
&+\int_{S^1}(p-2)\|\xi\|^{p-4}\langle\xi,\partial_t\xi\rangle^2dt
\end{split}
\end{equation*}
and thus
\begin{equation}\label{flow:dghhga}
\begin{split}
\int_{S^1}\|\xi\|^{p-2}&\left(\varepsilon^2\|\partial_t\xi\|^2+\|d_A\xi\|^2+\|d_A^*\xi\|^2\right)dt\\
\leq &\int_{S^1}\|\xi\|^{p-2}\langle\xi,-\varepsilon^2\partial^2_t\xi+\Delta_A\xi\rangle dt\\
\leq&\int_{S^1}\|\xi\|^{p-1}\|\varepsilon^2\partial_s\xi-\varepsilon^2\partial^2_t\xi+\Delta_A\xi\|\,dt\\
\leq &\left(\int_{S^1}\|\xi\|^p dt\right)^{\frac{p-1}p}\left(\int_{S^1}\|\varepsilon^2\partial_s\xi-\varepsilon^2\partial^2_t\xi+\Delta_A\xi\|^p dt\right)^{\frac{1}p}
\end{split}
\end{equation}
where the second step follows from the Cauchy-Schwarz inequality and the fourth from the H\"older inequality. Therefore, by lemma \ref{flow:lemma:lpl2}
\begin{align*}
\int_{S^1}\|\xi\|^p dt\leq &\int_{S^1}\|\xi\|^{p-2}\left(\|d_A\xi\|^2+\|d_A^*\xi\|^2+\|\pi_A(\xi)\|^2\right)dt\\
\intertext{and by (\ref{flow:dghhga}) we have that}
\leq& \left(\int_{S^1}\|\xi\|^p dt\right)^{\frac{p-1}p}\left(\int_{S^1}\|-\varepsilon^2\partial^2_t\xi+\Delta_A\xi\|^p dt\right)^{\frac{1}p}\\
&+\int_{S^1}\|\xi\|^{p-1}\|\pi_A(\xi)\|\,dt\\
\intertext{in addition by the H\"older inequality}
\leq& \left(\int_{S^1}\|\xi\|^p dt\right)^{\frac{p-1}p}\left(\int_{S^1}\|-\varepsilon^2\partial^2_t\xi+\Delta_A\xi\|^p dt\right)^{\frac{1}p}\\
&+\left(\int_{S^1}\|\xi\|^p dt\right)^{\frac{p-1}p}\left(\int_{S^1}\|\pi_A(\xi)\|^p dt\right)^{\frac{1}p};
\end{align*}
thus, we can conclude that
\begin{equation*}
\int_{S^1}\|\xi\|^p dt\leq c\int_{S^1}\left( \|-\varepsilon^2\partial^2_t\xi+\Delta_A\xi\|^p+\|\pi_A(\xi)\|^p\right)dt.
\end{equation*}
and hence we finished the proof of the lemma using that $\|\Psi\|_{L^\infty}+\|\partial_t\Psi\|_{L^\infty}$ is bounded by a constant.
\end{proof}

\begin{proof}[Proof of theorem \ref{lemma:evaluate}]
By lemma \ref{flow:lemma:lpl2}, for any $\delta>0$ there is a $c_0$ such that
\begin{equation*}
\begin{split}
\|\alpha\|_{L^p}^p\leq &\delta\left(\|d_A\alpha\|_{L^p}^p+\|d_A*\alpha\|_{L^p}^p\right)+ c_0\int_{S^1}\|\alpha\|_{L^2}^p dt \\
\leq &\delta\left(\|d_A\alpha\|_{L^p}^p+\|d_A*\alpha\|_{L^p}^p\right)+ c_0c_1\int_{S^1 }\|\pi_A(\alpha)\|_{L^2}^p dt\\
&+ c_0c_1\int_{S^1}\|-\varepsilon^2\nabla_t^2\alpha+\Delta_A\alpha\|^p_{L^2}dt\\
\leq &\delta\left(\|d_A\alpha\|_{L^p}^p+\|d_A*\alpha\|_{L^p}^p\right)+ c_0c_1c_2\|\pi_A(\alpha)\|_{L^p}^p\\
&+ c_0c_1c_2\|-\varepsilon^2\nabla_t^2\alpha+\Delta_A\alpha \|^p_{L^p}\\
\leq &\delta\left(\|d_A\alpha\|_{L^p}^p+\|d_A*\alpha\|_{L^p}^p\right)+ c_0c_1c_2\|\pi_A(\alpha)\|_{L^p}^p+c_4\varepsilon^{2p}\|\alpha\|_{L^p}^p\\
&+ c_0c_1c_2\varepsilon^{2p}\|\mathcal D_1^\varepsilon(\xi) \|^p_{L^p}+c_4\varepsilon^{2p}\|\psi\|_{L^p}^p
\end{split}
\end{equation*}
where the second step follows form the lemma \ref{flow:lemma:lplppia11} and the third by the H\"older's inequality with $c_2:=\left(\int_{\Sigma}\mathrm{dvol}_{\Sigma}\right)^{\frac{p-2}p}$. If we choose therefore $\delta$ and $\varepsilon$ small enough we can improve the estimate of the theorem \ref{lemma:evaluate0} using the last estimate and we obtain (\ref{eq:lemma:evaluate1}), i.e.
\begin{equation*}
\|\xi\|_{2,p,\varepsilon}\leq c\left(\varepsilon^2\|\mathcal D^\varepsilon(\xi)\|_{0,p,\varepsilon}+\|\pi_A(\alpha)\|_{L^p}\right);
\end{equation*}
furthermore (\ref{eq:lemma:evaluate2}) can be proved by
\begin{equation*}
\begin{split}
\|(1-\pi_A)\xi\|_{2,p,\varepsilon}\leq& c\varepsilon^2\|\mathcal D^\varepsilon((1-\pi_A)\xi)\|_{0,p,\varepsilon}\\
\leq& c\varepsilon^2\left(\|\mathcal D^\varepsilon(\xi)\|_{0,p,\varepsilon}+\left\|-\nabla_t\nabla_t\pi_A(\alpha)-d*X_t(A)\pi_A(\alpha)\right\|_{L^p}\right)\\
&+c\varepsilon^3\left\|\frac 2{\varepsilon^2}*\left[\pi_A(\alpha)\wedge *\left(\partial_tA-d_A\Psi\right)\right]\,dt\right\|_{L^p}\\
\leq& c\left(\varepsilon^2\|\mathcal D^\varepsilon(\xi)\|_{0,p,\varepsilon}+\varepsilon^2\left\|\nabla_t\nabla_t\pi_A(\alpha)\right\|_{L^p}+\varepsilon\left\|\pi_A(\alpha)\right\|_{L^p}\right).
\end{split}
\end{equation*}
(\ref{eq:lemma:evaluate3}) follows from
\begin{equation*}
\begin{split}
\|(1-\pi_A)\alpha\|_{2,p,\varepsilon}\leq& c\varepsilon^2\|\mathcal D^\varepsilon((1-\pi_A)\alpha)\|_{0,p,\varepsilon}\\
\leq& c\varepsilon^2\|\mathcal D_1^\varepsilon((1-\pi_A)\alpha)\|_{L^p}\\
&+c\varepsilon\left\|2*\left[(1-\pi_A)\alpha \wedge *\left(\partial_tA-d_A\Psi\right)\right]\right\|_{L^p}\\
\leq& c\varepsilon^2\|\mathcal D^\varepsilon_1(\xi)\|_{0,p,\varepsilon}+c\varepsilon\left\|(1-\pi_A)\alpha \right\|_{L^p}\\
&+\varepsilon^2\left\|-\nabla_t\nabla_t\pi_A(\alpha)-d*X_t(A)\pi_A(\alpha)\right\|_{L^p}\\
&+\varepsilon^2\left\|2\left[\psi,\left(\partial_tA-d_A\Psi\right)\right]\right\|_{L^p}\\
\leq& c\varepsilon^2\left(\|\mathcal D_1^\varepsilon(\xi)\|_{0,p,\varepsilon}+\left\|\nabla_t\nabla_t\pi_A(\alpha)\right\|_{L^p}+\left\|\pi_A(\alpha)\right\|_{L^p}\right)\\
&+c\varepsilon\left\|(1-\pi_A)\alpha \right\|_{L^p}+c\varepsilon^2\|\psi\|_{L^p},
\end{split}
\end{equation*}
indeed, if we choose $\varepsilon$ small enough and we use (\ref{eq:lemma:evaluate2}) to estimate $c\varepsilon^2\|\psi\|_{L^p}$ we conclude
\begin{equation*}
\begin{split}
\|(1-\pi_A)\alpha\|_{2,p,\varepsilon}\leq&c\varepsilon^2\left(\|\mathcal D_1^\varepsilon(\xi)\|_{L^p}+\varepsilon^2\|\mathcal D_2^\varepsilon(\xi)\|_{L^p}\right)\\
&+c\varepsilon^2\left(\left\|\nabla_t\nabla_t\pi_A(\alpha)\right\|_{L^p}+\left\|\pi_A(\alpha)\right\|_{L^p}\right).
\end{split}
\end{equation*}

\end{proof}

\section{Quadratic estimates}\label{chapter:qe}

In the next chapter we will prove the existence and the uniqueness of a map $\mathcal T^{\varepsilon,b}$ between the perturbed geodesics and the perturbed Yang-Mills connection provided that $\varepsilon$ is small enough; in order to do this we need the following quadratic estimates.

\begin{lemma}\label{lemma:diffoperator}
For any two constants $p\geq 2$ and $c_0>0$ there are two positive constants $c$ and $\varepsilon_0$ such that for any two connections $A+\Psi dt, \tilde A+\tilde \Psi dt \in \mathcal A^{1,p}(P\times S^1)$
\begin{equation}\label{eq:easy}
\begin{split}
\big\|\big(\mathcal D^\varepsilon(A+\Psi dt)
&-\mathcal D^\varepsilon(\tilde A+\tilde \Psi dt)\big)(\alpha,\psi)\big\|_{0,p,\varepsilon}\\
\leq&\frac c{\varepsilon^2}\|A-\tilde A+(\Psi-\tilde \Psi)\, dt\|_{\infty,\varepsilon}
\|\alpha+\psi dt\|_{1,p,\varepsilon}\\
&+\frac c{\varepsilon^2}\|\alpha+\psi dt\|_{\infty,\varepsilon}
\|A-\tilde A+(\Psi-\tilde \Psi)\, dt\|_{1,p,\varepsilon}
\end{split}
\end{equation}
\begin{equation}\label{eq:easyww}
\begin{split}
\big\|\big(\mathcal D^\varepsilon(A+\Psi dt)
&-\mathcal D^\varepsilon(\tilde A+\tilde \Psi dt)\big)(\alpha,\psi)\big\|_{0,p,\varepsilon}\\
\leq&\frac c{\varepsilon^2}\|\tilde\alpha+\tilde\psi dt\|_{\infty,\varepsilon}
\|\alpha+\psi dt\|_{1,p,\varepsilon}\\
&+\frac c{\varepsilon^2}\left(\|d_{A}\tilde \alpha\|_{L^\infty}+\|d_{A}^*\tilde \alpha\|_{L^\infty}+\varepsilon\|\nabla_t\tilde \alpha\|_{L^\infty}\right) \|\alpha+\psi dt\|_{0,p,\varepsilon}\\
&+\frac c{\varepsilon^2}\left(\varepsilon\|d_{A}\tilde \psi\|_{L^\infty}+\varepsilon^2\|\nabla_t\tilde \psi\|_{L^\infty}\right) \|\alpha+\psi dt\|_{0,p,\varepsilon}
\end{split}
\end{equation}
holds for every $\alpha+\psi\,dt \in W^{1,p}$ and  
$\tilde A+\tilde \Psi\, dt=A+\Psi\,dt+\tilde\alpha+\tilde\psi\, dt$ with $\|\tilde\alpha+\tilde\psi\, dt\|_{\infty,\varepsilon}\leq c_0$ and any $0<\varepsilon<\varepsilon_0$.
\end{lemma}
\begin{proof}
On the one side, the difference between the two first 
components can be written as
\begin{equation}\label{eq:easy1}
\begin{split}
\Big(\mathcal D^\varepsilon_1(A+\Psi dt)
&-\mathcal D^\varepsilon_1(\tilde A+\tilde \Psi dt)\Big)(\alpha,\psi)\\
=&
-\frac1{\varepsilon^2}*
\left[\alpha\wedge*\left(d_{\tilde A}(A-\tilde A)
+\frac12[(A-\tilde A)\wedge(A-\tilde A)]\right)\right]\\
&-\frac1{\varepsilon^2}*[(A-\tilde A)\wedge *[(A-\tilde A)\wedge\alpha]]\\
&+\frac1{\varepsilon^2}d_{\tilde A}^*[(A-\tilde A)\wedge\alpha]
-\frac 1{\varepsilon^2}*[(A-\tilde A)\wedge*d_{\tilde A}\alpha]\\
&-2\left[\psi,\left(\nabla_t(A-\tilde A)
-d_{\tilde A}(\Psi-\tilde \Psi)+[(\Psi-\tilde \Psi),(A-\tilde A)]\right)\right]\\
&-\left[(\Psi-\tilde \Psi),\left(\nabla_t\alpha+[(\Psi-\tilde \Psi),\alpha]\right)\right]
-\nabla_t[(\Psi-\tilde \Psi),\alpha]\\
&+\frac1{\varepsilon^2}\left[(A-\tilde A)\wedge\left(d_{\tilde A}^*\alpha
-*[(A-\tilde A)\wedge*\alpha]\right)\right]\\
&-\frac1{\varepsilon^2}d_{\tilde A}*[(A-\tilde A)\wedge*\alpha]+d*X_t(\tilde A)\alpha-d*X_t(A)\alpha
\end{split}
\end{equation}
and on the other side,
\begin{equation}\label{eq:easy2}
\begin{split}
\Big(\mathcal D^\varepsilon_2(A+\Psi dt)
&-\mathcal D^\varepsilon_2(\tilde A+\tilde \Psi dt)\Big)(\alpha,\psi)\\
=&
\frac2{\varepsilon^2}*
\left[\alpha\wedge*\left(\nabla_t(A-\tilde A)-d_{\tilde A}(\Psi-\tilde \Psi)
-[(A-\tilde A),(\Psi-\tilde \Psi)]\right)\right]\\
&-\frac1{\varepsilon^2}*\left[(A-\tilde A)\wedge *\left([(A-\tilde A),\psi]
+d_{\tilde A}\psi\right)\right]\\
&+\frac1{\varepsilon^2}d_{\tilde A}^*[(A-\tilde A)\wedge\psi]
-\left[(\Psi-\tilde \Psi),\left([(\Psi-\tilde \Psi),\psi]
+\nabla_t\psi\right)\right]\\
&-\nabla_t[(\Psi-\tilde \Psi),\psi].
\end{split}
\end{equation}
The lemma follows estimating term by term the last two identities.
\end{proof}
Next, we consider the expansions, for a connection $A+\Psi dt\in \mathcal A^{2,p}(P\times S^1)$ and a $1$-form $\alpha+\psi dt \in W^{2,p}$,
\begin{equation*}
\mathcal F^\varepsilon_1(A+\alpha,\Psi+\psi)=\mathcal F^\varepsilon_1(A,\Psi)
+\mathcal D_1^\varepsilon(A,\Psi)(\alpha,\psi)+C_1(A,\Psi)(\alpha,\psi)
\end{equation*}
\begin{equation*}
\mathcal F^\varepsilon_2(A+\alpha,\Psi+\psi)=\mathcal F^\varepsilon_2(A,\Psi)
+\mathcal D_2^\varepsilon(A,\Psi)(\alpha,\psi) dt+C_2(A,\Psi)(\alpha,\psi) dt
\end{equation*}
and we prove the following estimates for the non linear terms $C_1(A,\Psi)(\alpha,\psi)$ and $C_2(A,\Psi)(\alpha,\psi)$.

\begin{lemma}\label{lemma:estimate:c}
For any constants $c_0>0$, $p\geq2$ and any reference connection $A_0+\Psi_0  dt\in \mathcal A^{2,p}(P\times S^1)$, there are two positive constants $c$ and $\varepsilon_0$ such that for $A+\Psi dt\in \mathcal A^{2,p}(P\times S^1)$
\begin{equation}\label{eq:c1term}
\begin{split}
\|C_1(A,\Psi)(\alpha,\psi)&+C_2(A,\Psi)(\alpha,\psi) dt\|_{0,p,\varepsilon}\\
\leq& \frac 1{\varepsilon^2}c\,\|\alpha+\psi dt\|_{\infty,\varepsilon}
\|\alpha+\psi dt\|_{1,p,\varepsilon}\\
&+\frac 1{\varepsilon^2}c\,\|\alpha+\psi dt\|_{\infty,\varepsilon}
\|A-A_0+(\Psi-\Psi_0) dt\|_{1,p,\varepsilon},
\end{split}
\end{equation}
\begin{equation}\label{eq:c1term22}
 \begin{split}
  \big\|\pi_{A_0}&\left(C_1(A,\Psi)(\alpha,\psi)\right)\big\|_{L^p}\leq\frac c{\varepsilon^2} \|\alpha+\psi dt\|_{\infty,\varepsilon}\|(1-\pi_{A_0})\alpha+\psi dt\|_{1,p,\varepsilon}\\ &+c\|\alpha\|_{L^\infty}\|\alpha\|_{L^p}+\|\psi\|_{L^\infty}\|\nabla_t\pi_{A_0}(\alpha)\|_{L^p}\\
&+\frac c{\varepsilon^2}\|\alpha\|_{L^\infty}^2\left(\|\alpha\|_{L^p}+\|A-A_0\|_{L^p}\right)\\
&+\frac c{\varepsilon^2}\|\alpha\|_{L^\infty}\left( \left\|d_A^*(A-A_0)\right\|_{L^p}+\|(\Psi-\Psi_0) dt\|_{1,p,\varepsilon}\right)\\
&+\frac c{\varepsilon^2}\|A-A_0\|^2_{L^\infty}\|\alpha\|_{L^p}+c\|\psi\|^2_{L^\infty}\|A-A_0\|_{L^\infty}\|\Psi-\Psi_0\|_{L^p}
 \end{split}
\end{equation}
for every $\alpha+\psi\,dt\in {\mathrm W}^{1,p}$ with norm 
$\|\alpha+\psi\,dt\|_{\infty,\varepsilon}<c_0$ and every 
$0<\varepsilon<\varepsilon_0$.
\end{lemma}

\begin{proof}
By definition, $C_1$ and $C_2$ are
\begin{equation}\label{eq:c1term2}
\begin{split}
C_1(A,\Psi)(\alpha,\psi)=&X_t(A+\alpha)-*X_t(A)-d*X(A)\alpha\\
&+\frac1{2\varepsilon^2}d_A^*[\alpha\wedge\alpha]
+\frac 1{\varepsilon^2}*[\alpha\wedge *(d_A\alpha+[\alpha\wedge\alpha])]\\
&+\nabla_t[\psi,\alpha]
-[\psi,[\psi,\alpha]]+\frac 1{\varepsilon^2}[\alpha,d_A^*(A-A_0)+d_A^*\alpha]\\
&+[\psi,(\nabla_t\alpha-d_A\psi)]+\frac 1{\varepsilon^2}[\alpha,*[\alpha\wedge*(A-A_0)]]\\
&-\frac 1{\varepsilon^2}d_A*[\alpha\wedge*(A-A_0)]\\
&-[\alpha\wedge (\nabla_t(\Psi-\Psi_0)+\nabla_t\psi+[\psi,((\Psi-\Psi_0)+\psi)])]\\
&-d_A[\psi,((\Psi-\Psi_0)+\psi)],
\end{split}
\end{equation}
\begin{equation}\label{eq:c1term3}
\begin{split}
C_2(A,\Psi)(\alpha,\psi)=
&\frac1{\varepsilon^2}*[\alpha\wedge*(\nabla_t\alpha-d_A\psi-[\alpha,\psi])-\frac1{\varepsilon^2}d_A^*[\psi,\alpha]\\
&+\frac 1{\varepsilon^2}[\psi,(d_A^*(A-A_0+\alpha)-*[\alpha\wedge *(A-A_0)])]\\
&+\frac 1{\varepsilon^2}\nabla_t*[\alpha\wedge*(A-A_0)]\\
&-[\psi,(\nabla_t(\Psi-\Psi_0+\psi)+[\psi,(\Psi-\Psi_0)])]-\nabla_t[\psi,(\Psi-\Psi_0)]
\end{split}
\end{equation}
and if we estimate term by term, we have
\begin{equation*}
\begin{split}
\|C_1(A,\Psi)(\alpha,\psi)&+C_2(A,\Psi)(\alpha,\psi) dt\|_{0,p,\varepsilon}\\
\leq& \frac 1{\varepsilon^2}c\,\|\alpha+\psi dt\|_{\infty,\varepsilon}
\|\alpha+\psi dt\|_{1,p,\varepsilon}\\
&+\frac 1{\varepsilon^2}c\,\|\alpha+\psi dt\|_{\infty,\varepsilon}
\|A-A_0+(\Psi-\Psi_0) dt\|_{1,p,\varepsilon}.
\end{split}
\end{equation*}
Next, we consider
\begin{equation}\label{eq:c1term2har}
\begin{split}
\pi_{A_0}C_1(A,\Psi)(\alpha,\psi)
=&\pi_{A_0}\left(X_t(A+\alpha)-*X_t(A)-d*X(A)\alpha\right)\\
&+\pi_{A_0}\left(\frac 1{\varepsilon^2}*[\alpha\wedge *(d_A\alpha+[\alpha\wedge\alpha])]\right)\\
&-\pi_{A_0}\left(\frac 1{2\varepsilon^2}*[(A-A_0),*[\alpha\wedge\alpha]]+2[\psi,\nabla_t\pi_{A_0}(\alpha)]\right)
\\
&+\pi_{A_0}\left([\nabla_t\psi,\alpha]+2[\psi,\nabla_t(1-\pi_{A_0})\alpha]
-[\psi,[\psi,\alpha]]\right)\\
&+\pi_{A_0}\left(\frac 1{\varepsilon^2}[\alpha,d_A^*(A-A_0)+d_A^*\alpha]\right)\\
&+\pi_{A_0}\left(-[\psi,d_A\psi]+\frac 1{\varepsilon^2}[\alpha,*[\alpha\wedge*(A-A_0)]]\right)\\
&-\pi_{A_0}\left(\frac 1{\varepsilon^2}\left[(A-A_0),*[\alpha\wedge*(A-A_0)]\right]\right)\\
&+\pi_{A_0}\left(-[\alpha\wedge (\nabla_t(\Psi-\Psi_0)+\nabla_t\psi)]\right)\\
&+\pi_{A_0}\left(-[\alpha\wedge[\psi,((\Psi-\Psi_0)+\psi)]]\right)\\
&-\pi_{A_0}\left(\left[(A-A_0),[\psi,((\Psi-\Psi_0)+\psi)]\right]\right),
\end{split}
\end{equation}
thus if we estimate all the summands we obtain (\ref{eq:c1term22}).

\end{proof}

\section{The map $\mathcal T^{\varepsilon,b}$ between the critical connections} \label{section:construction}

In this section we will defined the map $\mathcal T^{\varepsilon,b}$ which relates the perturbed closed geodesics to the perturbed Yang-Mills connections and for this purpose we assume that the Jacobi operator is invertible for every geodesic. The definition will be based on the following two theorems.

\begin{theorem}[Existence]\label{thm:existence}
We choose a regular energy level $b$ of $E^H$ and $p\geq2$. There are costants $\varepsilon_0,\,c>0$ such that the following holds. If $\Xi^0=A^0+\Psi^0 dt\in \mathrm{Crit}^b_{E^H}$ is a perturbed closed geodesic and 
$$\alpha_0^\varepsilon(t)\in \textrm{im } \left(d_{A^0(t)}^*:\Omega^2(\Sigma,\mathfrak g_P)\to \Omega^1(\Sigma,\mathfrak g_P)\right)$$
is the unique solution of
\begin{equation}\label{eq:firststep}
d_{A^0}^*d_{A^0}\alpha_0^\varepsilon
=\varepsilon^2\nabla_t(\partial_tA^0-d_{A^0}\Psi^0)+\varepsilon^2
*X_t(A^0),
\end{equation}
then, for any positive $\varepsilon<\varepsilon_0$, there is a perturbed Yang-Mills connection $\Xi^\varepsilon\in\mathrm{Crit}^b_{\mathcal{YM}^{\varepsilon,H}}$ which satisfies 
\begin{equation}\label{eq:cbb}
d_{\Xi^0}^{*_\varepsilon}\left(\Xi^\varepsilon-\Xi^0\right)=0,\quad
 \left\| \Xi^\varepsilon-\Xi^0
\right\|_{2,p,\varepsilon}\leq c\varepsilon^2
\end{equation}
and, for $\alpha+\psi dt:=\Xi^\varepsilon-\Xi^0$,
\begin{equation}\label{eq:cbb2}
\left\| (1-\pi_{A^0})(\alpha-\alpha_0^\varepsilon)
\right\|_{2,p,\varepsilon}+\varepsilon\left\|\psi dt\right\|_{2,p,\varepsilon} \leq c\varepsilon^{4},
\end{equation}
\begin{equation}\label{eq:cbb2cxy}
\left\|\pi_{A^0}(\alpha)
\right\|_{2,p,1}+\varepsilon\left\|\pi_{A^0}(\alpha)\right\|_{L^\infty} \leq c\varepsilon^{2}.
\end{equation}
\end{theorem}
\begin{remark}
As we already mentioned, a similar version of the theorem \ref{thm:existence} was proved by Hong in \cite{MR1715156} for $p=2$ and we refer to \cite{remyj} for a complete proof in our setting; the proof for a general $p$ follows in the same way.
\end{remark}

\begin{remark}
The operator $d_{\Xi^0}^{*_\varepsilon}$ is defined using the $L^2$-inner product as we explained in the section \ref{sec:pre} and thus, it does not depend on the choice of $p$.
\end{remark}

\begin{theorem}[Local uniqueness]\label{thm:localuniqueness}
For any perturbed geodesic $\Xi^0\in \mathrm{Crit}^b_{E^H}$  and any $c>0$ there are an $\varepsilon_0>0$ and a $\delta>0$ such that the following holds for any positive $\varepsilon<\varepsilon_0$. If $\Xi^\varepsilon$, $\bar\Xi^\varepsilon$ are two perturbed Yang-Mills connections that satisfy the condition
$$d_{\Xi^0}^{*_\varepsilon}\left(\Xi^\varepsilon-\Xi^0\right)=d_{\Xi^0}^{*_\varepsilon}\left(\bar\Xi^\varepsilon-\Xi^0\right)=0$$ and the estimates
$$\varepsilon\left\| \Xi^\varepsilon-\Xi^0
\right\|_{2,p,\varepsilon} + \left\| (1-\pi_{A^0})(\Xi^\varepsilon-\Xi^0-\alpha_0^\varepsilon)
\right\|_{1,p,\varepsilon}\leq c\varepsilon^3$$
with $\alpha_0^\varepsilon$ defined uniquely as in (\ref{eq:firststep}) and
\begin{equation}\label{eq:themap22}
\left\|\bar \Xi^\varepsilon-\Xi^0
\right\|_{1,p,\varepsilon}+\left\|\bar \Xi^\varepsilon-\Xi^0
\right\|_{\infty,\varepsilon}\leq \delta\varepsilon,
\end{equation}
then $\bar \Xi^\varepsilon=\Xi^\varepsilon$.
\end{theorem}

If a connection $\tilde\Xi^\varepsilon\in \mathcal A(P\times S^1)$ satisfies $
\left\|\tilde \Xi^\varepsilon-\Xi^0
\right\|_{2,p,\varepsilon}\leq \delta' \varepsilon^{1+\frac 1p},
$ then it follows from the Sobolev embedding theorem \ref{lemma:sobolev}, that $\tilde\Xi^\varepsilon$ satisfies (\ref{eq:themap22}) with $\delta=(1+c_s)\delta'$, where $c_s$ ist the constant of theorem \ref{lemma:sobolev}. Therefore the inequality $
\left\|\tilde \Xi^\varepsilon-\Xi^0
\right\|_{2,p,\varepsilon}\leq c\varepsilon^{2}
$ implies (\ref{eq:themap22}) whenever $\varepsilon<\varepsilon_1$ and $\varepsilon_1$ is sufficiently small, i.e. if $$\varepsilon_1\leq \min \left\{\varepsilon_0,\left(\frac{\delta}{2c_sc}\right)^{\frac 1{1-\frac1p}}\right\}$$ where $\varepsilon_0$ is given in theorem \ref{thm:localuniqueness}. Thus, if we choose in the theorem \ref{thm:localuniqueness} $\varepsilon_0$ satisfying $c\varepsilon_0+c_sc\varepsilon_0^{1-\frac 1p} <\delta$ we have that, for each $0<\varepsilon<\varepsilon_0$, in the ball $B_{c\varepsilon^2}\left(\Xi^0, \|\cdot \|_{2,p,\varepsilon}\right)$ there is a unique perturbed Yang-Mills connection $\Xi^\varepsilon$ which satisfies the condition $d_{\Xi^0}^{*_\varepsilon}(\Xi^\varepsilon-\Xi^0)=0$. 

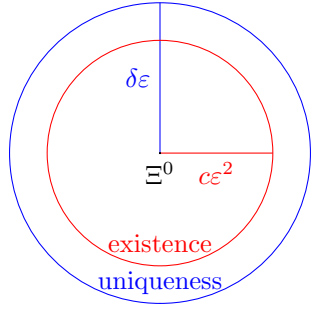
\begin{figure}[ht]
\begin{center}
\begin{tikzpicture}[scale=2] 
\draw[blue] (0,0) circle (1cm); 
\draw[blue] (0,0)--node[left]{$\delta\varepsilon$}(0,1);
\draw[red] (0,0) circle (0.75cm); 
\draw[red] (0,0)--node[below]{$c\varepsilon^2$}(0.75,0);

\draw (0,0) circle (0.1pt) node[below]{$\Xi^0$};

\draw[red] (0,-0.72) circle (0pt) node[above]  {existence} ; 
\draw[blue] (0,-1) circle (0.01pt) node[above]  {uniqueness} ; 
\end{tikzpicture}
\caption{Existence and uniqueness.}
\end{center}
\end{figure}

\begin{definition}\label{thm:defT}
For every regular value $b>0$ of the energy $E^H$  there are three positive constants $\varepsilon_0$, $\delta$ and $c$ such that the assertions of the theorems \ref{thm:existence} and \ref{thm:localuniqueness} hold with these constants. Shrink $\varepsilon_0$ such that $c\varepsilon_0+cc_s\varepsilon_0^{1-\frac 1p}<\delta$, where $c_s$ is the constant of Sobolev theorem \ref{lemma:sobolev}.  Theorems \ref{thm:existence} and \ref{thm:localuniqueness} assert that, for every $\Xi^0\in\mathrm{Crit}_{E^H}^b$ and every $\varepsilon$ with $0<\varepsilon<\varepsilon_0$, there is a unique perturbed Yang-Mills connection $\Xi^\varepsilon\in\mathrm{Crit}_{\mathcal{YM}^{\varepsilon,H}}^b$ satisfying
\begin{equation}\label{eq:themap2348}
\left\|\Xi^\varepsilon-\Xi^0\right\|_{2,p,\varepsilon}\leq c\varepsilon^2,\quad d_{\Xi^0}^{*_\varepsilon}(\Xi^\varepsilon-\Xi^0)=0.
\end{equation}
We define the map $\mathcal T^{\varepsilon, b}: \mathrm{Crit}_{E^H}^{b} \to \mathrm{Crit}_{\mathcal{YM}^{\varepsilon,H}}^b$ by $\mathcal{T}^{\varepsilon,b}(\Xi^0):=\Xi^\varepsilon$ where $\Xi^\varepsilon\in\mathrm{Crit}_{\mathcal{YM}^{\varepsilon,H}}^b$ is the unique Yang-Mills connection satisfying (\ref{eq:themap2348}).
\end{definition}

 The map $\mathcal{T}^{\varepsilon,b }$ is gauge equivariant because the construction of the perturbed Yang-Mills connection in the proof of theorem \ref{thm:existence} is gauge equivariant, since the map $\mathcal F^\varepsilon$ and the operator $\mathcal D^\varepsilon$ are so. Furthermore, since $\mathcal G_0(P)$ acts free on $\mathcal A(P)$, the gauge group $\mathcal G_0(P\times S^1)$ acts freely on $\mathcal A(P\times S^1)$ and on the set  $\mathrm{Crit}_{E^H}^b$ and thus $\mathcal T^{\varepsilon,b}$ defines a unique map
 \begin{equation}\label{creogmapeiwo}
\bar {\mathcal T}^{\varepsilon,b}:\mathrm{Crit}_{E^H}^b/\mathcal G_0(P\times S^1)\to 
 \mathrm{Crit}_{\mathcal {YM}^{\varepsilon,b}}^b/\mathcal G_0(P\times S^1).
 \end{equation}
 In addition, there is a $\gamma>0$ which bounds from below the distance between any two different perturbed geodesics on $\mathcal M^g(P)$. Therefore the map $\mathcal{T}^{\varepsilon,b }$ is injective if we choose $\varepsilon<\varepsilon_1$ such that $2c\varepsilon_1^2<\gamma$ and $\varepsilon_1<\varepsilon_0$, where $c$ and $\varepsilon_0$ are the constants in the last definition.\\

Next, we state two useful lemmas concerning the $1$-form $\alpha_0^\varepsilon$; the first one follows from the regularity properties of the geodesics (cf. \cite{MR1715156} or \cite{remyj}).
\begin{lemma}\label{lemma:step1xi1}
For any perturbed geodesic $\Xi^0=A^0+\Psi^0 dt\in \mathrm{Crit}^b_{E^H}$ there is a unique 1-form $\alpha_0^\varepsilon$, $\alpha_0(t)\in\Omega^1(\Sigma,\mathfrak g_P)$, which satisfies
\begin{equation}
d_{A^0}^*d_{A^0}\alpha_0^\varepsilon
=\varepsilon^2\nabla_t(\partial_tA^0-d_{A^0}\Psi^0)+\varepsilon^2
*X_t(A^0),\quad \alpha_0^\varepsilon\in\textrm{im } d_{A^0}^*.
\end{equation}
In addition there is a constant $c>0$ such that
\begin{equation}\label{eq:esa0}
\|\alpha_0^\varepsilon\|_{2,p,1}+\|\alpha_0^\varepsilon \|_{L^\infty}+\|d_{A^0}\alpha_0^\varepsilon \|_{L^\infty}+\|\nabla_t\alpha_0^\varepsilon \|_{L^\infty}\leq c\varepsilon^2
\end{equation}
for any $varepsilon$ and for $\Xi^\varepsilon_1:=\Xi^0+\alpha_0^\varepsilon\in \mathcal A(P\times S^1)$
\begin{equation}\label{eq:curvxi22}
\left\|\mathcal F_1^\varepsilon\left(\Xi^\varepsilon_1\right)\right\|_{L^p}
\leq c\varepsilon^{2},\,
\left\|\mathcal F_2^\varepsilon\left(\Xi_1^\varepsilon\right)\right\|_{L^p}\leq  c.
\end{equation}
\end{lemma}

\begin{lemma}\label{lemma:evaluate3}
For any perturbed geodesic $\Xi^0=A^0+\Psi^0dt$ and for $\Xi_1^\varepsilon$ defined as in lemma \ref{lemma:step1xi1} the following holds. There exist two constants $c>0$ and $\varepsilon_0>0$ such that
\begin{equation}\label{eq:lemma:evaluate31}
\begin{split}
\|\pi_{A^0}(\alpha)\|_{L^p}&+\|\nabla_t\pi_{A^0}(\alpha)\|_{L^p}+\|\nabla_t^2\pi_{A^0}(\alpha)\|_{L^p}\\
\leq& c\varepsilon 
\left\|\mathcal D^\varepsilon(\Xi_1^\varepsilon)(\alpha,\psi)\right\|_{0,p,\varepsilon}+c\left\|\pi_{A^0}\mathcal D_1^\varepsilon(\Xi_1^\varepsilon)(\alpha,\psi)\right\|_{0,p,\varepsilon},
\end{split}
\end{equation}
\begin{equation}\label{eq:lemma:evaluate32}
\begin{split}
\|\alpha-&\pi_{A^0}(\alpha)+\psi\, dt\|_{2,p,\varepsilon}\\
\leq &c\varepsilon^2 
\left\|\mathcal D^\varepsilon(\Xi_1^\varepsilon)(\alpha,\psi)
\right\|_{0,p,\varepsilon}+c\varepsilon\left\|\pi_{A^0}\mathcal D_1^\varepsilon(\Xi_1^\varepsilon)(\alpha,\psi)\right\|_{0,p,\varepsilon},
\end{split}
\end{equation}
\begin{equation}\label{eq:lemma:evaluate3287}
\begin{split}
\|\alpha-\pi_{A^0}(\alpha)\|_{2,p,\varepsilon}
\leq& c\varepsilon^2 
\left\|\mathcal D^\varepsilon_1(\Xi_1^\varepsilon)(\alpha,\psi)
\right\|_{L^p}+c\varepsilon^4 
\left\|\mathcal D^\varepsilon_2(\Xi_1^\varepsilon)(\alpha,\psi)
\right\|_{L^p},
\end{split}
\end{equation}
for every $\alpha+\psi\,dt\in {\mathrm W}^{2,p}$ and any positive $\varepsilon<\varepsilon_0$.
\end{lemma}

\begin{proof}[Proof of lemma \ref{lemma:evaluate3}]
On the one side by the quadratic estimate (\ref{eq:easyww}) 
\begin{equation}\label{eq:lemma123}
\begin{split}
\|\mathcal D^\varepsilon(\Xi^\varepsilon_1)(\alpha,\psi)&-
\mathcal D^\varepsilon(\Xi^0)(\alpha,\psi)\|_{0,p,\varepsilon}\\
\leq& c\varepsilon^{-2} \left(\|\alpha_0^\varepsilon\|_{L^\infty}+\|d_{A^0}\alpha_0^\varepsilon\|_{L^\infty}+\varepsilon\|\nabla_t\alpha_0^\varepsilon\|_{L^\infty}\right)\|\alpha+\psi\, dt\|_{1,p,\varepsilon}\\
\leq& c \|\alpha+\psi\, dt\|_{1,p,\varepsilon}.
\end{split}
\end{equation}
where the last estimate follows from (\ref{eq:esa0}). On the other side, we remark that the $\omega$ defined by ({eq:thm:geod:dasdsgf}) is exactly $\frac 1{\varepsilon^2}d_{A^0}\alpha_0^\varepsilon$ and thus for the harmonic part we obtain
\begin{equation*}
\begin{split}
\pi_{A^0}&\left(\mathcal D_1^\varepsilon(\Xi^\varepsilon_1)(\alpha,\psi)-
\left(\mathcal D_1^\varepsilon(\Xi^0)(\alpha,\psi)
-\frac1{\varepsilon^2}*[\alpha\wedge *d_{A^0}\alpha_0^\varepsilon]\right)\right)\\
=&\pi_{A^0}\Big(-\frac1{\varepsilon^2}*
\left[\alpha\wedge*\frac12[\alpha_0^\varepsilon\wedge\alpha_0^\varepsilon]\right]
-\frac1{\varepsilon^2}*[\alpha_0^\varepsilon\wedge *[\alpha_0^\varepsilon\wedge\alpha]]-2\left[\psi,\nabla_t\alpha_0^\varepsilon\right]\\
&-\frac 1{\varepsilon^2}*[\alpha_0^\varepsilon\wedge*d_{A^0}\alpha]
+\frac1{\varepsilon^2}\left[\alpha_0^\varepsilon\wedge\left(d_{A^0}^*\alpha-*[\alpha_0^\varepsilon\wedge*\alpha]\right)\right]\Big)
\end{split}
\end{equation*}
and hence
\begin{equation}\label{eq:lemma125}
\begin{split}
\Big|\Big|\pi_{A^0}\Big(\mathcal D_1^\varepsilon(&\Xi_1^\varepsilon)(\alpha,\psi)-
\mathcal D_1^\varepsilon(\Xi^0)(\alpha,\psi)
+\frac1{\varepsilon^2}*[\alpha\wedge *d_{A^0}\alpha_0]\Big)\Big|\Big|_{0,p,\varepsilon}\\
\leq &\frac c{\varepsilon^2}\|\alpha_0^\varepsilon\|^2_{L^\infty}\|\alpha\|_{L^p}\\
&+\frac c{\varepsilon^2}\left(\|\alpha_0^\varepsilon\|_{L^\infty}+\varepsilon\|\nabla_t\alpha_0^\varepsilon\|_{L^\infty}\right)\|(1-\pi_{A^0})\alpha+\psi\, dt\|_{1,p,\varepsilon}\\
\leq&c\varepsilon^2 \|\pi_{A^0}(\alpha)\|_{L^p}
+c \|(1-\pi_{A^0})\alpha+\psi\,dt\|_{1,p,\varepsilon}.
\end{split}
\end{equation}
By the lemma \ref{lemma:evaluate2} we have
\begin{equation*}
\begin{split}
\big\|\left(1-\pi_{A^0}\right)\alpha&+\psi dt\big\|_{2,p,\varepsilon}+\varepsilon\left\|\pi_{A^0}(\alpha)\right\|_{2,p,1}\\
\leq&c\varepsilon^2
\|\mathcal D^\varepsilon(\Xi^0)(\alpha,\psi)\|_{0,p,\varepsilon}+c\varepsilon \|\pi_{A^0}\left(\mathcal D_1^\varepsilon(\Xi^0)(\alpha,\psi)+*[\alpha\wedge*\omega]\right)\|_{L^p}\\
\leq&c\varepsilon^2
\|\mathcal D^\varepsilon(\Xi^\varepsilon_1)(\alpha,\psi)\|_{0,p,\varepsilon}+c\varepsilon\|\pi_{A^0}\mathcal D_1^\varepsilon(\Xi_1^\varepsilon)(\alpha,\psi)\|_{0,p,\varepsilon}\\
&+c\varepsilon\|\alpha-\pi_{A^0}\alpha+\psi\,dt\|_{1,p,\varepsilon}
+c\varepsilon^{2} \|\pi_{A^0}\alpha\|_{1,p,\varepsilon}.
\end{split}
\end{equation*}
where the second inequality follows from (\ref{eq:lemma123}) and (\ref{eq:lemma125}). Therefore (\ref{eq:lemma126}) implies the first and the second estimate of the lemma choosing $\varepsilon$ sufficiently small. The third estimates follows combining (\ref{eq:lemma:evaluate2277}), (\ref{eq:lemma123}), (\ref{eq:lemma125}) with the first two inequality of the lemma:

\begin{equation}\label{eq:lemma126}
\begin{split}
\big\|\left(1-\pi_{A^0}\right)\alpha\big\|_{2,p,\varepsilon}\leq&c\varepsilon^2
\|\mathcal D^\varepsilon(\Xi^0)(\alpha,\psi)\|_{L^p}+c\varepsilon^4
\|\mathcal D^\varepsilon(\Xi^0)(\alpha,\psi)\|_{L^p}\\
&+c\varepsilon^2 \|\pi_{A^0}\left(\mathcal D_1^\varepsilon(\Xi^0)(\alpha,\psi)+*[\alpha\wedge*\omega]\right)\|_{L^p}\\
\leq&c\varepsilon^2
\|\mathcal D^\varepsilon_1(\Xi^\varepsilon_1)(\alpha,\psi)\|_{L^p}+c\varepsilon^4
\|\mathcal D^\varepsilon_2(\Xi^\varepsilon_1)(\alpha,\psi)\|_{L^p}\\
&+c\varepsilon^2\|\alpha+\psi\,dt\|_{1,p,\varepsilon}\\
\leq&c\varepsilon^2
\|\mathcal D^\varepsilon_1(\Xi^\varepsilon_1)(\alpha,\psi)\|_{L^p}+c\varepsilon^4
\|\mathcal D^\varepsilon_2(\Xi^\varepsilon_1)(\alpha,\psi)\|_{L^p}
\end{split}
\end{equation}
\end{proof}

\begin{proof}[Proof of theorem \ref{thm:localuniqueness}]
 Since $\Xi^0$ is a geodesic, by lemma \ref{lemma:step1xi1} we can define a connection $\Xi_1^\varepsilon=\Xi^0+\alpha_0^\varepsilon$ such that  $\|\alpha_0^\varepsilon\|_{2,p,1}+\|d_{A_0}\alpha_0^\varepsilon \|_{L^\infty}+\varepsilon\|\nabla_t\alpha_0^\varepsilon \|_{L^\infty}+\|\alpha_0^\varepsilon \|_{L^\infty}\leq c\varepsilon^2$ and
\begin{equation}\label{eq:locuniq667}
\|\mathcal F_1^\varepsilon(\Xi_1^\varepsilon)\|_{L^p}
\leq c\varepsilon^{2},\quad
\|\mathcal F_2^\varepsilon(\Xi_1^\varepsilon)\|_{L^p}\leq c.
\end{equation}
Therefore we have, for $\bar\Xi^\varepsilon-\Xi_1^\varepsilon=:\alpha^\varepsilon+\psi^\varepsilon\, dt$ and $c\varepsilon<\delta$,
\begin{equation}\label{eq:convradius2}
\|\bar\Xi^\varepsilon-\Xi_1^\varepsilon\|_{1,p,\varepsilon}+\|\bar\Xi^\varepsilon-\Xi_1^\varepsilon\|_{\infty,\varepsilon}
\leq 2\delta \varepsilon
\end{equation}
and for $i=1,2$, since $\bar \Xi^\varepsilon$ is a Yang-Mills connection which satisfies $d_{\Xi^0}^{*_\varepsilon}\left(\bar\Xi^\varepsilon-\Xi^0\right)=0$, and thus $\mathcal F^\varepsilon(\bar\Xi^\varepsilon)=0$,
\begin{equation}\label{eq:fruttidibosco}
\mathcal D_i^\varepsilon(\Xi_1^\varepsilon)(\bar\Xi^\varepsilon-\Xi_1^\varepsilon)=-C^\varepsilon_i(\Xi_1^\varepsilon)(\bar\Xi^\varepsilon-\Xi_1^\varepsilon)-\mathcal F_i^\varepsilon(\Xi_1^\varepsilon).
\end{equation}
By lemma \ref{lemma:evaluate3} we get
\begin{equation*}
\begin{split}
\|(1-\pi_{A^0})\alpha^\varepsilon&+\psi^\varepsilon\,dt\|_{2,p,\varepsilon}+\varepsilon\|\pi_{A^0}(\alpha^\varepsilon)\|_{L^p}+\varepsilon\|\nabla_t\pi_{A^0}(\alpha^\varepsilon)\|_{L^p}\\
\leq&\, c\left(\varepsilon^2\|\mathcal D^\varepsilon(\Xi_1^\varepsilon)(\bar\Xi^\varepsilon-\Xi_1^\varepsilon)\|_{0,p,\varepsilon}+\varepsilon \|\pi_{A^0}(\mathcal D^\varepsilon(\Xi_1^\varepsilon)(\bar\Xi^\varepsilon-\Xi_1^\varepsilon))\|_{0,p,\varepsilon}\right)\\
\leq& c\varepsilon^{2}\|C^\varepsilon(\Xi_1^\varepsilon(\bar\Xi^\varepsilon-\Xi_1^\varepsilon)\|_{0,p,\varepsilon}
+c\varepsilon^{2}\|\mathcal F^\varepsilon(\Xi_1^\varepsilon)\|_{0,p,\varepsilon}\\
& +c\varepsilon\|\pi_{A^0}(C_1^\varepsilon(\Xi_1^\varepsilon)(\bar\Xi^\varepsilon-\Xi_1^\varepsilon))\|_{0,p,\varepsilon}
+c\varepsilon\|\pi_{A^0}(\mathcal F^\varepsilon_1(\Xi_1^\varepsilon))\|_{0,p,\varepsilon}\\
\leq&c\varepsilon^3+c\delta\|(1-\pi_{A^0})\alpha^\varepsilon+\psi^\varepsilon\,dt\|_{1,p,\varepsilon}\\
&+c\delta\left(\varepsilon\|\pi_{A^0}(\alpha^\varepsilon)\|_{L^p}+\varepsilon\|\pi_{A^0}(\alpha^\varepsilon)\|_{L^p}\right)
\end{split}
\end{equation*}
where in the second step we use (\ref{eq:fruttidibosco}) and the third step follows from lemma \ref{lemma:estimate:c} and the estimate of the curvatures (\ref{eq:locuniq667}). Thus we proved the estimates  $\|\bar\Xi^\varepsilon-\Xi_1^\varepsilon\|_{2,p,\varepsilon}\leq c\varepsilon^2$ and hence $\|\bar\Xi^\varepsilon-\Xi^\varepsilon\|_{2,p,\varepsilon}\leq c\varepsilon^2$. Since $\Xi^\varepsilon$ satisfies $\mathcal F^\varepsilon(\Xi^\varepsilon)=0$ by the assumptions, we can write
\begin{equation*}
\begin{split}
\mathcal D_i^\varepsilon(\Xi_1^\varepsilon)(\bar\Xi^\varepsilon-\Xi^\varepsilon)
=&\left(\mathcal D_i^\varepsilon(\Xi^\varepsilon)
+\left(\mathcal D_i^\varepsilon(\Xi_1^\varepsilon)-\mathcal D_i^\varepsilon(\Xi^\varepsilon)\right)\right)
(\bar\Xi^\varepsilon-\Xi^\varepsilon)\\
=&-C_i^\varepsilon(\Xi^\varepsilon)(\bar\Xi^\varepsilon-\Xi^\varepsilon)
+\left(\mathcal D_i^\varepsilon(\Xi_1^\varepsilon)
-\mathcal D_i^\varepsilon(\Xi^\varepsilon)\right)(\bar\Xi^\varepsilon-\Xi^\varepsilon)
\end{split}
\end{equation*}
and by the quadratic estimates of the chapter \ref{chapter:qe}
\begin{equation*}
\begin{split}
 \varepsilon^{2}\|C^\varepsilon(\Xi_1^\varepsilon)&(\Xi^\varepsilon-\Xi_1^\varepsilon)\|_{0,p,\varepsilon}
 +c\varepsilon\|\pi_{A^0}(C_1^\varepsilon(\Xi_1^\varepsilon)(\bar\Xi^\varepsilon-\Xi_1^\varepsilon))\|_{0,p,\varepsilon}\\
\leq&c \varepsilon^{1-\frac 1p} \|(\bar\Xi^\varepsilon-\Xi^\varepsilon)-\pi_{A^0}(\bar\Xi^\varepsilon-\Xi^\varepsilon)\|_{1,p,\varepsilon}\\
&+c\varepsilon^{2-\frac 1p}\|\nabla_t\pi_{A^0}(\bar\Xi^\varepsilon-\Xi^\varepsilon)\|_{0,p,\varepsilon}+c\varepsilon^{1-\frac 1p}\|\pi_{A^0}(\bar\Xi^\varepsilon-\Xi^\varepsilon)\|_{0,p,\varepsilon},
\end{split}
\end{equation*}
\begin{equation*}
\begin{split}
 \varepsilon^{2}\|(\mathcal D_i^\varepsilon(\Xi_1^\varepsilon)&
-\mathcal D_i^\varepsilon(\Xi^\varepsilon))(\bar\Xi^\varepsilon-\Xi^\varepsilon)\|_{0,p,\varepsilon}\\
& +c\varepsilon\|\pi_{A^0}(\left(\mathcal D_i^\varepsilon(\Xi_1^\varepsilon)
-\mathcal D_i^\varepsilon(\Xi^\varepsilon)\right)(\bar\Xi^\varepsilon-\Xi^\varepsilon))\|_{0,p,\varepsilon}\\
\leq&c \varepsilon^{1-\frac 1p} \|(\bar\Xi^\varepsilon-\Xi^\varepsilon)-\pi_{A^0}(\bar\Xi^\varepsilon-\Xi^\varepsilon)\|_{1,p,\varepsilon}\\
&+c\varepsilon^{1-\frac 1p}\|\nabla_t\pi_{A^0}(\bar\Xi^\varepsilon-\Xi^\varepsilon)\wedge dt\|_{0,p,\varepsilon}+c\varepsilon^{2-\frac 1p}\|\pi_{A^0}(\bar\Xi^\varepsilon-\Xi^\varepsilon)\|_{0,p,\varepsilon},
\end{split}
\end{equation*}
we obtain by the lemma \ref{lemma:evaluate3}
\begin{equation*}
\begin{split}
\|(1-\pi_{A^0})(\bar\Xi^\varepsilon&-\Xi^\varepsilon) \|_{2,p,\varepsilon}+\varepsilon\|\pi_{A^0}(\bar\Xi^\varepsilon-\Xi^\varepsilon) \|_{L^p}\\
&+\varepsilon\|\nabla_t\pi_{A^0}(\bar\Xi^\varepsilon-\Xi^\varepsilon) \|_{2,p,1}\\
\leq&
c\varepsilon^2 \|\mathcal D^\varepsilon(\Xi_1)(\bar\Xi^\varepsilon-\Xi^\varepsilon)\|_{0,p,\varepsilon}
+c\varepsilon \|\pi_{A^0}\mathcal D_1^\varepsilon(\Xi_1)(\bar\Xi^\varepsilon-\Xi^\varepsilon)\|_{0,p,\varepsilon}\\
\leq&c\varepsilon^{1-\frac 1p}\|\bar\Xi^\varepsilon-\Xi^\varepsilon-\pi_{A^0}(\bar\Xi^\varepsilon-\Xi^\varepsilon) \|_{2,p,\varepsilon}\\
&+c\varepsilon^{2-\frac 1p}\left(\|\pi_{A^0}(\bar\Xi^\varepsilon-\Xi^\varepsilon) \|_{L^p}+\|\nabla_t\pi_{A^0}(\bar\Xi^\varepsilon-\Xi^\varepsilon) \|_{L^p}\right)
\end{split}
\end{equation*}
and thus, $\|\bar\Xi^\varepsilon-\Xi^\varepsilon\|_{2,p,\varepsilon}=0$ and hence $\bar\Xi^\varepsilon=\Xi^\varepsilon$ in  for $\varepsilon$ small enough.
\end{proof}

\textbf{Local uniquess modulo gauge.} The following theorem states a uniqueness property. The result is interesting, but it will not be used in the next chapters and in particular it will not enter in the proof of the surjectivity of $\mathcal{T}^{\varepsilon,b }$ on the contrary to what one might expect.
\begin{theorem}[Uniqueness]\label{thm:uniqueness}
We choose $p> 3$. For every perturbed geodesic $\Xi^0\in \mathrm{Crit}^b_{E^H}$ there are constants $\varepsilon_0$, $\delta_1>0$ such that the following holds. If $0<\varepsilon<\varepsilon_0$ and $\tilde \Xi^\varepsilon\in\mathrm{Crit}^b_{\mathcal{YM}^{\varepsilon,H }}$ is a perturbed Yang-Mills connection satisfying 
\begin{equation}\label{eq:uniqcond}
\left\| \tilde \Xi^\varepsilon-\Xi^0
\right\|_{1,p,\varepsilon}\leq \delta_1\varepsilon^{1+1/p},
\end{equation}
then there is a $g\in\mathcal G_0^{2,p}(P\times S^1)$ such that $g^* \tilde \Xi^\varepsilon=\mathcal T^{\varepsilon,b}(\Xi^0)$.
\end{theorem}

\begin{theorem}\label{thm:crgc}
Assume that $q\geq p>2$ and $q>3$. Let $\Xi^0=A^0+\Psi^0dt\in \mathcal A^{1,p}(P\times S^1)$ be a connection flat on the fibers, i.e. $F_{A^0}=0$. Then for every $c_0>0$ there exist $\delta_0>0$, $c>0$ such that the following holds for $0<\varepsilon\leq1$. If $\Xi=A+\Psi dt\in\mathcal A^{1,p}(P\times S^1)$ satisfies
\begin{equation}\label{eq:thm:crgc}
\left\| d_{A^0}^*(A-A^0)-\varepsilon^2\nabla_t^{\Psi^0}(\Psi-\Psi^0)\right\|_{L^p}\leq c_0 \varepsilon^{1/p}, \quad
\left\|\Xi-\Xi^0\right\|_{0,q,\varepsilon }\leq \delta_0\varepsilon^{1/q},
\end{equation}
then there exists a gauge transformation $g\in\mathcal G^{2,p}_0$ such that $d_{\Xi^0 }^{*_\varepsilon}(g^*\Xi-\Xi^0)=0$ and
\begin{equation}\label{eq:thm:crgc2}
\left\|g^*\Xi-\Xi\right\|_{1,p,\varepsilon}\leq c\varepsilon^2\left(1+\varepsilon^{-1/p}\|\Xi-\Xi^0 \|_{1,p,\varepsilon}\right)\left\|d_{\Xi^0 }^{*_\varepsilon}(\Xi-\Xi^0)\right\|_{L^p}.
\end{equation}
\end{theorem}

\begin{proof}
The proof is the same as that of proposition 6.2 in \cite{MR1283871}. In fact the theorem \ref{thm:crgc} is the 3-dimensional version of the proposition 6.2 in \cite{MR1283871}\footnote{The 0-form $d_{\Xi^0}^{*_\varepsilon}(\Xi-\Xi^0)$ in the cited proposition is defined by $d_{A^0}^*(A-A^0)-\varepsilon^2\nabla_t^{\Psi^0}(\Psi-\Psi^0)-\varepsilon^2\nabla_t^{\Phi^0}(\Phi-\Phi^0)$ and the norms are defined in chapter 4 of the paper.} which works with 4-dimensional connections. Between this two statements there are a few changes that are a consequence of the differences in the Sobolev properties  (theorem \ref{lemma:sobolev} above and lemma 4.1 in \cite{MR1283871}). Therefore here we can work with $q>3$ instead of $q>4$ because we have a 3-dimensional manifold and we do not need the condition $qp/(q-p)>4$; furthermore, we can replace $\varepsilon^{2/p}$,  $\varepsilon^{-2/p}$, $\varepsilon^{2/q}$ by $\varepsilon^{1/p}$,  $\varepsilon^{-1/p}$, $\varepsilon^{1/q}$ because in the proof of the Sobolev theorem \ref{lemma:sobolev} we rescale a 1-dimensional domain instead of a 2-dimensional one.  In addition, we remark that the gauge transformation $g$ is an element of  $\mathcal G^{2,p}_0(P)$ and this follows from the proof of the theorem; in fact, the gauge transformation $g\in \mathcal G^{2,p}(P\times S^1)$ is a limit of a sequence 
$\{g_i\}_{i\in \mathbb N}\subset \mathcal G^{2,p}(P\times S^1)$ defined by $g_i=\exp(\eta_0)\exp(\eta_1)...\exp(\eta_i)$ where $\eta_i\in W^{2,p}(\Sigma\times S^1,\mathfrak g_P)$ are 0-forms. Therefore, the sequence 
$\{g_i\}_{i\in \mathbb N}$ lies in the unit component of the gauge group and hence in
$\mathcal G^{2,p}_0(P\times S^1)$. 
\end{proof}

\begin{remark}
In the proof of the last theorem we can not use the local slice theorem directly, because although the operator $d_{\Xi^0}^{*_\varepsilon}d_{\Xi^0}$ is Fredholm and invertible on the complement of its kernel, the norm of its inverse depends on $\varepsilon$ and hence we do not obtain an estimate independent on the metric and thus not independent on $\varepsilon$.
\end{remark}

\begin{proof}[Proof of the uniqueness theorem \ref{thm:uniqueness}]
Let $\tilde \Xi^\varepsilon= A^\varepsilon+\Psi^\varepsilon dt$ be a perturbed Yang-Mills connection which satisfies (\ref{eq:uniqcond}) with $\Xi^0=A^0+\Psi^0 dt$; then 
\begin{equation}
\begin{split}
&\left\|d_{A^0}^*(A^\varepsilon-A^0)-\varepsilon^2\nabla_t^{\Psi^0}(\Psi^\varepsilon-\Psi^0)\right\|_{L^p}\\
&\qquad\qquad\leq\left\|d_{A^0}^*(A^\varepsilon-A^0)\right\|_{L^p}
+\varepsilon^2\left\|\nabla_t^{\Psi^0}(\Psi^\varepsilon-\Psi^0)\right\|_{L^p}\\
&\qquad\qquad\leq2\left\| \tilde \Xi^\varepsilon-\Xi^0\right\|_{1,p,\varepsilon}\leq 2\delta_1\varepsilon^{1+1/p}
\end{split}
\end{equation}
and therefore the first condition of the assumption (\ref{eq:thm:crgc}) of theorem \ref{thm:crgc} is satisfied for $\varepsilon$ sufficiently small; 
the second one follows if we choose $\delta_1\varepsilon<\delta_0$ and $q=p$. Thus there exists, by theorem \ref{thm:crgc}, a gauge transformation
$g\in \mathcal G_0^{2,p}$ such that $d_{\Xi^0}^*(g^*\tilde \Xi^\varepsilon-\Xi^0)=0$ and 
\begin{equation}\label{eq:proof:dfhsoi}
\begin{split}
\left\|g^*\tilde \Xi^\varepsilon-\tilde \Xi^\varepsilon\right\|_{1,p,\varepsilon}
\leq&c\varepsilon^2\left(1+\varepsilon^{-1/p}\left\|\tilde \Xi^\varepsilon-\Xi^0\right\|_{1,p,\varepsilon}\right)
\left\|d_{\Xi^0}^{*_\varepsilon}(\tilde \Xi^\varepsilon-\Xi^0)\right\|_{L^p}\\
\leq&2c\left\|d_{A^0}^*(A^\varepsilon-A^0)-\varepsilon^2\nabla_t^{\Psi^0}(\Psi^\varepsilon-\Psi^0)\right\|_{L^p}\\
\leq&4c\delta_1\varepsilon^{1+1/p}.
\end{split}
\end{equation}
Then 
\begin{equation}
\left\|g^*\tilde \Xi^\varepsilon-\Xi^0\right\|_{1,p,\varepsilon}
\leq \left\|g^*\tilde \Xi^\varepsilon-\tilde \Xi^\varepsilon\right\|_{1,p,\varepsilon}+\left\|\tilde \Xi^\varepsilon-\Xi^0\right\|_{1,p,\varepsilon}
\leq (4c\delta_1+\delta_1)\varepsilon^{1+1/p}
\end{equation}
and by the Sobolev embedding theorem \ref{lemma:sobolev} we have also that
\begin{equation}
\left\|g^*\tilde \Xi^\varepsilon-\Xi^0\right\|_{\infty,\varepsilon}
\leq c_s \varepsilon^{-1/p}\left\|g^*\tilde \Xi^\varepsilon-\Xi^0\right\|_{1,p,\varepsilon}
\leq c_s(4c+1)\delta_1\varepsilon,
\end{equation}
where $c_s$ is the constant in theorem \ref{lemma:sobolev}. Finally, we can apply theorem 
\ref{thm:localuniqueness} with $\delta_1<\delta/((c_s+1)(4c+1))$ for $\bar\Xi^\varepsilon=g^*\tilde \Xi^\varepsilon$ 
and $\Xi^\varepsilon=\mathcal T^{\varepsilon,b}(\Xi^0)$ and we can conclude that $g^*\tilde \Xi^\varepsilon=\mathcal T^{\varepsilon,b}(\Xi^0)$.
\end{proof}

\section{A priori estimates for the perturbed Yang-Mills connections}\label{chapter:aprioriestimates}

In this chapter we explain some a priori estimates that we will need to prove the surjectivity of the map $\mathcal T^{\varepsilon,b}$ and we organize them in three theorems. First we show some $L^2(\Sigma)$-estimates for the curvature term $F_A$ (theorem \ref{lemma:secder}), then the $L^2(\Sigma)$- and the $L^\infty(\Sigma)$-estimates for the curvature term $\partial_tA-d_A\Psi$ (theorems \ref{crit:secder} and \ref{thm:norminty}).

\begin{theorem}\label{lemma:secder}
We choose $p\geq2$ and two constants $b$, $c_1>0$. Then there are two positive constants $\varepsilon_0$, $c$ such that the following holds. For any perturbed Yang-Mills connection $A+\Psi dt\in \mathrm{Crit}_{\mathcal {YM}^{\varepsilon,H}}^b$, with $0<\varepsilon<\varepsilon_0$, which satisfies 
\begin{equation}
\sup_{t\in S^1}\|\partial_tA-d_{A}\Psi\|_{L^4(\Sigma)}\leq c_1,
\end{equation}
we have the estimates\footnote{The operator $\nabla_t$ is defined using $\Psi$.}
\begin{equation}\label{eq:lemma:secder1}
\|F_{A}\|_{3,2,\varepsilon}\leq c\varepsilon^2,
\end{equation}
\begin{equation}\label{eq:lemma:secder3}
\begin{split}
\sup_{t\in S^1}\Big(&\|F_{A}\|_{L^2(\Sigma)}+\|F_{A}\|_{L^\infty(\Sigma)}+\|d_A^*F_{A}\|_{L^2(\Sigma)}\\
&+\|d_Ad_A^*F_{A}\|_{L^2(\Sigma)}+\varepsilon\|\nabla_tF_{A}\|_{L^2(\Sigma)}+\varepsilon^{2}\|\nabla_t\nabla_tF_{A}\|_{L^2(\Sigma)}\Big)\leq c\varepsilon^{2-1/p}.
\end{split}
\end{equation}
\end{theorem}

\begin{theorem}\label{crit:secder}
We choose two constants $c_1,c_2>0$, an open interval $\Omega\subset \mathbb R$ and a compact set $K\subset \Omega$. Then there are two positive constants $\delta$, $c$ such that the following holds. For any perturbed Yang-Mills connection $A+\Psi dt\in \mathrm{Crit}_{\mathcal {YM}^{\varepsilon,H}}^\infty$ which satisfies 
\begin{equation}
\sup_{t\in \Omega}\|F_A\|_{L^2(\Sigma)}\leq \delta,\quad\sup_{t\in \Omega} \|\partial_tA-d_{A}\Psi\|_{L^4(\Sigma)}\leq c_1,
\end{equation}
we have the estimates, for $B_t=\partial_tA-d_{A}\Psi$,
\begin{equation}\label{eq:crit:secder3}
\sup_{t\in K}\varepsilon^2\|B_t\|_{L^2(\Sigma)}^2
\leq c\int_{\Omega}\left(\varepsilon^2\|B_t\|_{L^2(\Sigma)}^2+\|F_A\|_{L^2(\Sigma)}^2+\varepsilon^2c_{\dot X_t(A)}\right) dt,
\end{equation}
\begin{equation}\label{eq:crit:secder4}
 \sup _{t\in K}\|d_AB_t\|_{L^2(\Sigma)}^2\leq c\int_{\Omega}\left(\|d_AB_t\|^2_{L^2(\Sigma)}+\frac 1{\varepsilon^2}\|F_A\|^2_{L^2(\Sigma)}+\|B_t\|^2_{L^2(\Sigma)}\right) dt
\end{equation}
where $\sqrt{c_{\dot X_t(A)}}$ is a constant which bounds the $L^\infty$-norm of $\dot X_t(A)$. The constants $c$ and $\delta$ depend on $\Omega$ and on $K$, but only on their length and on the distance between their boundaries. Furthermore, if $0<\varepsilon<c_2$, then
\begin{equation}\label{ap:crit:se}
\sup_{t\in S^1}\|d_A^*d_AB_t\|_{L^2(\Sigma)}^2
\leq c\int_{S^1}\left(\varepsilon^2\|B_t\|^2_{L^2(\Sigma)}+\|F_A\|^2_{L^2(\Sigma)}+\varepsilon^2c_{\dot X_t(A)}\right) dt.
\end{equation}
\end{theorem}

\begin{remark}
The estimates (\ref{eq:crit:secder3}) and (\ref{eq:crit:secder4}) hold for any $\varepsilon$ and this will play a fundamental role in the next section where we will have a sequence of perturbed Yang-Mills connections in $\mathrm{Crit}_{\mathcal {YM}^{\varepsilon_i,H}}^\infty$ with $\varepsilon_i\to \infty$.
\end{remark}

\begin{theorem}\label{thm:norminty}
We choose a constant $b>0$. Then there are $\varepsilon_0,\,c>0$ such that for every positive $\varepsilon<\varepsilon_0$ the following holds. If $\Xi^{\varepsilon}:= A^{\varepsilon}+\Psi^{\varepsilon}dt \in \mathrm{Crit}_{\mathcal {YM}^{\varepsilon,H}}^b$ is a perturbed Yang-Mills connection, then
\begin{equation}\label{jhapaa}
\left\| \partial_tA^{\varepsilon}-d_{A^{\varepsilon}}\Psi^{\varepsilon}\right\|_{L^\infty(\Sigma)}\leq c.
\end{equation}
\end{theorem}
First we prove the next theorem.
\begin{theorem}\label{crit:supleqd}
 We choose $\delta$, $b>0$, then there is a positive constant $\varepsilon_0$ such that the following holds. For any perturbed Yang-Mills connection $A+\Psi dt\in \mathrm{Crit}_{\mathcal {YM}^{\varepsilon,H}}^b$, with $0<\varepsilon<\varepsilon_0$, 
$$\sup_{t\in S^1}\left\|F_A\right\|_{L^\infty(\Sigma)}\leq \delta.$$
\end{theorem}

\begin{proof}The theorem follows from the perturbed Yang-Mills equation and the Sobolev theorem \ref{lemma:sobolev} in the following way. If we derive the identity
$$\frac 1{\varepsilon^2}d_A^*F_A-\nabla_tB_t-*X_t(A)=0$$
by $d_A$ and $\nabla_t$ we obtain
\begin{align*}
0=& \frac 1{\varepsilon^2}d_Ad_A^*F_A-d_A\nabla_tB_t-d_A*X_t(A)\\
=& \frac 1{\varepsilon^2}d_Ad_A^*F_A-\nabla_t\nabla_tF_A+[B_t\wedge B_t]-d_A*X_t(A)\\
0=& \frac 1{\varepsilon^2}\nabla_td_A^*F_A-\nabla_t\nabla_tB_t-\nabla_t*X_t(A)\\
=& \frac 1{\varepsilon^2}d_A^*d_AB_t-\nabla_t\nabla_tB_t-\frac 1{\varepsilon^2}*[B_t,*F_A]-\nabla_t*X_t(A)
\end{align*}
and the $L^2$-norm of the Laplace part of the last two identities is 
\begin{align*}
\varepsilon^4\left\|\frac 1{\varepsilon^2}d_Ad_A^*F_A-\nabla_t\nabla_tF_A\right\|_{L^2}^2= &\left\|d_Ad_A^*F_A\right\|_{L^2}^2+\varepsilon^4\left\|\nabla_t\nabla_tF_A\right\|_{L^2}^2\\
&+\varepsilon^2\left\|\nabla_td_A^*F_A\right\|_{L^2}^2
+{\varepsilon^2}\langle [B_t\wedge d_A^*F_A],\nabla_t F_A\rangle\\
&+\varepsilon^2\langle\nabla_t d_A^*F_A,*[B_t,*F_A]\rangle,
\end{align*}
\begin{align*}
\varepsilon^4\left\|\frac 1{\varepsilon^2}d_A^*d_AB_t-\nabla_t\nabla_tB_t\right\|_{L^2}^2= &\left\|d_A^*d_AB_t\right\|_{L^2}^2+\varepsilon^4\left\|\nabla_t\nabla_tB_t\right\|_{L^2}^2\\
&+\varepsilon^2\left\|\nabla_td_AB_t\right\|_{L^2}^2
+{\varepsilon^2}\langle -*[B_t,* d_AB_t],\nabla_t B_t\rangle\\
&-\varepsilon^2\langle\nabla_t d_AB_t,[B_t,B_t]\rangle.
\end{align*}
Therefore we can estimate the $\|\cdot\|_{2,2,\varepsilon}$-norm of $F_A$ and of $B_t$ using the H\"older inequality and the Sobolev theorem \ref{lemma:sobolev}:
\begin{align*}
\frac 12\|F_A\|_{2,2,\varepsilon}^2\leq& \varepsilon^4\left\|\frac 1{\varepsilon^2}d_Ad_A^*F_A-\nabla_t\nabla_tF_A\right\|_{L^2}^2 +\|F_A\|^2_{L^2}\\
&+c\varepsilon \|B_t\|_{L^2}\|d_A^*F_A\|_{L^4}\|\nabla_tF_A\wedge dt\|_{0,4,\varepsilon}\\
&+\delta \varepsilon^2\left\|\nabla_td_A^*F_A\right\|_{L^2}^2
+c\varepsilon^2 \|B_t\|_{L^2}^2\|F_A\|_{L^\infty}^2\\
\leq& \varepsilon^4\left\|[B_t\wedge B_t]-d_A*X_t(A)\right\|_{L^2}^2 +\|F_A\|_{L^2}^2\\
&+c\varepsilon^{\frac 12}\|F_A\|_{2,2,\varepsilon}^2+\delta \varepsilon^2\left\|\nabla_td_A^*F_A\right\|_{L^2}^2\\
\leq& c\varepsilon^3\|B_t\|_{L^2} \|B_t\|_{2,2,\varepsilon}+\varepsilon^4\|d_A*X_t(A)\|_{L^2}^2 +\|F_A\|_{L^2}^2\\
&+c\varepsilon^{\frac 12}\|F_A\|_{2,2,\varepsilon}^2+\delta \varepsilon^2\left\|\nabla_td_A^*F_A\right\|_{L^2}^2\\
\leq& c\varepsilon^3 \|B_t\|_{2,2,\varepsilon}+\|F_A\|_{L^2}^2
+c\varepsilon^{\frac 12}\|F_A\|_{2,2,\varepsilon}^2+\delta \varepsilon^2\left\|\nabla_td_A^*F_A\right\|_{L^2}^2,
\end{align*}

\begin{align*}
\frac 12\|B_t\|_{2,2,\varepsilon}^2\leq& \varepsilon^4\left\|\frac 1{\varepsilon^2}d_A^*d_AB_t-\nabla_t\nabla_tB_t\right\|_{L^2}^2 +\|B_t\|^2_{L^2}\\
&+c\varepsilon \|B_t\|_{L^2}\|d_AB_t\|_{L^4}\|\nabla_tB_t\wedge dt\|_{0,4,\varepsilon}\\
&+\delta \varepsilon^2\left\|\nabla_td_AB_t\right\|_{L^2}^2
+c\varepsilon^2 \|B_t\|_{L^2}^2\|B_t\|_{L^\infty}^2\\
\leq& \varepsilon^4\left\|\frac 1{\varepsilon^2}*[B_t,*F_A]+\nabla_t*X_t(A)\right\|_{L^2}^2 +\|B_t\|_{L^2}^2\\
&+c\varepsilon^{\frac 12}\|B_t\|_{2,2,\varepsilon}^2+\delta \varepsilon^2\left\|\nabla_td_AB_t\right\|_{L^2}^2\\
\leq&c\|F_A\|_{L^2}^2\|B_t\|_{L^\infty}^2+c\varepsilon^4\|B_t\|_{L^2}^2+c\varepsilon^4   +\|B_t\|_{L^2}^2\\
&+c\varepsilon^{\frac 12}\|B_t\|_{2,2,\varepsilon}^2+\delta \varepsilon^2\left\|\nabla_td_AB_t\right\|_{L^2}^2\\
\leq&2\|B_t\|_{L^2}^2+c\varepsilon^4+c\varepsilon^{\frac 12}\|B_t\|_{2,2,\varepsilon}^2+\delta \varepsilon^2\left\|\nabla_td_AB_t\right\|_{L^2}^2.
\end{align*}
Hence we can conclude that
$$\|B_t\|_{2,2,\varepsilon}^2\leq 4\|B_t\|_{L^2}^2+c\varepsilon^4\leq c, $$
$$\|F_A\|_{2,2,\varepsilon}^2\leq c\|F_A\|_{L^2}^2+ \varepsilon^3\|B_t\|^2_{L^2}+c\varepsilon^7\leq c\varepsilon^2$$
and thus, by the Sobolev theorem \ref{lemma:sobolev}, $\|F_A\|_{L^\infty}\leq  c\varepsilon^{-\frac 12}\|F_A\|_{2,2,\varepsilon}^2\leq c\varepsilon^{\frac 12}$.
\end{proof}

In order to prove the theorem \ref{lemma:secder} we need the following lemma.

\begin{lemma}\label{lemma:laplsal}
We choose $R,r>0$, $u:B_{R+r}\subset\mathbb R\to \mathbb R$ a $C^2$ function, $f,g:B_{R+r}\subset\mathbb R\to \mathbb R$ such that
$$f\leq g+\partial_t^2u,\,u\geq0,\,f\geq0,\,g\geq 0,$$
then
\begin{equation}
\int_{B_R}f\,dt\leq \int_{B_{R+r}}g\,dt+\frac4{r^2}\int_{B_{R+r}\backslash B_R }u\,dt.
\end{equation}
Furthermore, if $g=c u$ for a positive constant $c$, then
\begin{equation}
\frac 12\sup_{B_R} u\leq \left(c+\frac 4 r\right)\int_{B_{R+r}}u\, dt.
\end{equation}

\end{lemma}

\begin{proof}
For $B_{R}\subset \mathbb R^2$ and the Laplace operator instead of $\partial_t^2$ the first estimate was proved by Gaio and Salamon in \cite{MR2198773} and the second one by Dostoglou and Salamon in the lemma 7.3 of \cite{MR1283871}. These two proofs apply also for our case.
\end{proof}

\begin{proof}[Proof of theorem \ref{lemma:secder}]
In this proof we write $B_t$ instead of $\partial_tA -d_A\Psi$ and we denote by $\|\cdot \|$ and by $\langle\cdot,\cdot\rangle$ respectively the $L^2$-norm and the $L^2$-product on $\Sigma$. In order to prove the theorem  \ref{lemma:secder} we will apply the last lemma where we choose $u$ to be the $L^2$-norms on $\Sigma$ of $F_A$, $\nabla_tF_A$, $d_A^*F_A$ and $\nabla_t\nabla_tF_A$; since the perturbed Yang-Mills are smooth provided that we choose $\varepsilon$ sufficiently small, as we discussed in the section \ref{chapter:YM}, the regularity assumption of the lemma \ref{lemma:laplsal} is satisfied. In addition we recall that the Bianchi identity tell us that 
\begin{equation}\label{bianchiidentity}
d_AB_t=\nabla_tF_A
\end{equation}
and by the assumptions of the theorem
\begin{equation}\label{eq:thm:normintyy}
\int_0^1\left(\frac 1{\varepsilon^2}\|F_A\|^2+\|B_t\|^2\right) dt \leq b,\quad \sup_{t\in S^1}\|B_t\|_{L^4(\Sigma)}\leq c_1.
\end{equation}
Furthermore by the theorem \ref{crit:supleqd} we can assume that $\sup_{t\in S^1}\|F_A\|^2\leq \delta$ where $\delta$ satisfies the assumptions of the lemmas \ref{lemma76dt94} and \ref{lemma82dt94} for $p=2$, which allows us to estimate any 2-form in the following way
\begin{equation}\label{eq:estlem}
\|\beta \|_{L^q(\Sigma)}\leq c\|d_A^*\beta \|, \forall \beta\in \Omega^{2}(\Sigma,\mathfrak g_P), \quad 2\leq q<\infty.
\end{equation}

{\bf Step 1.} We prove the estimate (\ref{eq:lemma:secder1}).

\begin{proof}[Proof of step 1]If we derive $\|F_A\|^2$ we obtain
\begin{equation}\label{crit:ap:eq22}
\begin{split}
\partial_t^2\|F_{A}\|^2=&2\|\nabla_tF_{A}\|^2+2\langle\nabla_t\nabla_tF_A,F_A\rangle= 2\|\nabla_tF_{A}\|^2+2\langle\nabla_td_AB_t,F_A\rangle\\
=& 2\|\nabla_tF_{A}\|^2+2\langle d_A\nabla_tB_t,F_A\rangle+2\langle[B_t\wedge B_t],F_A\rangle\\
=& 2\|\nabla_tF_{A}\|^2+2\langle \nabla_tB_t,d_A^*F_A\rangle+2\langle[B_t\wedge B_t],F_A\rangle\\
=& 2\|\nabla_tF_{A}\|^2+\frac2{\varepsilon^2} \|d_A^*F_A\|^2-2\langle *X_t(A),d_A^*F_A\rangle+2\langle[B_t\wedge B_t],F_A\rangle\\
\geq& 2\|\nabla_tF_{A}\|^2+\frac2{\varepsilon^2} \|d_A^*F_A\|^2-2|\langle *X_t(A),d_A^*F_A\rangle|-\|B_t\|_{L^4(\Sigma)}^2\|F_A\|\\
\geq& 2 \|\nabla_tF_{A}\|^2+\frac2{\varepsilon^2} \|d_A^*F_A\|^2-c\|F_A\|-c\|d_A^*F_A\|
\end{split}
\end{equation}
where the second equality follows from the Bianchi identity (\ref{bianchiidentity}), the third from the commutation formula (\ref{commform}), the fifth from the perturbed Yang-Mills equation (\ref{epsiloneq1}) and the last one from (\ref{eq:thm:normintyy}). Thus, (\ref{eq:estlem}) and (\ref{crit:ap:eq22}) imply that
\begin{equation}
\|F_A\|^2\leq c \|d_A^*F_A\|^2+c\varepsilon^2\|\nabla_tF_A\|^2\leq c\partial_t^2(\varepsilon^2\|F_A\|^2)+\frac{c\varepsilon^4}{\delta_0}+c\delta_0\|F_A\|^2
\end{equation}
and hence for $\delta_0$ sufficiently small
\begin{equation}\label{in:FA}
\|F_A\|^2+ \|d_A^*F_A\|^2+\varepsilon^2\|\nabla_tF_A\|^2\leq c\partial_t^2(\varepsilon^2\|F_A\|^2)+{c\varepsilon^4};
\end{equation}
applying the lemma \ref{lemma:laplsal} for (\ref{in:FA})
\begin{equation}\label{crit:ap:eqkkk}
 \int_0^1\left( \|F_A\|^2 +\varepsilon^2\|\nabla_tF_{A}\|^2+ \|d_A^*F_A\|^2\right)dt \leq c\varepsilon^4+c\varepsilon^2\int_0^1\|F_A\|^2 dt\leq c\varepsilon^4
\end{equation}
by (\ref{eq:thm:normintyy}).
%
%
Analogously to (\ref{crit:ap:eq22}) one can show that
\begin{equation}\label{laplacedFA}
\begin{split}
\partial_t^2&\big(\varepsilon^4\|\nabla_tF_A\|^2+\varepsilon^2\|d_A^*F_A\|^2\big)\\
\geq&
\varepsilon^4\|\nabla_t\nabla_tF_A\|^2+{\varepsilon^2}\|d_A^*\nabla_tF_A\|^2+\|d_Ad_A^*F_A\|^2-c\varepsilon^4,
\end{split}
\end{equation}
\begin{equation}\label{laplaceddFA}
\begin{split}
\partial_t^2&\left(\varepsilon^6 \| \nabla_t  \nabla_tF_A\|^2+\varepsilon^4\|\nabla_tF_A\|^2+\varepsilon^2\|d_A^*F_A\|^2\right)\\
\geq&
\varepsilon^6 \|\nabla_t\nabla_t\nabla_tF_A\|^2+{\varepsilon^4}\|d_A^*\nabla_t\nabla_tF_A\|^2\\
&+\varepsilon^4\|\nabla_t\nabla_tF_A\|^2+{\varepsilon^2}\|d_A^*\nabla_t F_A\|^2+\|d_Ad_A^*F_A\|^2-c\varepsilon^4,
\end{split}
\end{equation}
Hence by the lemma \ref{lemma:laplsal}
\begin{equation}\label{eq:nonsoso}
\begin{split}
\int_0^1&\left(\varepsilon^4\|\nabla_t\nabla_tF_A\|^2+{\varepsilon^2}\|\nabla_td_A^*F_A\|^2+\|d_Ad_A^*F_A\|\right)\, dt\\
\leq& c\int_0^1\left(\varepsilon^2 \|\nabla_tF_A\|^2+\varepsilon^2\|d_A^*F_A\|^2+\varepsilon^2\|F_A\|^2+c\varepsilon^4\right)dt\leq c\varepsilon^4,
\end{split}
\end{equation}
\begin{equation}\label{eq:nonsososo}
\int_0^1\left(\varepsilon^6 \|\nabla_t\nabla_t\nabla_tF_A\|^2+{\varepsilon^4}\|\nabla_t\nabla_td_A^*F_A\|^2\right)dt\leq c\varepsilon^4.
\end{equation}
and thus, $\|F_A\|_{3,2,\varepsilon}\leq c\varepsilon^{2}$ by (\ref{crit:ap:eqkkk}), (\ref{eq:nonsoso}) and (\ref{eq:nonsososo}) and therefore we proved (\ref{eq:lemma:secder1}).
\end{proof}

{\bf Step 2.} $\int_0^1\left(\|F_{A}\|_{L^2(\Sigma) }^{2p}+\varepsilon^{2p} \|\nabla_tF_{A}\|_{L^2(\Sigma) }^{2p}+\varepsilon^{4p} \|\nabla_t\nabla_tF_{A}\|_{L^2(\Sigma) }^{2p}\right)dt\leq c \varepsilon^{4p}.$

\begin{proof}[Proof of step 2]

Using the estimates (\ref{laplacedFA}), (\ref{laplaceddFA}) combined with the lemma \ref{lemma76dt94} we obtain
\begin{equation*}
\begin{split}
&\left(\varepsilon^4 \| \nabla_t  \nabla_tF_A\|^2+\varepsilon^2\|\nabla_tF_A\|^2+\|d_A^*F_A\|^2\right)\\
\leq& c\varepsilon^4+c\varepsilon^2\partial_t^2\left(\varepsilon^4 \| \nabla_t  \nabla_tF_A\|^2+\varepsilon^2\|\nabla_tF_A\|^2+\|d_A^*F_A\|^2\right)
\end{split}
\end{equation*}
%
%
and since for $f(t)=\left(\varepsilon^4 \| \nabla_t  \nabla_tF_A\|^2+\varepsilon^2\|\nabla_tF_A\|^2+\|d_A^*F_A\|^2\right)$
\begin{equation*}
\begin{split}
\partial_t^2f(t)^p=&\frac p2 f(t)^{p-1}\partial_t^2f(t)+\frac {p(p-1)}4f(t)^{p-2}(\partial_tf(t))^2
 \geq \frac p2 f(t)^{p-1}\partial_t^2f(t)^2,
\end{split}
\end{equation*}
we have
\begin{equation*}
\begin{split}
&\left(\varepsilon^4 \| \nabla_t  \nabla_tF_A\|^2+\varepsilon^2\|\nabla_tF_A\|^2+\|d_A^*F_A\|^2\right)^p\\
\leq& c\varepsilon^4\left(\varepsilon^4 \| \nabla_t  \nabla_tF_A\|^2+\varepsilon^2\|\nabla_tF_A\|^2+\varepsilon^2\|d_A^*F_A\|^2\right)^{p-1}\\
&+c\varepsilon^2\partial_t^2\left(\varepsilon^4 \| \nabla_t  \nabla_tF_A\|^2+\varepsilon^2\|\nabla_tF_A\|^2+\|d_A^*F_A\|^2\right)^p.
\end{split}
\end{equation*}
%
Then, we apply the inequality $ab\leq \frac{a^p}p+\frac {b^q}q$ with $q=\frac p{p-1}$ for the first term on the right side of the inequality for $a=c\varepsilon^4$ and $b=f(t)^{p-1}$ and hence
\begin{equation}\label{crit:ap:dj}
\begin{split}
&\frac 1p\left(\varepsilon^4 \| \nabla_t  \nabla_tF_A\|^2+\varepsilon^2\|\nabla_tF_A\|^2+\|d_A^*F_A\|^2\right)^p\\
\leq &c\varepsilon^{4p}+c\varepsilon^2\partial_t^2\left(\varepsilon^4 \| \nabla_t  \nabla_tF_A\|^2+\varepsilon^2\|\nabla_tF_A\|^2+\|d_A^*F_A\|^2\right)^p.
\end{split}
\end{equation}
%
Finally using the previous lemma \ref{lemma:laplsal} 
\begin{equation*}
\begin{split}
\int_0^1&\left(\|d_A^*F_{A}\|_{L^2(\Sigma) }^{2p}+\varepsilon^{2p} \|\nabla_tF_{A}\|_{L^2(\Sigma) }^{2p}+\varepsilon^{4p} \|\nabla_t\nabla_tF_{A}\|_{L^2(\Sigma) }^{2p}\right)dt\\
\leq& c_2 \varepsilon^{4p}+\varepsilon^2\int_0^1 \left(\|d_A^*F_{A}\|_{L^2(\Sigma) }^{2p}+\varepsilon^{2p} \|\nabla_tF_{A}\|_{L^2(\Sigma) }^{2p}+\varepsilon^{4p} \|\nabla_t\nabla_tF_{A}\|_{L^2(\Sigma) }^{2p}\right)dt
\end{split}
\end{equation*}
and hence we conclude the proof of the third step choosing $\varepsilon$ sufficiently small.
\end{proof}

{\bf Step 3.} For any $p\geq2$, the estimate (\ref{eq:lemma:secder3}) holds.

\begin{proof}[Proof of step 3]
The estimate (\ref{crit:ap:dj}) yields to

\begin{equation*}
\begin{split}
0\leq& c\varepsilon^2\left(\varepsilon^4 \| \nabla_t  \nabla_tF_A\|^2+\varepsilon^2\|\nabla_tF_A\|^2+\|d_A^*F_A\|^2+\varepsilon^{4-\frac 2p}\right)^p\\
&+\varepsilon^2\partial_t^2\left(\varepsilon^4 \| \nabla_t  \nabla_tF_A\|^2+\varepsilon^2\|\nabla_tF_A\|^2+\|d_A^*F_A\|^2+\varepsilon^{4-\frac 2p}\right)^p
\end{split}
\end{equation*}
and thus by the lemma \ref{lemma:laplsal}
\begin{align*}
\sup_{t\in S^1}&\left(\varepsilon^2\|d_A^*F_A\|^{2p}+\varepsilon^{2+2p}\|\nabla F_A\|^{2p}+\varepsilon^{2+4p}\|\nabla_t\nabla_t F_A\|^{2p}\right)\\
\leq& c\varepsilon^{4p}+\varepsilon^2\int_0^1\left(\|d_A^*F_A\|^{2p}+\varepsilon^{2p}\|\nabla F_A\|^{2p}+\varepsilon^{4p}\|\nabla_t\nabla_t F_A\|^{2p}\right)dt\leq c \varepsilon^{4p}.
\end{align*}
By the perturbed Yang-Mills equation we can also estimate $\|d_Ad_A^*F_A\|$ in the following way:
\begin{equation*}
\begin{split}
 \|d_Ad_A^*F_A\|\leq& \varepsilon^2\|d_A\nabla_tB_t\|+c\varepsilon^2\\
\leq & \varepsilon^2 \|\nabla_t d_A B_t\|+\varepsilon^2\|[B_t\wedge B_t]\|+c\varepsilon^2\\
\leq & \varepsilon^2\|\nabla_t\nabla_tF_A\|+4\varepsilon^2\|B_t\|_{L^4(\Sigma)}^4+c\varepsilon^2
\end{split}
\end{equation*}
where the second inequality follows from (\ref{eq:estlem}) and the commutation formula (\ref{commform}) and the third from the Bianchi identity (\ref{bianchiidentity}) and the H\"older inequality. By the last two estimates and by the lemma \ref{lemma76dt94} we can conclude that
\begin{align*}
\sup_{t\in S^1}\big(&\|F_A\|+\|F_A\|_{L^\infty(\Sigma)}+\|d_A^*F_A\|\\
&+\|d_Ad_A^*F_A\|+\varepsilon\|\nabla_t F_A\|+\varepsilon^{2}\|\nabla_t\nabla_t F_A\|\big)
\leq c\varepsilon^{2-\frac 1p}. 
\end{align*}
\end{proof}
With the fourth step we finished also the proof of the theorem \ref{lemma:secder}.
\end{proof}

\begin{proof}[Proof of theorem \ref{crit:secder}]
During this proof we denote by $\|\cdot\|$ and by $\langle\cdot,\cdot\rangle$ respectively the $L^2$-norm and the inner product over $\Sigma$. We choose $\delta$ small enough to apply the lemma \ref{lemma76dt94} and hence $\|F_A\|\leq c\|d_A^*F_A\|$ holds for a constant $c$.\\

{\bf Step 1.} There is a constant $c>0$ such that 
\begin{equation*}
\sup_{t\in K}\varepsilon^2 \|B_t \|^2
\leq c\int_{\Omega } \left(\varepsilon^2 \|B_t \|^2 +\|F_{ A}\|^2+\varepsilon^2c_{ \dot X_t}\right)dt.
\end{equation*}
%
\begin{equation*}
\int_{K}\|d_AB_t\|^2 dt\leq c \int_{ \Omega} \left(\|F_A\|^2+ \frac 1{\varepsilon^2}\|F_A\|^2+\varepsilon^2\|B_t\|^2+\varepsilon^2 c_{\dot X_t(A)}\right) dt.
\end{equation*}

%
\begin{proof}[Proof of step 1.]
In order to prove the first step we compute $\partial_t^2\|B_t \|^2$ and then we apply the lemma \ref{lemma:laplsal}. By the perturbed Yang-Mills equation (\ref{epsiloneq1}), we have tha
\begin{align*}
\frac 12 \partial_t^2\|B_t \|^2=& \|\nabla_tB_t\|^2+\langle \nabla_t \nabla_t B_t,B_t\rangle\\
=&\|\nabla_tB_t\|^2+\frac 1{\varepsilon^2 }\langle \nabla_t d_{ A}^*F_{ A} ,B_t\rangle-\langle\nabla_t *X_{t}( A),B_t\rangle\\
=&\|\nabla_tB_t\|^2+\frac 1{\varepsilon^2}\langle d_{ A}^*\nabla_t F_{ A} ,B_t\rangle
+\frac 1{\varepsilon^2 }\langle *[B_t,*F_{ A}] ,B_t\rangle\\
&-\langle d*X_{t}( A)B_t+\dot X_t( A),B_t\rangle\\
=&\|\nabla_tB_t\|^2+\frac 1{\varepsilon^2 }\| d_{ A}B_t\|^2
+\frac 1{\varepsilon^2 }\langle *[B_t,*F_{ A}] ,B_t\rangle\\
&-\langle d*X_{t}( A)B_t+\dot X_{t}( A),B_t\rangle.
\end{align*}
where third step follows from the commutation formula (\ref{commform}) and the fourth from the Bianchi identity (\ref{bianchiidentity}). Thus, using the H\"older, the Cauchy-Schwarz inequality and the Sobolev estimate $\|B_t\|_{L^4(\Sigma)}\leq c\left(\|B_t\|+\|d_AB_t\|\right)$, one can easily see that
\begin{equation}\label{crit:est1:apriori}
\begin{split}
\partial_t^2 \|B_t\|^2
\geq &\|\nabla_tB_t\|^2+\frac 1{\varepsilon^2}\|d_AB_t\|^2- \frac c{\varepsilon^2}\|B_t\|_{L^4}(\|B_t\|+\|d_AB_t\|)\|F_A\|\\
&-c\|B_t\|^2-c\|\dot X_t(A)\|\cdot \|B_t\|\\
\geq &\|\nabla_tB_t\|^2+\frac 1{\varepsilon^2}\|d_AB_t\|^2- \frac c{\varepsilon^4}\|F_A\|^2\\
&- \frac c{\varepsilon^2}\|F_A\|^2-c\|B_t\|^2-c\|\dot X_t(A)\|^2.
\end{split}
\end{equation}
Hence using the lemma \ref{lemma:laplsal} we can conclude the second estimate of the first step:
\begin{equation*}
\begin{split}
\int_{S^1}&\left( \varepsilon^2\|\nabla_tB_t\|^2+\|d_AB_t\|^2\right) dt\\
\leq& c \int_{S^1} \left(\frac 1{\varepsilon^2}\|F_A\|^2+\varepsilon^2\|B_t\|^2+\varepsilon^2c_{\dot X_t(A)}+c\|F_A\|^2\right) dt.
\end{split}
\end{equation*}

Since $\|F_{ A }\|_{L^2(\Sigma)}\leq \delta$ and $\|F_A\|\leq c\|d_A^*F_A\|$, by the theorems \ref{lemma76dt94} and \ref{lemma82dt94} there is  a $ A_1 \in \mathcal A^0(P)$ such that $\| A- A_1 \|_{L^2 }\leq c \|F_{ A}\|_{L^2}$ and thus we can write
\begin{equation}
d_{ A}*X_{t}( A)
= d_{ A_1 }*X_{t}( A_1)+\left[(A-A_1)\wedge * X_{t}( A_1)\right]
\end{equation}
where $d_{ A_1 }*X_{t}( A_1)=0$. Therefore, by the fifth line of the computation (\ref{crit:ap:eq22})
\begin{equation}\label{crit:ineq:apriori:11}
\frac 12\partial_t^2\| F_A\|^2\geq \frac 1{4\varepsilon^2}\| d_A^*F_A\|^2+\frac 14\|\nabla_tF_A\|^2-c\varepsilon^2\|B_t\|^2-c\|F_A\|^2
\end{equation}
and with (\ref{crit:est1:apriori}) it follows that for a constant $c_0$ big enough
\begin{equation*}
\frac 12 \partial_t^2\left( c_0\|F_{ A}\|^2+\varepsilon^2\|B_t \|^2+\varepsilon^2c_{\dot X_t} \right)\geq -c \left(  c_0\|F_{ A}\|^2+ \varepsilon^2\|B_t \|^2+\varepsilon^2c_{\dot X_t} \right).
\end{equation*}
Finally by lemma \ref{lemma:laplsal}, we can conclude that 
\begin{equation*}
\sup_{t\in K} \varepsilon^2\|B_t \|^2
\leq c \int_{\Omega } \left(\varepsilon^2 \|B_t \|^2+ \|F_{ A}\|^2+\varepsilon^2c_{\dot X_t}\right)dt.
\end{equation*}
\end{proof}
%
%

%
%

{\bf Step 2.} There is a positive constant $c>0$ such that
\begin{equation*}
\begin{split}
\sup_{t\in K}&\|d_A^*d_AB_t\|^2
\leq c \int_{\Omega}\left(\|F_A\|^2+\varepsilon^2\|B_t\|^2+\varepsilon^2c_{\dot X_t(A)}+\|d_A^*d_AB_t\|^2\right) dt.
\end{split}
\end{equation*}

\begin{proof}[Proof of step 2] Analogously to the previous steps we need to compute $\frac 12 \partial_t^2\|d_A^*d_AB_t\|^2$:
\begin{align*}
\frac 12 \partial_t^2&||d_A^*d_AB_t||^2= ||\nabla_td_A^*d_AB_t||^2+\langle \nabla_t\nabla_td_A^*d_AB_t,d_A^*d_AB_t\rangle\\
\intertext{by the commutation formula (\ref{commform}) and the Yang-Mills equation (\ref{epsiloneq1}) we have}
=& ||\nabla_td_A^*d_AB_t||^2+\frac 1{\varepsilon^2}\langle \nabla_td_A^*d_Ad_A^*F_A,d_A^*d_AB_t\rangle\\
&-\langle \nabla_td_A^*d_A*X_t(A),d_A^*d_AB_t\rangle\\
&+\langle \nabla_t\left(-*[B_t\wedge,*d_AB_t]+d_A^*[B_t\wedge B_t]\right),d_A^*d_AB_t\rangle\\
\intertext{and applying one more time t(\ref{commform})}
=& ||\nabla_td_A^*d_AB_t||^2+\frac 1{\varepsilon^2}\langle d_A^*d_Ad_A^* \nabla_tF_A,d_A^*d_AB_t\rangle\\
&+\frac 1{\varepsilon^2}\langle-* [B_t\wedge *d_A d_A^* F_A]+ d_A^*[B_t\wedge d_A^* F_A],d_A^*d_AB_t\rangle\\
&-\frac 1{\varepsilon^2}\langle d_A*[B_t,*F_A],d_A d_A^*d_AB_t\rangle
-\langle d_A*\nabla_tX_t(A),d_Ad_A^*d_AB_t\rangle\\
&-\langle d_A^* [B_t\wedge *X_t(A)]-*[B_t\wedge *d_A*X_t(A)],d_A^*d_AB_t\rangle \\
&+\langle \nabla_t\left(-*[B_t\wedge,*d_AB_t]+d_A^*[B_t\wedge B_t]\right),d_A^*d_AB_t\rangle\\
\intertext{finally, by the Bianchi identity (\ref{bianchiidentity}) and the perturbed Yang-Mills equation (\ref{epsiloneq1}) we can conclude that}
=& ||\nabla_td_A^*d_AB_t||^2+\frac 1{\varepsilon^2}||d_Ad_A^* d_AB_t||^2\\
&-\langle  [*B_t\wedge *d_A(\nabla_tB_t+*X_t(A))],d_A^*d_AB_t\rangle
+\frac 1{\varepsilon^2}\langle d_A^*[B_t\wedge d_A^* F_A],d_A^*d_AB_t\rangle\\
&-\frac 1{\varepsilon^2}\langle d_A*[B_t,*F_A],d_A d_A^*d_AB_t\rangle
-\langle d_A*\nabla_tX_t(A),d_Ad_A^*d_AB_t\rangle\\
&-\langle d_A^* [B_t\wedge *X_t(A)]-*[B_t\wedge *d_A*X_t(A)],d_A^*d_AB_t\rangle \\
&+\langle \nabla_t\left(-*[B_t\wedge,*d_AB_t]+d_A^*[B_t\wedge B_t]\right),d_A^*d_AB_t\rangle;
\end{align*}
The last computation implies
\begin{equation*}\label{crit:apr:9}
\partial_t^2\|d_A^*d_AB_t\|^2
\geq -c\varepsilon^2\|B_t\|^2-c\frac 1{\varepsilon^2}\|d_A^*F_A\|^2-c \varepsilon^2\|\nabla_tB_t\|^2-c\|d_A^*d_AB_t\|^2.
\end{equation*}
Therefore combining (\ref{crit:est1:apriori}), (\ref{crit:ineq:apriori:11}) and (\ref{crit:apr:9})

\begin{equation*}
\begin{split}
\partial_t^2
\left(\|d_A^*d_AB_t\|^2+{c_0}\|F_A\|^2+c_0\varepsilon^2\|B_t\|\right)
\geq& -c \varepsilon^2\|B_t\|^2-c\|F_A\|^2\\
&-c\|d_A^*d_AB_t\|^2-\varepsilon^2 c_{\dot X_t}
\end{split}
\end{equation*}
and hence we conclude by the lemma \ref{lemma:laplsal} that
\begin{equation*}
\begin{split}
\sup_{t\in K}\|d_A^*d_AB_t\|^2 \leq &c \int_{\Omega}\left(\|F_A\|^2+\varepsilon^2\|B_t\|^2+\varepsilon^2c_{\dot X_t(A)}+\|d_A^*d_AB_t\|^2\right) dt.
\end{split}
\end{equation*}
\end{proof}


{\bf Step 3.} There is a constant $ c>0$ such that 
\begin{equation*}
 \sup _{t\in K}\|d_AB_t\|\leq c\int_{\Omega}\left(\|d_AB_t\|^2+\frac 1{\varepsilon^2}\|F_A\|^2+\|B_t\|^2\right) dt
\end{equation*}
and if $0<\varepsilon<c_2$, then
\begin{equation*}
\int_{S^1} \|d_A^*d_AB_t\|^2 dt
\leq  c\varepsilon^2\int_{S^1}\left( \|F_A\|^2+\varepsilon^2\|B_t\|^2+\varepsilon^2c_{\dot X_t(A)}\right) dt.
\end{equation*}

\begin{proof}[Proof of step 3] Like in the previous steps we will prove this one using the lemma \ref{lemma:laplsal} and therefore we need to compute $\frac 12 \partial_t^2||d_AB_t||^2$. We consider
\begin{align*}
\frac 12 \partial_t^2||d_AB_t||^2=& ||\nabla_td_AB_t||^2+\langle \nabla_t\nabla_td_AB_t,d_AB_t\rangle\\
\intertext{using the commutation formula (\ref{commform}) and the 
Yang-Mills flow equation (\ref{epsiloneq1}), we have}
=& ||\nabla_td_AB_t||^2+\frac 1{\varepsilon^2}\langle \nabla_td_Ad_A^*F_A,d_AB_t\rangle\\
&-\langle \nabla_td_A*X_t(A),d_AB_t\rangle+\langle \nabla_t[B_t\wedge B_t],d_AB_t\rangle\\
\intertext{by the commutation formula (\ref{commform}) }
=& ||\nabla_td_AB_t||^2+\frac 1{\varepsilon^2}\langle d_Ad_A^* \nabla_tF_A,d_AB_t\rangle\\
&+\frac 1{\varepsilon^2}\langle [B_t\wedge d_A^* F_A],d_AB_t\rangle
-\frac 1{\varepsilon^2}\langle*[B_t,*F_A],d_A^*d_AB_t\rangle\\
&-\langle \nabla_td_A*X_t(A),d_AB_t\rangle+\langle \nabla_t[B_t\wedge B_t],d_AB_t\rangle\\
\intertext{next, the Bianchi identity (\ref{bianchiidentity}) yields to}
=& ||\nabla_td_AB_t||^2+\frac 1{\varepsilon^2}||d_A^*d_AB_t||^2+\frac 1{\varepsilon^2}\langle [B_t\wedge d_A^* F_A],d_AB_t\rangle\\
&-\frac 1{\varepsilon^2}\langle*[B_t,*F_A],d_A^*d_AB_t\rangle
-\langle \nabla_td_A*X_t(A),d_AB_t\rangle\\
&+\langle2[ \nabla_tB_t\wedge B_t],d_AB_t\rangle
\end{align*}
and thus
\begin{equation}\label{crit:eq:oihds1}
\begin{split}
\frac 12 \partial_t^2 \|d_AB_t\|^2\geq &\frac 12 \|\nabla_td_AB_t\|^2+ \frac 1{2\varepsilon^2}\|d_A^*d_AB_t\|^2-\frac c{\varepsilon^2}\|d_A^*F_A\|^2\\
&-c \varepsilon^2 \|B_t\|^2-c\varepsilon^2\|\dot X_t(A)\|_{L^\infty }^2-c \varepsilon^2 \|\nabla_tB_t\|^2.
\end{split}
\end{equation}
and 
\begin{equation}\label{crit:eq:oihds2}
\begin{split}
\frac 12 \partial_t^2 \|d_AB_t\|^2\geq &\frac 12 \|\nabla_td_AB_t\|^2+ \frac 1{2\varepsilon^2}\|d_A^*d_AB_t\|^2-\frac c{\varepsilon^4}\|d_A^*F_A\|^2\\
&-c  \|d_AB_t\|^2-c\|B_t\|^2-\frac c{\varepsilon^2}\|F_A\|^2.
\end{split}
\end{equation}
Therefore, (\ref{crit:eq:oihds2}) combined with (\ref{crit:est1:apriori}) yields to
\begin{equation}
\begin{split}
\partial_t^2&\left(  \|d_AB_t\|^2+c_0\frac 1{\varepsilon^2}\|F_A\|^2+c_0\|B_t\|^2\right)\\
\geq&- c\|B_t\|^2-c\|\dot X_t(A)\|_{L^\infty}-\frac c{\varepsilon^2}\|F_A\|^2-c\|d_AB_t\|^2
\end{split}
\end{equation}
where we use that 
$$\partial_t^2\|F_A\|^2\geq -c\varepsilon^2\|B_t\|^2-c\varepsilon^2\|d_AB_t\|^2$$
by the fifth line of (\ref{crit:ap:eq22}). The lemma \ref{lemma:laplsal} applyed the last estimate give us
\begin{equation*}
 \sup _{t\in K}\|d_AB_t\|\leq c\int_{\Omega}\left(\|d_AB_t\|^2+\frac 1{\varepsilon^2}\|F_A\|^2+\|B_t\|^2\right) dt.
\end{equation*}
The estimate (\ref{crit:eq:oihds1}) combined with (\ref{crit:est1:apriori}) yields to
\begin{equation}
\begin{split}
\partial_t^2&\left(  \|d_AB_t\|^2+c_0\|F_A\|^2+c_0\varepsilon^2\|B_t\|^2\right)\\
\geq&\|\nabla_td_AB_t\|^2+ \frac 1{\varepsilon^2}\|d_A^*d_AB_t\|^2- c\varepsilon^2\|B_t\|^2-c\varepsilon^2\|\dot X_t(A)\|_{L^\infty}-c\|F_A\|^2
\end{split}
\end{equation}
for a constant $c_0$ big enough.
Hence, if $0<\varepsilon<c_2$, by lemma \ref{lemma:laplsal} we have 
\begin{equation}\label{crit:eq:apriori:djdk}
\begin{split}
\int_{S^1}& \left(\varepsilon^2\|\nabla_t d_AB_t\|^2+ \|d_A^*d_AB_t\|^2\right) dt\\
\leq & c\varepsilon^2\int_{S^1}\left( \|d_AB_t\|^2+\|F_A\|^2+\varepsilon^2\|B_t\|^2+\varepsilon^2c_{\dot X_t(A)}\right) dt\\
\leq & c\varepsilon^2\int_{S^1}\left( \|d_A^*d_AB_t\|^2+\|F_A\|^2+\varepsilon^2\|B_t\|^2+\varepsilon^2c_{\dot X_t(A)}\right) dt\\
\leq & c\varepsilon^2\int_{S^1}\left( \varepsilon^2\|d_AB_t\|^2+ \|F_A\|^2+\varepsilon^2\|B_t\|^2+\varepsilon^2c_{\dot X_t(A)}\right) dt\\
\leq & c\varepsilon^2\int_{S^1}\left( \|F_A\|^2+\varepsilon^2\|B_t\|^2+\varepsilon^2c_{\dot X_t(A)}\right) dt
\end{split}
\end{equation}
where the second estimate follows from the lemma \ref{lemma76dt94}, the third inequality follows from the first one and the first step implies the last estimate. 
\end{proof}
The estimate (\ref{ap:crit:se}) follows combining the second and the third step; hence, we finished the proof of the theorem \ref{crit:secder}.
\end{proof}

\begin{proof}[Proof of theorem \ref{thm:norminty}]
If we prove that $\left\| \partial_tA^{\varepsilon}-d_{A^{\varepsilon}}\Psi^{\varepsilon}\right\|_{L^4(\Sigma)}$ is uniformly bounded by a constant, then by the theorem \ref{crit:secder} and the Sobolev estimate it follows that
$$\|\partial_tA^\varepsilon-d_{A^\varepsilon}\Psi^\varepsilon\|_{L^\infty(\Sigma)}\leq \|\partial_tA^\varepsilon-d_{A^\varepsilon}\Psi^\varepsilon\|_{L^4(\Sigma)}+\|d_{A^\varepsilon}^*d_{A^\varepsilon}(\partial_tA^\varepsilon-d_{A^\varepsilon}\Psi^\varepsilon)\|_{L^2(\Sigma)}\leq c$$
and hence (\ref{jhapaa}) is satisfied for $\varepsilon$ sufficiently small. We prove the theorem by an indirect argument assuming that there is a sequence $\{\Xi^{\varepsilon_\nu}=A^{\varepsilon_\nu}+\Psi^{\varepsilon_\nu} dt\}_{\nu\in \mathbb N}$, $\varepsilon_\nu\to 0$, of perturbed Yang-Mills connections that satisfies $\mathcal{YM}^{\varepsilon_\nu,H}(\Xi^{\varepsilon_\nu})\leq b$ and $m_\nu:=\sup_{t\in S^1} \left\| \partial_tA^{\varepsilon_\nu}-d_{A^{\varepsilon_\nu}}\Psi^{\varepsilon_\nu}\right\|_{L^4(\Sigma)}\to \infty$. In addition we assume that there is a convergent sequence $t_\nu\to t^\infty$ in $S^1$ such that
\begin{equation}
\left\| \partial_tA^{\varepsilon_\nu}(t_\nu)-d_{A^{\varepsilon_\nu}(t_\nu)}\Psi^{\varepsilon_\nu}(t_\nu)\right\|_{L^4(\Sigma)}=m_\nu.
\end{equation}
We need to check three cases that depend from the behavior of the sequence $\varepsilon_\nu m_\nu$.\\

{\bf Case 1:} $\lim_{\nu\to \infty}\varepsilon_\nu m_\nu=0$. We define a new sequence of connections $\bar \Xi^{\varepsilon_\nu}:=\bar A^{\varepsilon_\nu}+\bar\Psi^{\varepsilon_\nu} dt$ by $\bar A^{\varepsilon_\nu}(t):=A^{\varepsilon_\nu}(t_\nu+t/m_\nu)$, and $\bar \Psi^{\varepsilon_\nu}(t):=\frac 1{m_\nu}\Psi^{\varepsilon_\nu}(t_\nu+t/m_\nu)$. This sequence satisfies the perturbed Yang-Mills equations
\begin{equation*}
\frac 1{\varepsilon_\nu^2m_\nu^2}d_{\bar A^{\varepsilon_\nu}}^*F_{\bar A^{\varepsilon_\nu}}=\bar \nabla_t\left(\partial_t\bar A^{\varepsilon_\nu}-d_{\bar A^{\varepsilon_\nu}}\bar\Psi^{\varepsilon_\nu}\right)+\frac 1{m_\nu^2}*X_{t_\nu+\frac t{m_\nu}}(\bar A^{\varepsilon_\nu}),
\end{equation*}
\begin{equation*}
 d_{\bar A^{\varepsilon_\nu}}^*\left(\partial_t\bar A^{\varepsilon_\nu}-d_{\bar A^{\varepsilon_\nu}}\bar\Psi^{\varepsilon_\nu}\right)=0.
\end{equation*}
In addition, we have the following estimates for the norms for  $\bar\varepsilon_\nu:=\varepsilon_\nu m_\nu$
\begin{equation}\label{crit:lidn}
\sup_{t\in \left[-\frac{m_\nu}2,\frac {m_\nu}2\right]} \left\| \partial_t\bar A^{\varepsilon_\nu}-d_{\bar A^{\varepsilon_\nu}}\bar \Psi^{\varepsilon_\nu}\right\|_{L^4(\Sigma)}=\left\| \partial_t\bar A^{\varepsilon_\nu}(0)-d_{\bar A^{\varepsilon_\nu}(0)}\bar \Psi^{\varepsilon_\nu}(0)\right\|_{L^4(\Sigma)}= 1,
\end{equation}
\begin{equation}
\begin{split}
\frac 1{\bar\varepsilon_\nu^2}\left\|F_{\bar A^{\varepsilon_\nu}}\right\|_{L^2}^2
=&\int_{-\frac {m_\nu}2}^{\frac {m_\nu}2}\frac1{\bar\varepsilon_\nu^2}\left\|F_{\bar A^{\varepsilon_\nu}}\right\|_{L^2(\Sigma)}^2 d t\\
=&\int_{-\frac 12}^{\frac 12}\frac1{m_\nu\varepsilon_\nu^2}\left\|F_{A^{\varepsilon_\nu}}\right\|_{L^2(\Sigma)}^2 d t\leq \frac {b^2}{m_\nu},\end{split}
\end{equation}
\begin{equation}
\begin{split}
\left\| \partial_t\bar A^{\varepsilon_\nu}-d_{\bar A^{\varepsilon_\nu}}\bar\Psi^{\varepsilon_\nu}\right\|_{L^2}^2
=&\int_{-\frac{m_\nu}2}^{-\frac {m_\nu}2}\left\| \partial_t\bar A^{\varepsilon_\nu}-d_{\bar A^{\varepsilon_\nu}}\bar\Psi^{\varepsilon_\nu}\right\|_{L^2(\Sigma)}^2 d t \\
=&\int_{-\frac{1}2}^{-\frac {1}2}\frac 1{m_\nu^2}\left\| \partial_tA^{\varepsilon_\nu}-d_{ A^{\varepsilon_\nu}}\Psi^{\varepsilon_\nu}\right\|_{L^2(\Sigma)}^2 m_\nu d t\leq \frac {b^2}{m_\nu}.
\end{split}
\end{equation}
We denote $\partial_t\bar A^{\varepsilon_\nu}-d_{\bar A^{\varepsilon_\nu}}\bar \Psi^{\varepsilon_\nu}$ by $\bar B_t^\nu$ and we remark that the $L^\infty$-norm of $\frac 1{m_\nu^2}\dot X_{t_\nu+\frac t{m_\nu}}(\bar A)$ can be estimate by $\frac c{m_\nu^3}$ where $c$ is a positive constant; thus, by the Sobolev estimate and the theorem \ref{crit:secder} we can conclude that
\begin{equation*}
\begin{split}
\sup_{t\in \left[-\frac{m_\nu}2,\frac {m_\nu}2\right]} \left\| \bar B_t^\nu\right\|_{L^4(\Sigma)}^2
\leq & c \sup_{t\in \left[-\frac{m_\nu}2,\frac {m_\nu}2\right]} \left(\|\bar B_t^\nu\|_{L^2(\Sigma)}^2+\|d_{\bar A^{\varepsilon_\nu}}^*d_{\bar A^{\varepsilon_\nu}}\bar B_t^\nu\|_{L^2(\Sigma)}^2\right)\\
\leq & c\int_{-\frac{m_\nu}2}^{\frac {m_\nu}2}\left(\|\bar B_t^\nu\|^2_{L^2(\Sigma)}+\frac 1{\varepsilon_\nu^2m_\nu^2}\|F_{\bar A^{\varepsilon_\nu}}\|^2_{L^2(\Sigma)}+\frac 1{m_\nu^3}\right) dt\\
\leq&\frac c{m_\nu}\left(1+\frac 1{m_\nu}+\frac 1{m_\nu^2}\right)
\end{split}
\end{equation*}
which converges to $0$ in contradiction to (\ref{crit:lidn}).\\

{\bf Case 2:} $\lim_{\nu\to \infty}\varepsilon_\nu m_\nu=c_1>0$. This time we choose a different rescaling to define $\bar \Xi^{\varepsilon_\nu}:=\bar A^{\varepsilon_\nu}+\bar\Psi^{\varepsilon_\nu} dt$, i.e.
$$\bar A^{\varepsilon_\nu}(t):=A^{\varepsilon_\nu}(t_\nu+\varepsilon_\nu t),\quad\bar \Psi^{\varepsilon_\nu}(t):=\varepsilon_\nu\Psi^{\varepsilon_\nu}(t_\nu+\varepsilon_\nu t)$$
which satisfies the perturbed Yang-Mills equations
\begin{equation*}
d_{\bar A^{\varepsilon_\nu}}^*F_{\bar A^{\varepsilon_\nu}}=\bar \nabla_t\left(\partial_t\bar A^{\varepsilon_\nu}-d_{\bar A^{\varepsilon_\nu}}\bar\Psi^{\varepsilon_\nu}\right)+\varepsilon_\nu^2*X_{t_\nu+\varepsilon_\nu t}(\bar A^{\varepsilon_\nu}),
\end{equation*}
\begin{equation*}
d_{\bar A^{\varepsilon_\nu}}^*\left(\partial_t\bar A^{\varepsilon_\nu}-d_{\bar A^{\varepsilon_\nu}}\bar\Psi^{\varepsilon_\nu}\right)=0
\end{equation*}
and 
\begin{equation}\label{crit:lidn1}
\sup_{t\in \left[-\frac 1{2\varepsilon_\nu},\frac 1{2\varepsilon_\nu}\right]} \left\| \partial_t\bar A^{\varepsilon_\nu}-d_{\bar A^{\varepsilon_\nu}}\bar \Psi^{\varepsilon_\nu}\right\|_{L^4(\Sigma)}=\left\| \partial_t\bar A^{\varepsilon_\nu}(0)-d_{\bar A^{\varepsilon_\nu}(0)}\bar\Psi^{\varepsilon_\nu}(0)\right\|_{L^4(\Sigma)}\leq 2 c_1
\end{equation}
for $\nu$ sufficiently big. Furthermore, we have the estimates
\begin{equation}
\begin{split}
\left\|F_{\bar A^{\varepsilon_\nu}}\right\|_{L^2}^2
=&\int_{-\frac {1}{2\varepsilon_\nu}}^{\frac {1}{2\varepsilon_\nu}}\left\|F_{\bar A^{\varepsilon_\nu}}\right\|_{L^2(\Sigma)}^2 dt\\
=&\int_{-\frac 12}^{\frac 12}\frac 1{\varepsilon_\nu}\left\|F_{A^{\varepsilon_\nu}}\right\|_{L^2(\Sigma)}^2 d t\leq b\varepsilon_\nu,
\end{split}
\end{equation}
\begin{equation}
\begin{split}
\left\| \partial_t\bar A^{\varepsilon_\nu}-d_{\bar A^{\varepsilon_\nu}}\bar\Psi^{\varepsilon_\nu}\right\|_{L^2}^2
=&\int_{-\frac 1{2\varepsilon_\nu}}^{\frac 1{2\varepsilon_\nu}}\left\| \partial_t\bar A^{\varepsilon_\nu}-d_{\bar A^{\varepsilon_\nu}}\bar\Psi^{\varepsilon_\nu}\right\|_{L^2(\Sigma)}^2d t\\
=&\int_{-\frac 12}^{\frac 12}\varepsilon_\nu^2\left\| \partial_t A^{\varepsilon_\nu}-d_{ A^{\varepsilon_\nu}}\Psi^{\varepsilon_\nu}\right\|_{L^2(\Sigma)}^2\frac 1{\varepsilon_\nu}dt \leq b\varepsilon_\nu.
\end{split}
\end{equation}
If we denote $\partial_t\bar A^{\varepsilon_\nu}-d_{\bar A^{\varepsilon_\nu}}\bar \Psi^{\varepsilon_\nu}$ by $\bar B_t^\nu$ and we consider  $c\varepsilon_\nu^3$ as the bound for the $L^\infty$-norm of $\varepsilon_\nu^2\dot X_{t_\nu+\varepsilon_\nu t}(\bar A)$, then, by the Sobolev estimate and the theorem \ref{crit:secder} we can conclude that
\begin{equation*}
\begin{split}
\sup \left\| \bar B_t^\nu\right\|_{L^4(\Sigma)}^2
\leq & c\sup \left(\|\bar B_t^\nu\|_{L^2(\Sigma)}^2+\|d_{\bar A^{\varepsilon_\nu}}^*d_{\bar A^{\varepsilon_\nu}}\bar B_t^\nu\|_{L^2(\Sigma)}^2\right)\\
\leq & c\int_{ S^1}\left(\|\bar B_t^\nu\|^2_{L^2(\Sigma)}+\|F_{\bar A^{\varepsilon_\nu}}\|^2_{L^2(\Sigma)}+\varepsilon_\nu^3\right) dt\\
\leq& c\varepsilon_\nu\left(1+\varepsilon_\nu+\varepsilon_\nu^2\right)
\end{split}
\end{equation*}
which converges to $0$ in contradiction to (\ref{crit:lidn1}).\\

{\bf Case 3:} $\lim_{\nu\to \infty}\varepsilon_\nu m_\nu=\infty$. First, we define $\bar \Xi^{\varepsilon_\nu}:=\bar A^{\varepsilon_\nu}+\bar\Psi^{\varepsilon_\nu} dt$ as in the first case, i.e.
$$\bar A^{\varepsilon_\nu}(t):=A^{\varepsilon_\nu}\left(t_\nu+ \frac t{m_\nu}\right),\quad\bar \Psi^{\varepsilon_\nu}(t):=\frac 1{m_\nu}\Psi^{\varepsilon_\nu}\left(t_\nu+\frac t{m_\nu}\right).$$
The new sequence satisfies then
\begin{equation}\label{sjsjs}
d_{\bar A^{\varepsilon_\nu}}^*F_{\bar A^{\varepsilon_\nu}}=\varepsilon_\nu^2m_{\nu}^2\bar \nabla_t\left(\partial_t\bar A^{\varepsilon_\nu}-d_{\bar A^{\varepsilon_\nu}}\bar\Psi^{\varepsilon_\nu}\right)+\varepsilon_\nu^2*X_{t_\nu+t/m_\nu}(\bar A^{\varepsilon_\nu}),
\end{equation}
\begin{equation*}
d_{\bar A^{\varepsilon_\nu}}^*\left(\partial_t\bar A^{\varepsilon_\nu}-d_{\bar A^{\varepsilon_\nu}}\bar\Psi^{\varepsilon_\nu}\right)=0.
\end{equation*}
In addition, we obtain the following estimates for any compact set $K\subset \mathbb R$ that
\begin{equation}\label{crit:lidn2}
\sup_{t\in \left[-\frac 1{2m_\nu},\frac 1{2m_\nu}\right]} \left\| \partial_t\bar A^{\varepsilon_\nu}-d_{\bar A^{\varepsilon_\nu}}\bar \Psi^{\varepsilon_\nu}\right\|_{L^4(\Sigma)}=\left\| \partial_t\bar A^{\varepsilon_\nu}(0)-d_{\bar A^{\varepsilon_\nu}(0)}\bar\Psi^{\varepsilon_\nu}(0)\right\|_{L^4(\Sigma)}= 1,
\end{equation}
\begin{equation*}
\left\|F_{\bar A^{\varepsilon_\nu}}\right\|_{L^2(\Sigma\times K)}^2
=\int_{ K}\left\|F_{\bar A^{\varepsilon_\nu}}\right\|_{L^2(\Sigma)}^2 d t\leq m_\nu \int_{ K}\left\|F_{ A^{\varepsilon_\nu}}\right\|_{L^2(\Sigma)}^2 d t\leq c \varepsilon_\nu^2m_\nu,
\end{equation*}
\begin{equation}\label{dllallalh}
\begin{split}
\left\| \partial_t\bar A^{\varepsilon_\nu}-d_{\bar A^{\varepsilon_\nu}}\bar\Psi^{\varepsilon_\nu}\right\|_{L^2(K)}^2
=&\int_K \left\| \partial_t\bar A^{\varepsilon_\nu}-d_{\bar A^{\varepsilon_\nu}}\bar\Psi^{\varepsilon_\nu}\right\|_{L^2(\Sigma)}^2d t\\
\leq &\int_K \frac 1{m_\nu^2}\left\| \partial_t A^{\varepsilon_\nu}-d_{ A^{\varepsilon_\nu}}\Psi^{\varepsilon_\nu}\right\|_{L^2(\Sigma)}^2m_\nu dt\leq \frac b{m_\nu}.
\end{split}
\end{equation}
Analogously as in the first two cases  we denote $\partial_t\bar A^{\varepsilon_\nu}-d_{\bar A^{\varepsilon_\nu}}\bar \Psi^{\varepsilon_\nu}$ by $\bar B_t^\nu$ and we consider  $\frac 1{m_\nu^3}$ as the bound for the $L^\infty$-norm of $\frac 1{m_\nu^2}\dot X_{t_\nu+\frac t{m_\nu}}(\bar A)$, then, by the Sobolev estimate and the theorem \ref{crit:secder} we can conclude that, for a compact set $K$ and an open set $\Omega$ with $0\in K\subset\Omega $,
\begin{equation*}
\begin{split}
\sup_{t\in K}\left\| \bar B_t^\nu\right\|_{L^4(\Sigma)}^2
\leq & c\sup_{t\in K}\left(\|\bar B_t^\nu\|_{L^2(\Sigma)}^2+\|d_{\bar A^{\varepsilon_\nu}}\bar B_t^\nu\|_{L^2(\Sigma)}^2\right)\\
\leq & c\int_{ \Omega}\left(\|\bar B_t^\nu\|^2_{L^2(\Sigma)}+\frac 1{\varepsilon_\nu^2m_\nu^2}\|F_{\bar A^{\varepsilon_\nu}}\|^2_{L^2(\Sigma)}\right) dt\\
 &+ c\int_{ \Omega}\left(\frac{\varepsilon_\nu^2}{m_\nu}+\|d_{\bar A^{\varepsilon_\nu}}\bar B_t^\nu\|_{L^2(\Sigma)}^2 \right) dt\\
\leq& \frac c{m_\nu}+\frac c{m_\nu}\int_{ S^1}\left \|d_{ A^{\varepsilon_\nu}}\left(\partial_tA^{\varepsilon_\nu}-d_{A^{\varepsilon_\nu}}\Psi^{\varepsilon_\nu}\right)\right\|^2_{L^2(\Sigma)} dt\\
\leq& \frac c{m_\nu}+\frac {c\varepsilon_\nu^\frac 12}{m_\nu}
\end{split}
\end{equation*}
where the last step follows from the next claim. Then the $L^4$-norm of $\bar B_t^\nu$ converges to $0$ by the last estimate in contradiction to (\ref{crit:lidn2}).\\

{\bf Claim:} For any perturbed Yang-Mills connection $\Xi=A+\Psi dt$
\begin{equation*}
\|d_AB_t \|_{L^2}\leq c\varepsilon^{\frac 14}
\end{equation*}
where we denote $\partial_tA-d_A\Psi$ by $B_t$.\\

\begin{proof}If we consider the identity
\begin{equation}\label{crit:ap:eqls}
\begin{split}
&
\varepsilon^2\left\|\frac 1{\varepsilon^2}d_A^*F_A-\nabla_t B_t-*X_t(A)\right\|^2_{L^2}
+\|\nabla_t F_A-d_AB_t\|^2_{L^2}\\
=&
\frac 1{\varepsilon^2}\left\|d_A^*F_A\right\|^2_{L^2}+\varepsilon^2\left\|\nabla_t B_t\right\|^2_{L^2}+\varepsilon^2
\left\|X_t(A)\right\|^2_{L^2}\\
&+\left\|\nabla_t F_A\right\|^2_{L^2}+\left\|d_AB_t\right\|^2_{L^2}-2\varepsilon^2\left\langle *X_t(A),\frac 1{\varepsilon^2}d_A^*F_A-\nabla_t B_t\right\rangle\\
&-2 \left\langle d_A^*F_A,\nabla_t B_t\right\rangle -2\left\langle\nabla_t F_A,d_AB_t\right\rangle,
\end{split}
\end{equation}
we can remark that first line vanishes by the perturbed Yang-Mills equation (\ref{epsiloneq1}) and by the Bianchi identity $\nabla_t F_A=d_AB_t$; in addition, the last line can be written as 
\begin{equation*}
-2\left\langle d_A^*F_A,\nabla_t B_t\right\rangle -2\left\langle\nabla_t F_A,d_AB_t\right\rangle
=2\left\langle F_A,[B_t\wedge B_t]\right\rangle
\end{equation*}
by the commutation formula (\ref{commform}). The identity (\ref{crit:ap:eqls}) yields therefore to
\begin{align*}
\|d_AB_t\|^2_{L^2}&+\varepsilon^2\|\nabla_tB_t\|^2_{L^2}\\
\leq& 2\left|\left\langle d_A*X_t(A),F_A\right\rangle\right|+\varepsilon^2|\langle *\nabla_t X_t(A),B_t\rangle| +c\|F_A\|_{L^2}\cdot \|B_t\|^2_{L^4}\\
\leq& c\|F_A\|^2_{L^2}+\varepsilon^2(1+\|B_t\|^2_{L^2}) +c\varepsilon^{-\frac 12}\|F_A\|_{L^2}\cdot \|B_t\|^2_{1,2,\varepsilon}\\
\leq& c\varepsilon^2(1+\|B_t\|^2_{L^2})+c\varepsilon^{\frac 12}\left( \|B_t\|^2_{L^2}+\|d_AB_t\|^2_{L^2}+\varepsilon^2\|\nabla_tB_t\|^2_{L^2}\right)\\
\leq& c\varepsilon^{\frac12}+c\varepsilon^{\frac 12}\left( \|d_AB_t\|^2_{L^2}+\varepsilon^2\|\nabla_tB_t\|^2_{L^2}\right)
\end{align*}
where we use the H\"older inequality and the Sobolev estimate in the second estimate and the assumption $\frac 1{\varepsilon^2}\|F_A\|_{L^2}^2+\|B_t\|_{L^2}^2\leq 2b$ in the last two estimates. Thus choosing $\varepsilon$ small enough the claim holds.
\end{proof}

Since we have found a contradiction for all the tree cases, we can conclude that $\sup_{t\in S^1}\|\partial_tA-d_A\Psi\|_{L^4}$ is uniformly bounded for $\varepsilon$ sufficiently small and thus the proof of the theorem \ref{thm:norminty} is finished.
\end{proof}

\section{Surjectivity of $\mathcal T^{b,\varepsilon}$}\label{chapter:surjectivity}

In the fifth chapter we defined the injective map $\mathcal T^{\varepsilon,b}$ in a unique way, in this one we show that it is also surjective provided that $\varepsilon$ is chosen sufficiently small.

\begin{theorem}\label{thm:surjj}
Let $b>0$ be a regular value of $E^H$. Then there is a constant $\varepsilon_0>0$ such that
\begin{equation*}
\mathcal T^{\varepsilon, b}:\mathrm{Crit}_{E^H}^b\to\mathrm{Crit}_{\mathcal{YM}^{\varepsilon,H}}^b
\end{equation*}
is bijective for $0<\varepsilon<\varepsilon_0$.
\end{theorem}
\begin{proof} The indirect proof will be divided in five steps. First, we assume that there is a decreasing sequence $\varepsilon_\nu$, $\nu\to \infty$, converging to $0$ and a sequence of perturbed Yang-Mills connections $\Xi^{\nu}=A^{\nu}+\Psi^{\nu}dt\in \mathrm{Crit}_{\mathcal {YM}^{\varepsilon_\nu,H}}^b$ that are not in the image of  $\mathcal T^{\varepsilon_\nu,b}$. By the theorems \ref{lemma:secder} and \ref{thm:norminty}  the sequence satisfies
\begin{equation}\label{crit:eq:supest1}
\left\|F_{A^{\nu}}\right\|_{L^\infty(\Sigma)}\leq c\varepsilon^{2-\frac 1p},\quad \left\|\partial_t A^{\nu}-d_{A^{\nu}}\Psi^{\nu}\right\|_{L^\infty(\Sigma)}
\leq c,
\end{equation}
\begin{equation}\label{crit:eq:supest}
\begin{split}
\sup_{s\in S^1} \big(&\left\|F_{A^{\nu}}\right\|_{L^2(\Sigma)}+\left\|d_{A^\nu}^*F_{A^{\nu}}\right\|_{L^2(\Sigma)}+\left\|d_{A^\nu}d_{A^\nu}^*F_{A^{\nu}}\right\|_{L^2(\Sigma)}\\
&+\varepsilon_\nu\left\|\nabla_t^{\Psi^{\nu}}F_{A^{\nu}}\right\|_{L^2(\Sigma)}
+ \varepsilon_\nu^2\left\|\nabla_t^{\Psi^{\nu}}\nabla_t^{\Psi^{\nu}}F_{A^{\nu}}\right\|_{L^2(\Sigma)}\big)
\leq c\varepsilon_\nu^{2-1/p}.
\end{split}
\end{equation}
In the estimate (\ref{crit:eq:supest}) the constant $c$ depends on $p\geq2$ which can be taken as big as we want. In order to conclude the proof we will need to choose $p>6$ as we will see in the proof of the fifth step.\\

In step 1, for each $\Xi^{\nu}$  we will define a connection 
 $\bar\Xi^{\nu}=\bar A^{\nu}
 +\bar\Psi^{\nu}dt$ near $\Xi^{\nu}$, flat on the fibers, which satisfies, for a constant $c>0$, $\left\|\pi_{\bar A^{\nu}}\left(\mathcal F^0\left(\bar\Xi^{\nu}\right)\right)\right\|_{L^p}\leq c\varepsilon_\nu^{{1-1/p}}$. Next, in the second step, we will find a representative $\Xi^0$ of a perturbed geodesic for which $\left\|\Xi^{\nu}-\Xi^0\right\|_{1,p,1}+\left\|\Xi^{\nu}-\Xi^0\right\|_{L^\infty }\leq c\varepsilon_\nu^{{1-1/p}}$ for a subsequence of $\{\Xi^{\nu}\}_{\nu\in \mathbb N}$ (step 3). Then, in step 5, we will improve this last estimate in order to apply the local uniqueness theorem \ref{thm:localuniqueness} which requires that the norms are bounded by $\delta \varepsilon$ for $\delta$ and $\varepsilon$ sufficiently small; in this way we will have a contradiction, because a subsequence of $\{\Xi^\nu\}_{\nu\in\mathbb N }$ will turn out to be in the image of $\mathcal T^{\varepsilon_\nu,b }$. \\

{\bf Step 1.} There are two positive constants $c$ and $\nu_0$ such that the following holds. For every $\Xi^{\nu}$, $\nu>\nu_0$, there is a connection $\bar\Xi^{\nu}=\bar A^{\nu}+\bar\Psi^{\nu}dt$ which satisfies
$$\begin{array}{llll}
i)&F_{\bar A^{\nu}}=0,
&ii)\,\,& d_{\bar A^{\nu}}^*(\partial_t\bar A^{\nu}-d_{\bar A^{\nu}}\bar\Psi^{\nu})=0,\\[5pt]
iii)\,\,&\left\|\Xi^{\nu}-\bar\Xi^{\nu}\right\|_{\bar\Xi^{\nu},1,p,\varepsilon_\nu}\leq c\varepsilon_\nu^{2-\frac 1p},\quad
&iv)&\left\|\pi_{\bar A^{\nu}}\left(\mathcal F^0\left(\bar\Xi^{\nu}\right)\right)\right\|_{L^p}\leq c\varepsilon_\nu^{{1-1/p}}.
\end{array}$$

\vspace{-3pt}
\begin{proof}[Proof of step 1]
Since $\|F_{A^\nu}\|_{L^\infty(\Sigma)}\leq c\varepsilon^{2-\frac 1p}$, by lemma \ref{lemma82dt94} there is a positive constant $c$ such that for any $A^{\nu}$ there is a unique $0$-form $\gamma^{\nu}$ which satisfies $F_{A^{\nu}+*d_{A^{\nu}}\gamma^{\nu}}=0$, $\|d_{A^{\nu}}\gamma^{\nu}\|_{L^\infty(\Sigma)}\leq c\, \|F_{A^{\nu}}\|_{L^\infty(\Sigma) }$. We denote $\bar\Xi^{\nu}:=\bar A^{\nu}+\bar\Psi^{\nu}dt$ where 
$\bar A^{\nu}:=A^{\nu}+*d_{A^{\nu}}\gamma^{\nu}$, 
$\alpha^{\nu}:=*d_{A^{\nu}}\gamma^{\nu}$ and 
$\bar\Psi^{\nu}:=\Psi^{\nu}+\psi^{\nu}$ is the unique $0$-form such that $d_{\bar A^{\nu}}^*(\partial_t\bar A^{\nu}-d_{\bar A^{\nu}}\bar\Psi^{\nu})=0$; $\bar \Psi^\nu$ is well defined because $d_{A}^*d_{A}:\Omega^0(\Sigma, \mathfrak g_P)\to \Omega^0(\Sigma,\mathfrak g_p)$ is bijective for any flat connection $A$. Hence, $\alpha^{\nu}$ satisfies the estimate
\begin{equation}\label{flow:surj:est:df}
\|\alpha^{\nu}\|_{L^\infty(\Sigma)}=\|d_{A^{\nu}}\gamma^{\nu}\|_{L^\infty(\Sigma)}
\leq c\, \|F_{A^{\nu}}\|_{L^\infty(\Sigma) }\leq c\varepsilon_\nu^{2-\frac 1p}.
\end{equation}
Since the $\Xi^{\nu}$ is a perturbed Yang-Mills connection, i.e.
\begin{equation}\label{eq:ym111}
\frac 1{\varepsilon_\nu^2}d_{A^{\nu}}^*F_{A^{\nu}}
=\nabla_t^{\Psi^{\nu}}(\partial_tA^{\nu}
-d_{A^{\nu}}\Psi^{\nu})+*X(A^{\nu}),
\end{equation}
we have that the connections $\bar \Xi^\nu$ satisfy
\begin{equation*}
\begin{split}
\nabla_t^{\bar\Psi^{\nu}}(\partial_t\bar A^{\nu}&-d_{\bar A^{\nu}}\bar\Psi^{\nu})+*X(\bar A^{\nu})\\
=\,\,\,&\nabla_t^{\Psi^{\nu}}(\partial_tA^{\nu}-d_{A^{\nu}}\Psi^{\nu})+*X(A^{\nu})\\
& +[\psi^{\nu},(\partial_tA^{\nu}-d_{A^{\nu}}\Psi^{\nu})] +\nabla_t^{\bar\Psi^{\nu}}(\nabla^{\bar\Psi^{\nu}}_t \alpha^{\nu}-d_{ A^{\nu}}\psi^{\nu})\\
=&\frac 1{\varepsilon_\nu^2}d_{\bar A^{\nu}}^*F_{A^{\nu}}-
\frac 1{\varepsilon_\nu^2}*\left[ \alpha^{\nu}\wedge*F_{A^{\nu}}\right]+2[\psi^{\nu},(\partial_tA^{\nu}-d_{A^{\nu}}\Psi^{\nu})]\\
&+\nabla_t^{\bar\Psi^{\nu}}\nabla^{\bar\Psi^{\nu}}_t \alpha^{\nu}
-d_{ A^{\nu}}\nabla_t^{\Psi^\nu}\psi^{\nu}-[\psi^\nu,d_{A^\nu}\psi^\nu]
\end{split}
\end{equation*}
where in the last equality we used (\ref{eq:ym111}) and the commutation formula (\ref{commform}); thus,
\begin{equation*}
\begin{split}
\pi_{\bar A^{\nu}}\left(\mathcal F^0(\bar A^{\nu},\bar \Psi^{\nu})\right)
=&-\pi_{\bar A^{\nu}}\left(\nabla_t^{\bar\Psi^{\nu}}(\partial_t\bar A^{\nu}
-d_{\bar A^{\nu}}\bar\Psi^{\nu})+*X(\bar A^{\nu})\right)\\
=&
\pi_{\bar A^{\nu}}\left(*\frac 1{\varepsilon_\nu^2}\left[ \alpha^{\nu}\wedge*F_{A^{\nu}}\right]
-2[\psi^{\nu},(\partial_tA^{\nu}-d_{A^{\nu}}\Psi^{\nu})]\right)\\
&-\pi_{\bar A^{\nu}}\left([\psi^\nu,d_{A^\nu}\psi^\nu]
+\nabla_t^{\bar\Psi^{\nu}}\nabla^{\bar\Psi^{\nu}}_t\alpha^{\nu}+[\alpha^\nu,\nabla_t^{\Psi^{\nu}}\phi^\nu]\right).
\end{split}
\end{equation*}
Therefore, by (\ref{flow:surj:est:df}) and the next lemma,
\begin{equation*}
\begin{split}
\Big\|\pi_{\bar A^{\nu}}&\left(\mathcal F^0(\bar A^{\nu},\bar\Psi^{\nu})\right)\Big\|_{L^p}\\
\leq &\left\|
\frac 1{\varepsilon_\nu^2}\pi_{\bar A^{\nu}}\left(*\left[ \alpha^{\nu}\wedge*
F_{A^{\nu}}\right]\right)\right\|_{L^p}
+\left\|\pi_{\bar A^{\nu}}\left(2[\psi^{\nu},(\partial_tA^{\nu}-d_{A^{\nu}}\Psi^{\nu})]\right)\right\|_{L^p}\\
&+\left\|\pi_{\bar A^{\nu}}\left(
\nabla_t^{\bar\Psi^{\nu}}\nabla^{\bar\Psi^{\nu}}_t\alpha^{\nu}+[\alpha^\nu,\nabla_t^{\Psi^{\nu}}\phi^\nu]-[\psi^\nu,d_{A^\nu}\psi^\nu]\right)\right\|_{L^p}\\
\leq&\,c\varepsilon_\nu^{2-\frac2p}+\left\|\pi_{\bar A^{\nu}}
\left(\nabla_t^{\Psi^{\nu}}\nabla^{\bar\Psi^{\nu}}_t\alpha^{\nu}\right)\right\|_{L^p}
\end{split}
\end{equation*}
where 
\begin{equation*}
\begin{split}
\Big\|\pi_{\bar A^{\nu}}&
\left(\nabla_t^{\Psi^{\nu}}\nabla^{\bar\Psi^{\nu}}_t\alpha^{\nu}\right)\Big\|_{L^p}\\
\leq&\left\|\pi_{\bar A^{\nu}}
\left(\nabla_t^{\Psi^{\nu}}[\psi^{\nu},\alpha^{\nu}]+*\nabla_t^{\Psi^{\nu}}[(\partial_tA^{\nu}-d_{A^{\nu}}\Psi^{\nu}),\gamma^{\nu}]\right)\right\|_{L^p}\\
&+\left\|\pi_{\bar A^{\nu}}
\left(d_{A^{\nu}}\nabla_t^{\Psi^{\nu}}\nabla_t^{\Psi^{\nu}}\gamma^{\nu}
+[(\partial_tA^{\nu}-d_{A^{\nu}}\Psi^{\nu}),\nabla_t^{\Psi^{\nu}}\gamma^{\nu}]\right)\right\|_{L^p}\\
\leq&c\varepsilon_\nu^{{1-1/p}}
\end{split}
\end{equation*}
follows from lemma \ref{lemma:var2} and hence
\begin{equation}\label{eq:estimatepif}
\Big\|\pi_{\bar A^{\nu}}\left(\mathcal F^0(A^{\nu},\Psi^{\nu})\right)\Big\|_{0,p,\varepsilon}
\leq c\varepsilon_\nu^{{1-1/p}}.
\end{equation}
The estimate $\left\|\Xi^{\nu}-\bar\Xi^{\nu}\right\|_{\bar\Xi^{\nu},1,p,\varepsilon_\nu}\leq c\varepsilon_\nu^{2-\frac 1p}$ follows from the lemma \ref{lemma:var2}.
\end{proof}

\begin{lemma}\label{lemma:var2}
There are constants $c>0$, $\bar\varepsilon_0>0$ such that
\begin{gather*}
\left\|\psi^{\nu}\right\|_{L^\infty(\Sigma)}+\left\|d_{A^{\nu}}\psi^{\nu}\right\|_{L^{p}(\Sigma) }\leq c\varepsilon_\nu^{2-1/p}, \\
\left\|\nabla^{\Psi^{\nu}}_t\alpha^{\nu}\right\|_{L^p(\Sigma)}
+\left\|\nabla^{\Psi^{\nu}}_t\gamma^{\nu}\right\|_{L^\infty(\Sigma)}
\leq c\varepsilon_\nu^{1-1/p}, \\
\left\|\nabla^{\Psi^{\nu}}_t\psi^{\nu}\right\|_{L^\infty(\Sigma) }+\varepsilon \left\|\nabla^{\Psi^{\nu}}_t\nabla^{\Psi^{\nu}}_t\gamma^{\nu}\right\|_{L^\infty (\Sigma)}
\leq c\varepsilon_\nu^{1-1/p}
\end{gather*}
for any $0<\varepsilon_\nu<\bar\varepsilon_0$.
\end{lemma}

\begin{proof}[Proof of lemma \ref{lemma:var2}]
Since the Yang-Millas connections $\Xi^\nu$ satisfy $$d_{A^\nu}^*\left(\partial_tA^\nu-d_{A^\nu}\Psi^\nu\right)=0,$$ from the definition of $\psi^{\nu}$ we have 
\begin{equation}\label{eq:lemma201}
\begin{split}
0=&d_{\bar A^{\nu}}^*
\left(\partial_t\bar A^{\nu}-d_{\bar A^{\nu}}\bar\Psi^{\nu}\right)\\
=&-*\left[\alpha^{\nu}\wedge*\left(\partial_tA^{\nu}-d_{A^{\nu}}\Psi^{\nu}\right)\right]
+d_{\bar A^{\nu}}^*\nabla_t^{\Psi^{\nu}}\alpha^{\nu}
-d_{\bar A^{\nu}}^*d_{\bar A^{\nu}}\psi^{\nu}
\end{split}
\end{equation}
where
\begin{equation*}
\begin{split}
d_{\bar A^{\nu}}^*\nabla_t^{\Psi^{\nu}}\alpha^{\nu}
=&-*[(\partial_tA^{\nu}-d_{A^{\nu}}\Psi^{\nu})\wedge*\alpha^{\nu}]-*[\alpha^{\nu}\wedge *\nabla_t^{\Psi^{\nu}}\alpha^{\nu}]+\nabla_t^{\Psi^{\nu}}d_{A^{\nu}}^*\alpha^{\nu},\\
\nabla_t^{\Psi^{\nu}}d_{A^{\nu}}^*\alpha^{\nu}=&\nabla_t^{\Psi^{\nu}}d_{A^{\nu}}^*d_{A^{\nu}}^**\gamma^{\nu}
=\nabla_t^{\Psi^{\nu}}*[F_{A^{\nu}}\wedge\gamma^{\nu}]\\
=&*[\nabla_t^{\Psi^{\nu}}F_{A^{\nu}}\wedge\gamma^{\nu}]+*[F_{A^{\nu}}\wedge\nabla_t^{\Psi^{\nu}}\gamma^{\nu}].
\end{split}
\end{equation*}
Since we know that 
$\|\alpha^{\nu}\|_{L^\infty(\Sigma)}
+\|\gamma^{\nu}\|_{L^\infty(\Sigma)}\leq c\varepsilon_\nu^{2-\frac1p}$, the proof of the first two inequalities of the lemma is completed by showing that there exists a constant $c$ such that
\begin{equation}\label{crit:est:surj1}
\|\nabla_t^{\Psi^{\nu}}\alpha^{\nu}\|_{L^p(\Sigma)}
+ \|\nabla_t^{\Psi^{\nu}}\gamma^{\nu}\|_{L^p(\Sigma)}\leq c \varepsilon_\nu^{1-1/p}
\end{equation}
and estimating the norms of $\psi^{\nu}$ and of $d_{\bar A^{\nu}}\psi^{\nu}$ using (\ref{eq:lemma201}):
\begin{align*}
\|\psi^{\nu}\|_{L^\infty(\Sigma)}\leq &c\|d_{\bar A^{\nu}}\psi^{\nu}\|_{L^p(\Sigma)}\leq c\|d_{\bar A^{\nu}}^*d_{\bar A^{\nu}}\psi^{\nu}\|_{L^2(\Sigma)}\\
= &
\|-*\left[\alpha^{\nu}\wedge*\left(\partial_tA^{\nu}-d_{A^{\nu}}\Psi^{\nu}\right)\right]
+d_{\bar A^{\nu}}^*\nabla_t^{\Psi^{\nu}}\alpha^{\nu}\|_{L^2(\Sigma)}\\
\leq&c\|\alpha^\nu\|_{L^2(\Sigma)}+\|\alpha^\nu\|_{L^\infty(\Sigma)}\|\nabla_t^{\Psi^\nu}\alpha^\nu\|_{L^2(\Sigma)}\\
&+\|\gamma^\nu\|_{L^\infty(\Sigma)}\|\nabla_t^{\Psi^\nu}F_{A^\nu}\|_{L^2(\Sigma)}+\|F_{A^\nu}\|_{L^\infty(\Sigma)}\|\nabla_t^{\Psi^\nu}\gamma^\nu\|_{L^2(\Sigma)}\\
\leq & c\varepsilon_\nu^{2-\frac 1p}.
\end{align*}
In order to show (\ref{crit:est:surj1}) we derive 
$$F_{A^{\nu}}+d_{A^{\nu}}*d_{A^{\nu}}\gamma^{\nu}
+\frac12 [d_{A^{\nu}}\gamma^{\nu}\wedge d_{A^{\nu}}\gamma^{\nu}]=0$$ 
by $\nabla_t^{\Psi^{\nu}}$ and we obtain
\begin{equation}\label{crit:surj:hh}
\begin{split}
d_{A^{\nu}}*d_{A^{\nu}}\nabla_t^{\Psi^{\nu}}\gamma^{\nu}=&
-\nabla_t^{\Psi^{\nu}}F_{A^{\nu}}
-[d_{A^{\nu}}\nabla_t^{\Psi^{\nu}}\gamma^{\nu}\wedge d_{A^{\nu}}\gamma^{\nu}]\\
&-[[(\partial_tA^\nu-d_{A^\nu}\Psi^{\nu})\wedge \gamma^{\nu}]\wedge d_{A^{\nu}}\gamma^{\nu}]\\
&-[(\partial_tA^\nu-d_{A^\nu}\Psi^{\nu})\wedge *d_{A^{\nu}}\gamma^{\nu}]
\end{split}
\end{equation}
and hence, by (\ref{crit:eq:supest1}),
\begin{align*}
\|d_{A^{\nu}}*d_{A^{\nu}}&\nabla_t^{\Psi^{\nu}}\gamma^{\nu}\|_{L^2(\Sigma) }\\
\leq&  c\|\nabla_t^{\Psi^{\nu}}F_{A^{\nu}}\|_{L^2(\Sigma)}
+c\|\alpha^\nu\|_{L^\infty(\Sigma)}\|d_{A^{\nu}}*d_{A^{\nu}}\nabla_t^{\Psi^{\nu}}\gamma^{\nu}\|_{L^2(\Sigma) }\\
&+c\left\|\partial_tA^\nu-d_{A^\nu}\Psi^{\nu}\right\|_{L^\infty(\Sigma)}\|\alpha^\nu\|_{L^\infty(\Sigma)}\left(1+\|\gamma^\nu\|_{L^2(\Sigma)}\right)\\
\leq& c\left(\|\nabla_t^{\Psi^{\nu}}F_{A^{\nu}}\|_{L^2(\Sigma)}+\|\alpha^{\nu}\|_{L^\infty(\Sigma)}\right)+c\varepsilon^{\frac 2-\frac 1p}\|d_{A^{\nu}}*d_{A^{\nu}}\nabla_t^{\Psi^{\nu}}\gamma^{\nu}\|_{L^2(\Sigma) }.
\end{align*}
Choosing $\varepsilon$ sufficiently small, we have by (\ref{crit:eq:supest}),
$$\|d_{A^{\nu}}*d_{A^{\nu}}\nabla_t^{\Psi^{\nu}}\gamma^{\nu}\|_{L^2(\Sigma) }
\leq c\varepsilon_\nu^{1-\frac 1p}$$
which yields to
\begin{equation*}
\begin{split}
\|\nabla_t^{\Psi^{\nu}}\gamma^{\nu}\|_{L^\infty(\Sigma) }
\leq&c\|d_{A^{\nu} }*d_{A^{\nu}}\nabla_t^{\Psi^{\nu}}\gamma^{\nu}\|_{L^2(\Sigma) }\\
\leq &c\left(\|\nabla_t^{\Psi^{\nu}}F_{A^{\nu}}\|_{L^2(\Sigma) }
+\|\alpha^{\nu}\|_{L^\infty(\Sigma)}\right)\leq c\varepsilon_\nu^{1-1/p}
\end{split}
\end{equation*}
by lemma \ref{lemma76dt94} and
\begin{equation*}
\begin{split}
\|\nabla_t^{\Psi^{\nu}}&\alpha^{\nu}\|_{L^p(\Sigma) }
=\|\nabla_t^{\Psi^{\nu}}d_{A^{\nu}}\gamma^{\nu}\|_{L^p(\Sigma)}\\
\leq& \|d_{A^{\nu}}\nabla_t^{\Psi^{\nu}}\gamma^{\nu}\|_{L^p(\Sigma)}
+\|[(\partial_tA^{\nu}-d_{A^{\nu}}\Psi^{\nu}),\gamma^{\nu}]\|_{L^p(\Sigma)}\\
\leq&c \|d_A*d_{A^{\nu}}\nabla_t^{\Psi^{\nu}}\gamma^{\nu}\|_{L^2(\Sigma)}
+c\|\gamma^{\nu}\|_{L^p(\Sigma)}
\leq c_2\varepsilon_\nu^{1-1/p}.
\end{split}
\end{equation*}
Analogously, one can obtain the third inequality of the lemma; the starting point is to derive (\ref{eq:lemma201}) and (\ref{crit:surj:hh}) by $\nabla_t^{\Psi^{\nu}}$ and to use the estimate $$\left\|\nabla_t^{\Psi^{\nu}}\nabla_t^{\Psi^{\nu}}F_{A^{\nu}}\right\|_{L^2(\Sigma)}\leq c_2\varepsilon^{-1/p}$$
in order to show
$$\left\|\nabla^{\Psi^{\nu}}_t\psi^{\nu}\right\|_{L^\infty(\Sigma) }+\varepsilon_\nu\left\|\nabla^{\Psi^{\nu}}_t\nabla^{\Psi^{\nu}}_t\gamma^{\nu}\right\|_{L^\infty (\Sigma)}
\leq c\varepsilon_\nu^{1-1/p}.$$

\end{proof}

In the following, by the Nash embedding theorem, we consider $\mathcal M^g(P)$ to be a compact submanifold of $\mathbb R^n$.\\

{\bf Step 2.} The sequence $\left\{u^\nu:=\left[\bar A^{\nu}\right]\right\}_{\nu\in \mathbb N}$ has a subsequence, still denoted by $u^\nu$, which converges to a perturbed geodesic $u^0$ respect to the norm $\|\cdot\|_{W^{1,p}}$ or more precisely
\begin{equation*}
\left\|u^\nu-u^0\right\|_{W^{1,p}}\leq c \varepsilon_\nu^{{1-1/p}}
\end{equation*}
for a constant $c>0$.
\vspace{-3pt}
\begin{proof}[Proof of step 2]

Since $F_{\bar A^{\nu}}=0$, the vector $\partial_t\bar A^{\nu}$ lies on the tangent space $T_{\bar A^{\nu}}\mathcal A_0(P)$ and hence in the kernel of $d_{\bar A^{\nu}}$; thus $d_{\bar A^{\nu}}(\partial_t\bar A^{\nu}-d_{\bar A^{\nu}}\bar\Psi^{\nu})=0$. Every $[\bar A^{\nu}]$ is therefore a curve in the moduli space $\mathcal M^{g}(P)$ with velocity $\partial_t\bar A^{\nu}-d_{\bar A^{\nu}}\bar\Psi^{\nu}$; moreover it approximates a geodesic in the sense of inequality (\ref{eq:estimatepif}). Therefore $\{u^\nu\}_{\nu \in \mathbb N}$ is a bounded Palais-Smale sequence and hence, using next lemma, it has a strong convergent subsequence that converge in the norm $\|\cdot\|_{W^{1,p}}$ to a perturbed geodesic $u^0$ and $\|u^\nu-u^0\|_{W^{1,p}}\leq c\varepsilon_\nu^{1-\frac1p}$.
\end{proof}

\begin{lemma}\label{lemma:struwe}
Let $p\geq2$ and $\mathcal M$ be a compact embedded manifold. We choose the energy
$$E(u)=\frac 12 \int_0^1\left(|\nabla u|^2+H_t(u)\right)dt$$
for any $u\in W^{1,p}(S^1,\mathcal M)$ where $H_t:\mathcal M\to \mathbb R$ is a smooth Hamiltonian. For every bounded sequence $\{u^\nu\}_{\nu\in \mathbb N}\subset W^{1,p}(S^1,\mathcal M)$ which satisfies
$$\|dE(u^\nu)\|_{L^p}\to 0$$
there is a critical curve $u_\infty\in W^{1,p}(S^1,\mathcal M)$ such that for a subsequence $\{u^{\iota_\nu }\}_{\nu\in\mathbb N }\subset \{u^\nu\}_{\nu\in\mathbb N }$ we have
\begin{enumerate}
\item $\|u^{\iota_\nu }-u^0\|_{W^{1,p}}\to 0\quad (k\to\infty)$;

\item The $\{LE(u^{\iota_\nu })\}_{\nu\in\mathbb N}$, where $LE$ denote the linearisation of $dE$, converges in $L^{p}$ to the Jacobi operator of $u^0$;

\item If the Jacobi operator of $u_\infty$ is invertible, then there is a constant $c>0$ such that
$\|u^{\iota_\nu }-u^0\|_{W^{1,p}}\leq c \|dE(u^\nu)\|_{L^p}.$
\end{enumerate}
\end{lemma}

\begin{proof}[Proof of lemma \ref{lemma:struwe}]$\,$
(1) The energy functional $E$ satisfies the Palais-Smale condition for the norm $\|\cdot\|_{W^{1,2}}$: We refer the reader to \cite{MR1736116}, theorem 4.4, for the proof in the case $\mathcal M$ is a surface and $H_t=0$, but the proof applied also for the general case. Therefore, $\{u^\nu\}_{\nu\in \mathbb N}$ has a subsequence, still denoted by $\{u^\nu\}_{\nu\in \mathbb N}$, which converges to a perturbed geodesic $u^0$ in $W^{1,2}(S^1,\mathcal M)$. It remains to prove that the sequence converges to $u^0$ also in $\|\cdot\|_{W^{1,p}}$, in fact
\begin{equation*}
\begin{split}
\|u^\nu-u^0\|_{W^{1,p}}\leq&
\sup_{v\in W^{1,q},\|v\|_{W^{1,q}}=1}
\int_0^1\Big(\langle \nabla(u^\nu-u^0),\nabla v\rangle +\langle u^\nu-u^0,v\rangle \Big)dt\\
=& \sup_{v\in W^{1,q},\|v\|_{W^{1,q}}=1}
\left(-\int_0^1\langle dH(u^\nu)-dH_t(u^0), v\rangle dt\right.\\
&+\left.\int_0^1\langle \Delta(u^\nu)+dH_t(u^\nu), v\rangle dt  +\int_0^1\langle u^\nu-u^0,v\rangle dt\right)
\end{split}
\end{equation*}
converges to 0.\\

(2) The $L^p$ convergence of $dE(u^\nu)$ implies the convergence of 
$\nabla_t\dot u^\nu$ because
\begin{equation*}
||\nabla_t\dot u^\nu-\nabla_t\dot u^0||_{L^p}
\leq ||dE(u^\nu)-dE(u^0)||_{L^p}+||dH(u^\nu)-dH(u^0)||_{L^p}
\end{equation*}
goes to $0$ for $\nu\to\infty$. We denote by $R$ the Riemann tensor of the manifold $\mathcal M$ and by $\Pi$ the projection on its tangent space. Then the linearisation of $dE$ respect to the loops $u^\nu$ is (cf. appendix B in \cite{Weberthesis})
$$LE(u^\nu)X(u^\nu)=-\nabla_{\dot u^\nu}\nabla_{\dot u^\nu}X(u^\nu)-R(X(u^\nu),\dot u^\nu)\dot u^\nu-\nabla_{X(u^\nu)}\nabla H_t(u^\nu)$$
for any vector field $X$ on $\mathcal M$ and the first term can be written as
\begin{equation*}
\begin{split}
\nabla_{\dot u^\nu}\nabla_{\dot u^\nu}X(u^\nu)=
&\,\nabla_{\dot u^\nu}\left(\Pi(u^\nu)dX(u^\nu)\dot u^\nu\right)\\
=&\,\Pi(u^\nu)\left(d\Pi(u^\nu)\dot u^\nu\right)\left(dX(u^\nu)\dot u^\nu\right)\\
&+\Pi(u^\nu)d^2X(u^\nu)\dot u^\nu\dot u^\nu)\\
&+\Pi(u^\nu)dX(u^\nu)\nabla_{\dot u^\nu}\dot u^\nu.
\end{split}
\end{equation*}
Thus, for a constant $c>0$,
\begin{align*}
\left\| LE(u^\nu)-LE(u^0)\right\|_{L^p}\leq & c\left(\left\|u^\nu-u^0\right\|_{L^p}+\left\|\dot u^\nu-\dot u^0\right\|_{L^p}\right)\\
&+c\left\|\nabla_{\dot u^\nu}\dot u^\nu-\nabla_{\dot u^0}\dot u^0\right\|_{L^p}.
\end{align*}
The sequence $\{LE(u^\nu)\}_{k\in\mathbb N}$ converges therefore to the Jacobi operator of $u^0$ in $L^p$.\\

(3) The third conclusion of the theorem can be proved using the following theorem (see Proposition A.3.4. in \cite{MR2045629}). In our case we chose
$$
\begin{array}{lrcl}
f:&W^{2,p}((u^\nu)^*T\mathcal M)&\to& L^{2}((u^\nu)^*T\mathcal M)\\
&x&\mapsto&f(x):=g_{x}(\mathcal F_0(\exp_{u^\nu}(x))
\end{array}
$$
where $g_x:L^p(\exp_{u^\nu}(x)^*T\mathcal M)\to L^p((u^\nu)^*T\mathcal M)$ is the parallel transport along $t\mapsto \exp_{u^\nu}((1-t)x).$
\end{proof}

\begin{theorem}
Let X and Y be Banach spaces, $U\subset X$ be an open set, and $f:U\to Y$ be a continuously differentiable map. Let $x_0\in U$ be such that $D:=df(x_0):X\to Y$ is surjective and has a (bounded linear) right inverse $Q:Y\to X$. Choose positive constants $\delta$ and $c$ such that $\|Q\|\leq c$, $B_\delta(x_0;X)\subset U$, and 
$$\|x-x_0\|<\delta\quad\Rightarrow\quad \|df(x)-D\|\leq\frac1{2c}.$$
Suppose that $x_1\in X$ satisfies
$$\|f(x_1)\|<\frac \delta{4c},\quad \|x_1-x_0\|<\frac\delta8.$$
Then there exists a unique $x\in X$ such that
$$f(x)=0,\quad x-x_1\in\textrm{im } Q,\quad \|x-x_0\|\leq\delta.$$
Moreover, $\|x-x_1\|\leq2c\|f(x_1)\|$.
\end{theorem}

{\bf Step 3.}There is a lift $\Xi^0$ of the closed geodesic $u^0$ and a sequence $g_\nu\subset\mathcal G_0^{2,p}(P\times S^1)$ such that
\begin{equation}
\left\|g_\nu^*\Xi^{\nu}-\Xi^0\right\|_{1,p,1}+\left\|g_\nu^*\Xi^{\nu}-\Xi^0\right\|_{L^\infty}\leq c\varepsilon_\nu^{{1-1/p}},\, \|d_{A^0}(g_\nu^*A^\nu-A^0)\|_{L^p}\leq c\varepsilon_\nu^{2-2/p}.
\end{equation}
and $d_{A^0}^*(g_\nu^*A^\nu-A^0)\|_{L^2}=0$. For expositional reasons we will still denote by $\Xi^\nu$ the sequence $g_\nu^*\Xi^\nu$.

\begin{proof}[Proof of step 3]
We choose now a representative $\Xi^0=A^0+\Psi^0 dt$ of the geodesic $u^0$. Since the sequence of curves on the moduli space converges to a geodesic $[\Xi^0]$ in $W^{1,p}$, i.e. 
\begin{equation}\label{est:epsm}
\left\|\left[\bar\Xi^{\nu}\right]-[\Xi^0]\right\|_{W^{1,p}(S^1,\mathcal M)}\leq c \varepsilon^{{1-1/p}}_\nu,
\end{equation}
by the Sobolev embedding theorem we have that
$$\left\|\left[\bar\Xi^{\nu}\right]-[\Xi^0]\right\|_{L^\infty}\leq c \varepsilon^{{1-1/p}}_\nu.$$
Therefore there is a sequence $g_\nu\subset \mathcal G_0^{2,p}(P\times S^1)$ such that 
\begin{equation}\label{eq:g=0}
d_{A^0}^*\left(g_\nu^*A^{\nu}-A^0\right)=0
\end{equation}
and in order to symplify the exposition we still denote the sequence $g_\nu^*\Xi^\nu$ by $\Xi^\nu$. The condition (\ref{eq:g=0}) means that we choose the closest connection in the orbit of $A_\nu$ to $A_0$ respect to the $L^2(\Sigma)$-norm. The existence of $g_\nu$ is assured by the lemma \ref{lemma:regauge} and by the local slice theorem (see theorem 8.1 in \cite{MR2030823}). Therefore $\left\|\bar\Xi^{\nu}-\Xi^0\right\|_{L^\infty}\leq c \varepsilon^{{1-1/p}}_\nu$ and thus by the first step
\begin{equation}\label{eqkk1}\left\|\Xi^{\nu}-\Xi^0\right\|_{L^\infty}\leq c \varepsilon^{{1-1/p}}_\nu.\end{equation}
Since 
$$d_{A^0}\left(A^{\nu}-A^0\right)=F_{A^{\nu}}-\frac 12 \left[\left(A^{\nu}-A^0\right)\wedge\left(A^{\nu}-A^0\right)\right],$$
we have the estimate 
\begin{equation}\label{eqkk2}
\left\|d_{A^0}\left(A^{\nu}-A^0\right)\right\|_{L^p}
\leq \|F_{A^{\nu}}\|_{L^p}+c\|A^{\nu}-A^0\|_{L^\infty}\|A^{\nu}-A^0\|_{L^p}\leq   c \varepsilon^{2-2/p}_\nu.
\end{equation}
Next, we remark, using $\nabla_t:=\partial_t+[\Psi^0,\cdot]$, that
\begin{equation*}
0=\nabla_t d_{A^0}^* \left(A^{\nu}-A^0\right)
=d_{A^0}^* \nabla_t \left(A^{\nu}-A^0\right)+*\left[\left(\partial_tA^0-d_{A^0}\Psi^0\right)\wedge *\left(A^{\nu}-A^0\right)\right],
\end{equation*}
thus,
\begin{equation}\label{eq:dfkdjfhdjas;}
\begin{split}
d_{A^0}^*d_{A^0}&\left(\Psi^{\nu}-\Psi^0\right)
=d_{A^0}^*\left(\partial_tA^0-d_{A^0}\Psi^0\right)
-d_{A^{\nu}}^*\left(\partial_tA^{\nu}-d_{A^{\nu}}\Psi^{\nu}\right)\\
&+d_{A^0}^* \nabla_t \left(A^{\nu}-A^0\right)
-d_{A^0}^*\left[\left(A^{\nu}-A^0\right)\wedge \left(\Psi^{\nu}-\Psi^0\right)\right]\\
&-*\left[(A^\nu-A^0)\wedge *\left(\partial_tA^\nu-d_{A^\nu }\Psi^\nu\right)\right]\\
=&-*\left[\left(A^{\nu}-A^0\right)\wedge *\left(\partial_tA^{\nu}-d_{A^{\nu}}\Psi^{\nu}\right)\right]\\
&-*\left[\left(A^{\nu}-A^0\right)\wedge *\left(\partial_tA^{0}-d_{A^{0}}\Psi^{0}\right)\right]\\
&-*\left[*(A^\nu-A^0)\wedge d_{A^0}(\Psi^\nu-\Psi^0)\right]
\end{split}
\end{equation}
allows us to compute the estimate using (\ref{crit:eq:supest1})
\begin{equation}\label{eqkk3}
\begin{split}
\left\|d_{A^0}\left(\Psi^{\nu}-\Psi^0\right)\right\|_{L^p}&+\left\|\Psi^{\nu}-\Psi^0 \right\|_{L^p}\leq c\left\|d_{A^0}*d_{A^0}\left(\Psi^{\nu}-\Psi^0\right)\right\|_{L^p}\\
\leq & c\|A^\nu-A^0\|_{L^p}+\|A^\nu-A^0\|_{L^\infty}\|d_{A^0}(A^\nu-A^0)\|_{L^p}\leq c\varepsilon_\nu^{{1-1/p}}.
\end{split}
\end{equation}
Furthermore, since, by (\ref{est:epsm}),
\begin{equation*}
\left\|\left(\partial_tA^{\nu}-d_{A^{\nu}}\Psi^{\nu}\right)-\left(\partial_tA^{0}-d_{A^0}\Psi^{0}\right)\right\|_{L^p}\leq c\varepsilon_\nu^{{1-1/p}},
\end{equation*}
\begin{equation}\label{eqkk4}
\left\|\nabla_t\left(A^{\nu}-A^{0}\right)\right\|_{L^p}\leq c\varepsilon_\nu^{{1-1/p}}.
\end{equation}
On the other side, we have
\begin{equation*}
\begin{split}
d_{A^0}^*d_{A^0}\nabla_t\left(\Psi^{\nu}-\Psi^0\right)=&
\nabla_td_{A^0}^*d_{A^0}\left(\Psi^{\nu}-\Psi^0\right)\\
&-*\left[\left(\partial_tA^0-d_{A^0}\Psi^0\right)\wedge *d_{A^0}\left(\Psi^{\nu}-\Psi^0\right)\right]\\
&+*\left[d_{A^0}*\left(\partial_tA^0-d_{A^0}\Psi^0\right)\wedge \left(\Psi^{\nu}-\Psi^0\right)\right]\\
&-*\left[*\left(\partial_tA^0-d_{A^0}\Psi^0\right)\wedge d_{A^0}\left(\Psi^{\nu}-\Psi^0\right)\right]
\end{split}
\end{equation*}
and deriving (\ref{eq:dfkdjfhdjas;}) by $\nabla_t$ we obtain
\begin{equation}\label{eqkkgg}
\begin{split}
\left\|\nabla_t\left(\Psi^{\nu}-\Psi^{0}\right)\right\|_{L^p}\leq& \left\| d_{A^0}^*d_{A^0}\nabla_t\left(\Psi^{\nu}-\Psi^0\right)\right\|_{L^p}\\
\leq & c \|d_{A^0}\nabla_t(\Psi^\nu-\Psi^0)\|_{L^p} +c \left\|\nabla_t\left(A^{\nu}-A^{0}\right)\right\|_{L^p}\\
&+\left\|\nabla_t\left(A^{\nu}-A^{0}\right)\right\|_{L^{2p}}\left\|d_{A^0}\left(\Psi^{\nu}-\Psi^{0}\right)\right\|_{L^{2p}}\\
&+c\|A^\nu-A^0\|_{L^\infty}\left(1+\frac 1{\varepsilon^2}\left\|d_{A^nu}d_{A^\nu}^*F_{A^\nu}\right\|_{L^2(\Sigma)}\right)\\
&+\|A^\nu-A^0\|_{L^\infty}\left\| d_{A^0}^*d_{A^0}\nabla_t\left(\Psi^{\nu}-\Psi^0\right)\right\|_{L^p}
\end{split}
\end{equation}
where in the second estimate we use that, by the perturbed Yang-Mills equations,
$$\|\nabla_t^{\Psi^\nu}(\partial_tA^\nu-d_{A^\nu}\Psi^\nu)\|_{L^p}\leq c+c\frac 1{\varepsilon^2}\left\|d_{A^\nu}^*F_{A^\nu}\right\|_{L^p(\Sigma)}\leq c+c\frac 1{\varepsilon^2}\left\|d_{A^nu}d_{A^\nu}^*F_{A^\nu}\right\|_{L^2(\Sigma)};$$
thus,
\begin{equation}\label{eqkk5}
\left\|\nabla_t\left(\Psi^{\nu}-\Psi^{0}\right)\right\|_{L^p}\leq c\varepsilon_\nu^{{1-1/p}}.
\end{equation}
Finally by the estimates (\ref{eqkk1}), (\ref{eqkk2}), (\ref{eq:g=0}), (\ref{eqkk3}), (\ref{eqkk4}) and (\ref{eqkk5}) we have 
\begin{equation}\label{eqndtobvfiua}
\left\|\Xi^{\nu}-\Xi^0\right\|_{1,p,1}+\varepsilon_\nu^{1/p}\left\|\Xi^{\nu}-\Xi^0\right\|_{L^\infty}\leq c\varepsilon_\nu^{{1-1/p}}
\end{equation}
wich proves the third step.
\end{proof}

{\bf Step 4. } Let $p>3$. There is sequence $\{g_\nu\}_{\nu\in\mathbb N }$ of gauge transformations $g_\nu\in\mathcal G^{2,p}_0(P\times S^1)$ such that 
\begin{equation}
d_{\Xi^0 }^{*_{\varepsilon_\nu } }(g_\nu^*\Xi^{\nu}-\Xi^0)=0,
\end{equation}
\begin{equation}
\left\|d_{A^0}^*(g_\nu^*A^{\nu }-A^0)\right\|_{L^p}\leq c \varepsilon_\nu^{3-1/p},\quad \left\|d_{A^0}(g_\nu^*A^{\nu }-A^0)\right\|_{L^p}\leq c \varepsilon_\nu^{2-2/p},
\end{equation}
and
\begin{equation}
\left\|g_\nu^*\Xi^{\nu}-\Xi^0\right\|_{1,p,1}+\varepsilon_\nu^{1/p}\left\|g_\nu^*\Xi^{\nu}-\Xi^0\right\|_{L^\infty}\leq c\varepsilon_\nu^{1-1/p}.
\end{equation}

\begin{proof}[Proof of step 4]
By the last step the perturbed Yang-Mills connections $\Xi^\nu$, that satisfy the estimate (\ref{eqndtobvfiua}) and in addition 
\begin{equation*}
 \varepsilon_\nu^{2}\left\|d_{\Xi^0 }^{*_\varepsilon}(\Xi^{\nu }-\Xi^0)\right\|_{L^p}
\leq \left\|d_{A^0}^*(A^{\nu }-A^0)\right\|_{L^p}+\varepsilon^2_\nu\left\|\nabla_t(\Psi^{\nu }-\Psi^0)\right\|_{L^p}\leq c \varepsilon_\nu^{3-1/p},
\end{equation*}
$$\left\|\Xi^{\nu}-\Xi^0 \right\|_{0,p,\varepsilon }\leq \left\|\Xi^{\nu}-\Xi^0\right\|_{1,p,\varepsilon }\leq c\varepsilon_\nu^{1-1/p}\leq \delta_0\varepsilon_\nu^{1/p}$$
for all $0<\varepsilon_\nu\leq\varepsilon_0$, $c\varepsilon_0^{1-2/p}\leq\delta_0$ and the $\delta_0$ given in theorem \ref{thm:crgc}; hence $\Xi^{\nu}$, $\Xi^0$ satisfy the assumption (\ref {eq:thm:crgc}) of theorem \ref{thm:crgc} with $q=p$. Therefore by this last theorem we can find a sequence $g_\nu\in \mathcal G^{2,p}_0(P\times S^1)$ such that 
\begin{equation*}
d_{\Xi^0 }^{*_{\varepsilon_\nu} }(g_\nu^*\Xi^{\nu }-\Xi^0)=0
\end{equation*}
and
\begin{equation}\label{crit:surjlnog}
\left\|g_\nu^*\Xi^{\nu }-\Xi^{\nu }\right\|_{1,p,\varepsilon }\leq c \varepsilon^{2}\left\|d_{\Xi^0 }^{*_\varepsilon}(\Xi^{\nu }-\Xi^0)\right\|_{L^p}\leq c \varepsilon_\nu^{3-1/p}
\end{equation}
and therefore, by the Sobolev theorem \ref{lemma:sobolev},

\begin{equation*}
\begin{split}
\varepsilon_\nu^{1/p} \left\|g_\nu^*\Xi^{\nu }-\Xi^0\right\|_{\infty,\varepsilon}\leq &c \left\|g_\nu^*\Xi^{\nu }-\Xi^0\right\|_{1,p,\varepsilon}\\
\leq &c \left( \left\|g_\nu^*\Xi^{\nu }-\Xi^{\nu } \right\|_{1,p,\varepsilon}+\left\|\Xi^{\nu }-\Xi^0\right\|_{1,p,\varepsilon}\right)\\
\leq &c\varepsilon_\nu^{1-1/p}.
\end{split}
\end{equation*}
The estimates (\ref{eqndtobvfiua}), (\ref{crit:surjlnog}) and the triangular inequality yield also to
$$\left\|d_{A^0}^*(g_\nu^*A^{\nu }-A^0)\right\|_{L^p}\leq \left\|d_{A^0}^*(A^{\nu }-A^0)\right\|_{L^p}+\left\|d_{A^0}^*(g_\nu^*A^{\nu }-A^\nu)\right\|_{L^p}\leq c \varepsilon_\nu^{3-1/p},$$

$$\left\|d_{A^0}(g_\nu^*A^{\nu }-A^0)\right\|_{L^p}\leq \left\|d_{A^0}(A^{\nu }-A^0)\right\|_{L^p}+\left\|d_{A^0}(g_\nu^*A^{\nu }-A^\nu)\right\|_{L^p}\leq c \varepsilon_\nu^{2-2/p},$$

$$\left\|g_\nu^*\Xi^{\nu }-\Xi^0 \right\|_{1,p,1}\leq \left\|\Xi^{\nu }-\Xi^0 \right\|_{1,p,1}+\frac 1{\varepsilon^2}\left\|g_\nu^*\Xi^{\nu }-\Xi^\nu \right\|_{1,p,\varepsilon}\leq c \varepsilon_\nu^{1-1/p}.$$
Thus, we concluded the proof of the fourth step.
\end{proof}
 We still denote the new sequence $g_\nu^*\Xi^\nu$ by $\Xi^\nu$ in order to semplify the notation.
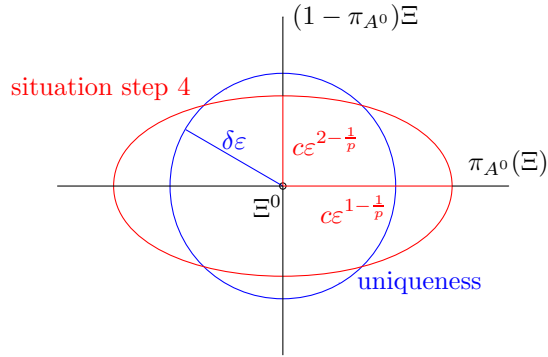
\begin{figure}[ht]
\begin{center}
\begin{tikzpicture}[scale=3] 
\draw[blue] (0,0) circle (0.5cm); 
\draw[blue] (0,0)--node[above]{$\delta\varepsilon$}(-0.43,0.25);
\draw[blue] (0.25,-0.43) circle (0.01pt) node[right]  {$\,\,$uniqueness} ; 

\draw (0,0) circle (0.4pt) node[below]{$\Xi^0\quad\,$};

\draw[red] (0.75,0) .. controls (0.75,0.25) and (0.4,0.4) .. (0,0.4); 
\draw[red] (-0.75,0) .. controls (-0.75,0.25) and (-0.4,0.4) .. (0,0.4); 
\draw[red] (0.75,0) .. controls (0.75,-0.25) and (0.4,-0.4) .. (0,-0.4); 
\draw[red] (-0.75,0) .. controls (-0.75,-0.25) and (-0.4,-0.4) .. (0,-0.4); 
\draw[red] (0,0) -- node[right]{$c\varepsilon^{2-\frac 1p}$} (0,0.4);
\draw[red] (0,-0) -- node[below]{$c\varepsilon^{1-\frac 1p}\quad$} (0.75,-0);
\draw[red] (-0.25,0.43) circle (0.01pt) node[left]  {situation step 4$\quad$} ; 

\draw (-1,0)--(0,0);
\draw (0.75,0)--(1,0);
\draw (1,0) circle (0.01pt) node[above]{$\pi_{A^0}(\Xi)$};

\draw (0,-0.75)--(0,0);
\draw (0,0.4)--(0,0.75);
\draw (0,0.75) circle (0.01pt) node[right]{$(1-\pi_{A^0})\Xi$};

\end{tikzpicture}
\caption{Uniqueness (circle) and the result of step 4 (ellipse).}
\end{center}
\end{figure}

{\bf Step 5. }There are three positive constants $\delta_1, \varepsilon_0, c$ such that for any positive $\varepsilon_\nu<\varepsilon_0$
\begin{equation}
\|\pi_{A^0}(A^\nu-A^0)\|_{L^2}+\|\pi_{A^0}(A^\nu-A^0)\|_{L^\infty }\leq c \varepsilon_\nu^{1+\delta_1 }.
\end{equation}

{\bf End of the proof:} Since our sequence  satisfies the assumptions of the uniqueness theorem \ref{thm:localuniqueness} because by the fourth step $d_{\Xi^0}^{*_\varepsilon}(\Xi^\nu-\Xi^0)=0$ and by the fourth and the last step 
$$\|\Xi^\nu-\Xi^0\|_{1,2,\varepsilon}+\|\Xi^\nu-\Xi^0\|_{\infty,\varepsilon}\leq \delta \varepsilon_\nu,$$
for $\nu$ big enough $\Xi^{\nu}=\mathcal T^{\varepsilon_\nu, b }(\Xi^0)$ which 
is a contradiction.

\begin{proof}[Proof of step 5]
In order to estimate the norms of $\pi_{A^0}(A^\nu-A^0)$ we use the estimate (\ref{eq:lemma:nonono}), i.e.
\begin{equation*}
\begin{split}
 \|\pi_A(A^\nu-A^0) \|_{L^2}&+\|\nabla_t\pi_A(A^\nu-A^0)\|_{L^2}\\
\leq&c \|\pi_A\big(\mathcal D^{\varepsilon_\nu}_1(\Xi^0)(A^\nu-A^0,\Psi^\nu-\Psi^0)+*[(A^\nu-A^0)\wedge *\omega]\big)\|_{L^2}\\
&+c \|A^\nu-A^0-\pi_A(A^\nu-A^0)\|_{L^2}\\
&+c\|\nabla_t(A^\nu-A^0-\pi_A(A^\nu-A^0))\|_{L^2}\\
&+c\varepsilon^2 \|\nabla_t(\Psi^\nu-\Psi^0)\|_{L^2}+\varepsilon^2\|\Psi^\nu-\Psi^0\|_{L^2}\\
&+c\varepsilon_\nu^2\|\mathcal{D}_2^{\varepsilon_\nu}(\Xi^0)(A^\nu-A^0,\Psi^\nu-\Psi^0)\|_{L^2}
\end{split}
\end{equation*}
which, by (\ref{eq:lemma125}), can be written as
\begin{equation}\label{crit:est:kkka}
\begin{split}
 \|\pi_A(A^\nu-A^0) \|_{L^2}&+\|\nabla_t\pi_A(A^\nu-A^0)\|_{L^2}\\
\leq&c \|\pi_A\big(\mathcal D^{\varepsilon_\nu}_1(\Xi^0+\varepsilon_\nu^2\alpha_0)(A^\nu-A^0,\Psi^\nu-\Psi^0)\big)\|_{L^2}\\
&+c \|(A^\nu-A^0)-\pi_A(A^\nu-A^0)\|_{1,2,\varepsilon_\nu }+c\varepsilon_\nu^{3-\frac 1p}\\
&+c\|\nabla_t((A^\nu-A^0)-\pi_{A^0}(A^\nu-A^0))\|_{L^2}\\
&+\varepsilon^2_\nu\|\mathcal{D}_2^{\varepsilon_\nu}(\Xi^0+\varepsilon_\nu^2\alpha_0)(A^\nu-A^0,\Psi^\nu-\Psi^0)\|_{L^2}
\end{split}
\end{equation}
where $\alpha_0\in \textrm{im } d_{A^0}^*$ is defined in lemma \ref{lemma:step1xi1} choosing $\varepsilon=1$ and satisfies
\begin{equation}\label{crit:surj:est:a0}
\|\alpha_0\|_{2,2,1}
+\|\alpha_0\|_{L^\infty}\leq c;
\end{equation}
we denote $\Xi^{1,\nu}=\Xi^0+\varepsilon_\nu^2\alpha_0=A^{1,\nu}+\Psi^0dt$ and we recall also that, always by lemma \ref{lemma:step1xi1},
\begin{equation*}
\|\mathcal F^\varepsilon_1(\Xi^{1,\nu})\|_{L^2}\leq c\varepsilon^2,\quad \|\mathcal F^\varepsilon_2(\Xi^{1,\nu})\|_{L^2}\leq c.
\end{equation*}
 In the following, we will work with the difference $\Xi^\nu-\Xi^{1,\nu}=\tilde\alpha^\nu+\tilde\psi^\nu dt+\tilde \phi^\nu ds$ which by step 4 and (\ref{crit:surj:est:a0}) satisfies
\begin{equation}\label{fgdsklhru}
\|\Xi^\nu-\Xi^{1,\nu}\|_{1,2,1}
+\varepsilon_\nu^{1/p}\|\Xi^\nu-\Xi^{1,\nu}\|_{L^\infty }
\leq c \varepsilon_\nu^{1-1/p}.
\end{equation}
Furthermore we consider the decomposition 
\begin{equation}
A^\nu-A^{1,\nu}=(A^\nu-A^\nu_1)+(A^\nu_1-A^0)+(A^0-A^{1,\nu})=\alpha^\nu+\bar\alpha^\nu-\varepsilon_\nu^2\alpha_0=\tilde \alpha^\nu
\end{equation}
where $\alpha^\nu=A^\nu-A^\nu_1$ is the 1-form defined in the first step and $\bar\alpha^\nu:=A^\nu_1-A^0$.\\

The idea of the proof is to use the situation described in the picture \ref{picsur1} and in order to compute the norms of $A^\nu-A^0$ we use the properties of the orthogonal splitting $H_{A^0}^1\oplus \textrm{im } d_{A^0}\oplus \textrm{im } d_{A^0}^*$ combined with the facts that $\alpha^\nu\in \textrm{im } d_{A_1^\nu}^*$ and that the norm of $\Pi_{\textrm{im } d_{A^0}^*}(\bar\alpha^\nu)$ can be estimate using the identity $d_{A^0}\bar\alpha^\nu=-\frac 12 [\bar\alpha^\nu\wedge \bar\alpha^\nu]$ which can be deduced from the flat curvatures $F_{A_1^\nu}$ and $F_{A^0}$.

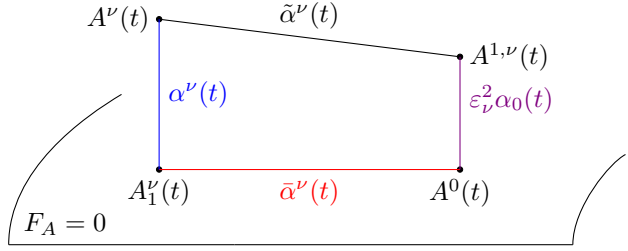
\begin{figure}[ht]
\begin{center}
\begin{tikzpicture} 

\draw (-1,-1) .. controls (-1,0) and (0,0.7) .. (0.5,1); 
\draw (6.5,-1) .. controls (6.5,-0.5) and (7,0.1) .. (7.2,0.2); 
\draw (-1,-1).. controls (0,-1) and (1,-1) .. node[near start,above] {$F_A=0$}(2,-1); \draw (2,-1).. controls (3,-1) and (4,-1) .. (6.5,-1); 

\draw (1,2) .. controls (3,1.75) and (3,1.75) ..node[above] {$\tilde \alpha^\nu(t)$}(5,1.5); 

\filldraw (1,0) circle (1pt) node[below] {$A_1^\nu(t)$}
(5,0) circle (1pt) node[below] {$A^0(t)$}
(5,1.5) circle (1pt) node[above,right] {$A^{1,\nu}(t)$}
(1,2) circle (1pt) node[above, left] {$A^\nu(t)$}; 
\draw[red] (1,0) .. controls (2,0) and (4,0) .. node[below] {$\bar \alpha^\nu(t)$}(5,0); 
\draw[violet] (5,0) .. controls (5,1) and (5,1) .. node[right] {$\varepsilon_\nu^2\alpha_0(t)$}(5,1.5); 
\draw[blue] (1,0) .. controls (1,1) and (1,1) .. node[right] {$\alpha^\nu(t)$}  (1,2); 

\end{tikzpicture} 
\caption{The splitting of the fifth step.}\label{picsur1}
\end{center}
\end{figure}

{\bf Claim 1:} $\|\tilde \alpha^\nu-\pi_{A^0}(\tilde\alpha^\nu)\|_{1,2,\varepsilon_\nu}\leq c\varepsilon_{\nu}^{2-2/p}$.

\begin{proof}[Proof of claim 1]
By the triangular inequality and $d_{A^0}^*\alpha_0=0$ we obtain
\begin{equation}\label{fgdsklhru1}
\|d_{A^0}(A^\nu-A^{1,\nu})\|_{L^2}\leq \varepsilon_\nu^2\|d_{A^0}\alpha_0\|_{L^2}
+\|d_{A^0}(A^\nu-A^0)\|_{L^2}
\leq c\varepsilon_{\nu}^{2-3/p},
\end{equation}
\begin{equation}\label{fgdsklhru2}
\|d_{A^0}^*\tilde\alpha^\nu\|_{L^2}\leq \|d_{A^0}^*(A^\nu-A^0)\|_{L^2}+\varepsilon^2\|d_{A^0}^*\alpha_0\|_{L^2}\leq c\varepsilon_\nu^{3-1/p}.
\end{equation}
and therefore by (\ref{fgdsklhru}), (\ref{fgdsklhru1}) and (\ref{fgdsklhru2})
\begin{equation}
\|\tilde \alpha^\nu-\pi_{A^0}(\tilde\alpha^\nu)\|_{1,2,\varepsilon_\nu}\leq c\varepsilon_{\nu}^{2-2/p}.
\end{equation} 
\end{proof}

{\bf Claim 2:} $\varepsilon_\nu^2\|\mathcal D_2^{\varepsilon_\nu}(\Xi^{1,\nu})
(\tilde\alpha^\nu,\tilde\psi^\nu)\|_{L^2}\leq c \varepsilon^{2-3/p}$.

\begin{proof}[Proof of claim 2]
The estimate follows from
\begin{equation*}
\mathcal D_2^{\varepsilon_\nu}(\Xi^{1,\nu})(\tilde\alpha^\nu,\tilde\psi^\nu)
=-C_2^{\varepsilon_\nu}(\Xi^{1,\nu})(\tilde\alpha^\nu,\tilde\psi^\nu)-
\mathcal F_2^{\varepsilon_\nu}(\Xi^{\nu,1}),
\end{equation*}
where $\| \mathcal F_2^{\varepsilon_\nu}(\Xi^{\nu,1})\|_{L^2}\leq c$ and the quadratic estimates of the lemma \ref{lemma:estimate:c}.
\end{proof}

{\bf Claim 3:} 
\begin{equation}
\begin{split}
\|\pi_{A^0}(\mathcal D_1^{\varepsilon_\nu}(\Xi^{1,\nu})(\tilde\alpha^\nu,\tilde\psi^\nu)\|_{L^2}\leq &c \varepsilon^{2-3/p}+\frac 1{\varepsilon^2}
\|\pi_{A^0}([\tilde\alpha^\nu\wedge *d_{A^0}(\tilde\alpha^\nu-\bar\alpha^\nu)])\|_{L^2}\\
&+\varepsilon_\nu^{1/2-2/p}\|\pi_{A^0}\tilde\alpha^\nu\|_{L^2}.
\end{split}
\end{equation}

\begin{proof}[Proof of the claim 3]
By $\| \mathcal F_1^{\varepsilon_\nu}(\Xi^{\nu,1})\|_{L^2}\leq c\varepsilon_\nu^2$ and by the identity
\begin{equation*}
\mathcal D_1^{\varepsilon_\nu}(\Xi^{1,\nu})(\tilde\alpha^\nu,\tilde\psi^\nu)
=-C_1^{\varepsilon_\nu}(\Xi^{1,\nu})(\tilde\alpha^\nu,\tilde\psi^\nu)-
\mathcal F_1^{\varepsilon_\nu}(\Xi^{\nu,1}),
\end{equation*}
we have
\begin{equation}
\begin{split}
\|\pi_{A^0}(\mathcal D_1^{\varepsilon_\nu}(\Xi^{1,\nu})(\tilde\alpha^\nu,\tilde\psi^\nu)\|_{L^2}
\leq & \|\mathcal F_1^{\varepsilon_\nu}(\Xi^{1,\nu})\|_{L^2}
+\|\pi_{A^0}(C_1^{\varepsilon_\nu}(\Xi^{1,\nu})(\tilde\alpha^\nu,\tilde\psi^\nu))\|_{L^2}\\
\leq& c \varepsilon^{2-3/p}+\frac 1{\varepsilon^2}
\|\pi_{A^0}([\tilde\alpha^\nu\wedge *(d_{A^0}\tilde\alpha^\nu+\frac 12[\tilde\alpha^\nu\wedge \tilde\alpha^\nu])])\|_{L^2}\\
\leq &c \varepsilon^{2-3/p}+\frac 1{\varepsilon^2}
\|\pi_{A^0}([\tilde\alpha^\nu\wedge *d_{A^0}(\tilde\alpha^\nu-\bar\alpha^\nu)])\|_{L^2}\\
&+\varepsilon_\nu^{1/2-2/p}\|\pi_{A^0}\tilde\alpha^\nu\|_{L^2}
\end{split}
\end{equation}
where for the second inequality we estimate every term of $C_1^{\varepsilon_\nu}(\Xi^{1,\nu})(\tilde\alpha^\nu,\tilde\psi^\nu)$ using the formula (\ref{eq:c1term2})  and for the third one we we applied  
\begin{equation*}
0=F_{A^0+\bar\alpha^\nu }=d_{A^0}\bar\alpha^\nu+\frac 12[\bar\alpha^\nu\wedge\bar\alpha^\nu], \quad
\|\alpha^\nu\|_{L^2(\Sigma)}+\|d_{A^\nu }\alpha^\nu \|_{L^2(\Sigma)}\leq c\varepsilon^{2-1/p}
\end{equation*}
and the decomposition of $\tilde\alpha^\nu$.

\end{proof}

{\bf Claim 4:} $\|\nabla_t(\bar\alpha^\nu-\pi_{A^0}(\bar\alpha^\nu))\|_{L^2}\leq c\varepsilon_\nu^{2-3/p}$.

\begin{proof}[Proof of claim 4]
We denote by $\Pi_{\textrm{im } d_{A^0}}$ and $\Pi_{\textrm{im } d_{A^0}^*}$ respectively the projections on the linear spaces $\textrm{im } d_{A^0}$ and $\textrm{im } d_{A^0}^*$ using the orthogonal splitting (\ref{split}). For $\bar\alpha^\nu-\pi_{A^0}\bar\alpha^\nu=d_{A^0}\bar\gamma^\nu+d_{A^0}\omega^\nu$, where $\gamma$ is a 0-form and $\omega$ a 2-form, we then have that
\begin{equation*}
\begin{split}
\|\nabla_t(\bar\alpha^\nu-\pi_{A^0}(\bar\alpha^\nu))\|_{L^2}\leq &c\|\bar\alpha^\nu-\pi_{A^0}(\bar\alpha^\nu)\|_{L^2}+\left\|\Pi_{\textrm{im } d_{A^0}}\left( \nabla_td_{A^0}\bar\gamma^\nu\right)\right\|_{L^2}\\
&+\left\| \Pi_{\textrm{im } d_{A^0}^*}\left(\nabla_td_{A^0}^*\omega^\nu \right)\right\|_{L^2}\\
\leq& c\varepsilon_\nu^{2-3/p}
\end{split}
\end{equation*}
where the last estimate follows from the next two:
\begin{equation*}
\begin{split}
\left\| \Pi_{\textrm{im } d_{A^0}^*}\left(\nabla_td_{A^0}^*\omega^\nu\right) \right\|_{L^2}
\leq& \left\| d_{A^0}\nabla_td_{A^0}^*\omega^\nu\right\|_{L^2}\\
\leq& \left\| \nabla_td_{A^0}\bar \alpha^\nu\right\|_{L^2}+\|\partial_tA^0-d_{A^0}\Psi^0\|_{L^\infty} \|\bar\alpha^\nu-\pi_{A^0}\bar\alpha^\nu\|_{L^2}\\
\leq& \left\|\nabla_t[\bar\alpha^\nu\wedge\bar\alpha^\nu]\right\|_{L^2}+c\|\bar\alpha^\nu-\pi_{A^0}\bar\alpha^\nu\|_{L^2}\\
\leq& c\|\bar\alpha^\nu\|_{L^\infty}\|\nabla_t\bar\alpha^\nu\|_{L^2}+c\|\bar\alpha^\nu-\pi_{A^0}\bar\alpha^\nu\|_{L^2}\leq  c\varepsilon_\nu^{2-3/p},
\end{split}
\end{equation*}
\begin{equation*}
\begin{split}
\left\| \Pi_{\textrm{im } d_{A^0}}\left(\nabla_td_{A^0}\bar\gamma^\nu\right)\right\|_{L^2}
\leq &\left\| \Pi_{\textrm{im } d_{A^0}}\left(\nabla_t\left(\Pi_{\textrm{im } d_{A^0}}\left(\tilde\alpha^\nu\right)\right) \right)\right\|_{L^2}\\
&+\left\|\Pi_{\textrm{im } d_{A^0}}\left( \nabla_t\left(\Pi_{\textrm{im } d_{A^0}}\left(\alpha^\nu\right) \right)\right)\right\|_{L^2}\\
\leq & c\left\| d_{A^0}^*\nabla_t\left(\Pi_{\textrm{im } d_{A^0}}\left(\tilde\alpha^\nu\right)\right)\right\|_{L^2}\\
&+\left\| \Pi_{\textrm{im } d_{A^0}}\left(\nabla_t\left(\Pi_{\textrm{im } d_{A^0}}\left(*[\tilde\alpha^\nu,\gamma^\nu]\right)\right)\right) \right\|_{L^2}\\
\leq &  \left\|\nabla_td_{A^0}^*\tilde\alpha^\nu\right\|_{L^2}+\|\partial_tA^0-d_{A^0}\Psi^0\|_{L^\infty}\left\|\Pi_{\textrm{im } d_{A^0}}\left(\tilde\alpha^\nu\right)\right\|\\
&+c\left\|\left[\tilde \alpha^\nu,\gamma^\nu\right]\right\|_{L^2}+\left\|\nabla_t\left[\tilde\alpha^\nu,\gamma^\nu\right]\right\|_{L^2}\\
\leq & c\varepsilon_\nu^{2-2/p}.
\end{split}
\end{equation*}
\end{proof}

{\bf Claim 5:} $\varepsilon_\nu \|\nabla_t(\tilde \alpha^\nu-\pi_{A^0}\alpha^\nu-\bar\alpha^\nu)\|_{0,2,\varepsilon_\nu }+\|d_{A^0}(\tilde\alpha^\nu-\bar\alpha^\nu)\|_{0,2,\varepsilon_\nu }\leq c\varepsilon_\nu^{3-6/p }$.\\

Therefore, using (\ref{crit:est:kkka}) and (\ref{crit:surj:est:a0}), we can estimate the norm of the harmonic part by
\begin{align*}
\|\pi_A(\tilde \alpha^\varepsilon)\|_{L^2}&+\|\nabla_t\pi_A(\tilde \alpha^\varepsilon)\|_{L^2}\\
\leq &
\|\pi_{A^0}(\mathcal D_1^{\varepsilon_\nu}(\Xi^{1,\nu})(\tilde\alpha^\nu,\tilde\psi^\nu)\|_{L^2}+\|\tilde\alpha^\nu-\pi_{A^0}(\tilde\alpha^\nu)\|_{1,2,\varepsilon}\\
&+\|\nabla_t(\tilde\alpha^\nu-\pi_{A^0}(\tilde\alpha^\nu))\|_{L^2}
+\varepsilon_\nu^2 \|\mathcal D_2^{\varepsilon_\nu}(\Xi^{1,\nu})(\tilde\alpha^\nu,\tilde\psi^\nu)\|_{L^2}\\
\intertext{by the first thee claims}
\leq&c \varepsilon_\nu^{2-3/p}+\frac1{\varepsilon_\nu^2}\|\pi_{A^0}([\tilde \alpha^\nu\wedge *d_{A^0}(\tilde\alpha^\nu-\bar\alpha^\nu)]) \|_{L^2}\\
&+\|\nabla_t(\tilde\alpha^\nu-\pi_{A^0}(\tilde\alpha^\nu))\|_{L^2}+\varepsilon_\nu^{1/2-2/p}\|\pi_{A^0}\tilde\alpha^\nu\|_{L^2}\\
\leq& c \varepsilon_\nu^{2-5/p}+\frac1{\varepsilon_\nu^2}\|\tilde \alpha^\nu\|_{L^\infty}\|d_{A^0}(\tilde\alpha^\nu-\bar\alpha^\nu)\|_{L^2}\\
&+\|\nabla_t(\tilde\alpha^\nu-\pi_{A^0}\alpha^\nu-\bar\alpha^\nu)\|_{L^2}\\
&+\|\nabla_t(\bar\alpha^\nu-\pi_{A^0}\bar\alpha^\nu)\|_{L^2}+\varepsilon_\nu^{1/2-2/p}\|\pi_{A^0}\tilde\alpha^\nu\|_{L^2}\\
\intertext{and because of the fourth and of the fifth claim we can conclude}
\leq & c \varepsilon_\nu^{2-6/p}+\varepsilon_\nu^{1/2-2/p}\|\pi_{A^0}\tilde\alpha^\nu\|_{L^2}.
\end{align*}
We finish therefore the proof of the fifth step by choosing $p>6$.
\end{proof}

\begin{proof}[Proof of claim 5]
We choose an operator as follows.
\begin{equation*}
\begin{split}
Q^{\varepsilon_\nu}(\Xi^0)\left(\tilde\alpha^\nu,\tilde\psi^\nu\right):=&\mathcal D^{\varepsilon_\nu}(\Xi^0)\left(\tilde\alpha^\nu,\tilde\psi^{\nu}\right)+\frac1{2{\varepsilon_\nu}^2}d_{A^0}^*[\bar \alpha^\nu\wedge \bar\alpha^\nu]\\
&+*\frac1{\varepsilon_\nu^2 }\left[\tilde\alpha^\nu\wedge(d_{A^0}\tilde\alpha^\nu+\frac12\left[\bar \alpha^\nu\wedge\bar \alpha^\nu]\right)\right]
\end{split}
\end{equation*}
Since $d_{A^0}\bar\alpha^\nu+\frac12[\bar\alpha^\nu\wedge\bar\alpha^\nu]=0$, 
\begin{equation}
d_{A^0}^*d_{A^0}\alpha^\nu+\frac12d_{A^0}^*[\bar \alpha^\nu\wedge\bar \alpha^\nu]
=d_{A^0}^*d_{A^0}(\alpha^\nu-\bar\alpha^\nu) 
\end{equation}
and  $\|d_{A^\nu }^*\bar\alpha^\nu \|_{L^2}\leq c \varepsilon_\nu^{3-1/p}$ by
\begin{equation}
\begin{split}
\|d_{A^\nu }^*\bar\alpha^\nu \|_{L^2}\leq&\|d_{A^\nu }^*(A^\nu-A^0)\|_{L^2}+\|d_{A^\nu }^*(A^\nu-A^0-\bar\alpha^\nu)\|_{L^2}\\
\leq& c \varepsilon_\nu^{3-1/p}+\| d_{A^\nu }^**d_{A^\nu }\gamma^\nu \|_{L^2}\\
\leq&  c \varepsilon_\nu^{3-1/p}+2\| F_{A^\nu }\|_{L^\infty}\|\gamma^\nu \|_{L^2}\leq c \varepsilon_\nu^{3-1/p}
\end{split}
\end{equation}
and hence
\begin{equation}
\|d_{A^0 }^*\bar\alpha^\nu \|_{L^2}\leq\|d_{A^\nu }^*\bar\alpha^\nu\|_{L^2}+c\|\alpha^\nu \|_{L^4 }\|\bar\alpha^\nu \|_{L^4}\leq c\varepsilon_\nu^{3-2/p}
\end{equation}
\begin{equation}
\langle d_{A^0}d_{A^0}^*\tilde\alpha^\nu,\tilde\alpha^\nu-\bar\alpha^\nu\rangle\geq \|d_{A^0}^*(\tilde \alpha^\nu-\bar\alpha^\nu)\|^2_{L^2}-c\|d_{A^0}^*\bar\alpha\|_{L^2}\|d_{A^0}^*(\tilde \alpha^\nu-\bar\alpha^\nu
)\|_{L^2}.
\end{equation}

Then 
\begin{equation}
\begin{split}
\varepsilon_\nu^2\langle&Q^{\varepsilon_\nu}_1(\Xi^0)\left(\tilde\alpha^\nu,\tilde\psi^\nu\right),\tilde\alpha^\nu-\pi_{A^0}\alpha^\nu-\bar\alpha^\nu \rangle
\geq \| d_{A^0}^*(\tilde\alpha^\nu-\bar\alpha^\nu)\|_{L^2}^2\\
&+\| d_{A^0}(\alpha^\varepsilon-\bar\alpha^\varepsilon)\|_{L^2}^2+\frac{\varepsilon_\nu^2 }2\|\nabla_t(\tilde\alpha^\nu-\pi_{A^0}\alpha^\nu-\bar\alpha^\nu)\|_{L^2}^2\\
&-c\varepsilon_\nu^{3-2/p}\|d_{A^0}^*(\tilde\alpha^\nu-\bar\alpha^\nu)\|_{L^2}\\
&-c\varepsilon_\nu^2\|\tilde\alpha^\nu\|_{L^2}\|\alpha^\varepsilon-\pi_{A^0}\alpha^\nu-\bar\alpha^\nu\|_{L^2}-c\varepsilon_\nu^2\|\tilde\alpha^\nu-\pi_{A^0}\alpha^\nu-\bar\alpha^\nu\|_{0,2,\varepsilon_\nu}\|\tilde\psi^\nu\|_{0,2,\varepsilon_\nu}\\
&-c\varepsilon_\nu^2|\langle \nabla_t\pi_{A^0}(\tilde \alpha^\nu),\nabla_t(\tilde\alpha^\nu-\pi_{A_0}\alpha^\nu-\bar\alpha^\nu) \rangle|\\
&-c\varepsilon^2|\langle \nabla_t(\bar\alpha^\nu-\pi_{A^0}(\bar \alpha^\nu)),\nabla_t(\tilde\alpha^\nu-\bar\alpha^\nu-\pi_{A^0}\alpha^\nu) \rangle|\\
&-\|\tilde\alpha^\nu \|^2_{L^\infty }\|d_{A^0}(\tilde\alpha^\nu-\bar\alpha^\nu) \|_{L^2 }\|\tilde\alpha^\nu-\pi_{A^0}\alpha^\nu-\bar\alpha^\nu \|_{L^2 }.
\end{split}
\end{equation}
We can conclude therefore that
\begin{equation}
\begin{split}
\varepsilon_\nu \|\nabla_t(\tilde \alpha^\nu-\pi_{A^0}\alpha^\nu-\bar\alpha^\nu)\|_{0,2,\varepsilon_\nu }&+\|d_{A^0}(\tilde\alpha^\nu-\bar\alpha^\nu)\|_{0,2,\varepsilon_\nu }\\
\leq& \varepsilon_\nu^2 \|Q^{\varepsilon_\nu}(\Xi^0)\left(\tilde \alpha^{\nu},\tilde \psi^{\nu}\right)\|_{0,2,\varepsilon_\nu }+c\varepsilon_\nu^{3-2/p}\\
&+c\varepsilon_\nu^2\|\tilde \alpha^{\nu} \|_{0,2,\varepsilon_\nu }+c\varepsilon_\nu^2\|\tilde\psi^{\nu} \|_{0,2,\varepsilon_\nu }\\
&+c\varepsilon_\nu^2\|\nabla_t\pi_{A^0}(\alpha^{\nu})\|_{0,2,\varepsilon_\nu }\\
&+c\varepsilon_\nu^2\|\pi_{A^0}(\alpha^{\nu})\|_{0,2,\varepsilon_\nu }\leq c\varepsilon_\nu^{3-6/p }
\end{split}
\end{equation}
where the last step follows because
\begin{equation*}
\|Q_1^{\varepsilon_\nu}(\Xi^{1,\nu})(\alpha^\nu,\phi^\nu)-Q_1^{\varepsilon^\nu}(\Xi^0)(\alpha^\nu,\phi^\nu) \|_{L^2}\leq c \|\alpha^\nu \|_{1,2,\varepsilon }+c \|\pi_{A^0}(\alpha^\nu)\|_{L^\infty}\|\pi_{A^0}(\alpha^\nu) \|_{L^2}
\end{equation*}
and
\begin{equation*}
\begin{split}
Q_1^{\varepsilon_\nu}(\Xi^{1,\nu})(\Xi^\nu-\Xi^1)=&-\mathcal F_1^{\varepsilon_\nu}(\Xi^{1,\nu})-C_1(\Xi^{1,\nu})(\Xi^\nu-\Xi^1)\\
&+\frac1{2\varepsilon^2_\nu }d_{A^0}^*[\bar \alpha^\nu\wedge\bar\alpha^\nu)]+*\frac1{\varepsilon_\nu^2 }[\alpha^\nu,*[\bar \alpha^\nu\wedge \bar \alpha^\nu]]
\end{split}
\end{equation*}
whose norm can be bounded by $c\varepsilon_\nu^{1-6/p }$ by the triangular and the H\"older inequalities. 

\end{proof}
\end{proof}


\section{Proof of the main theorem}\label{section:themaintheorem}

The theorem \ref{thm:mainthm} states the bijectivity of the map $\mathcal T^{\varepsilon,b}$ which follows directly from its definition \ref{thm:defT} and the theorem \ref{thm:surjj}, which prove its surjectivity, and in addition it shows that $\mathcal T^{b,\varepsilon}$ maps perturbed closed geodesics of Morse index $k$ in to perturbed Yang-Mills of the same Morse index.

\begin{theorem}\label{thm:indexym}
We choose a regular value $b>0$ of the energy $E^H$ and an $\varepsilon_0>0$ as in definition \ref{thm:defT}, then there is a constant $c>0$ such that for every $\Xi^0=A^0+\Psi^0 dt \in \mathrm{Crit}_{E^H}^b$ the following holds. Let $\Xi^\varepsilon=A^\varepsilon+\Psi^\varepsilon dt:=\mathcal T^{\varepsilon,b}(\Xi^0)$, $0<\varepsilon<\varepsilon_0$, then 
\begin{enumerate}
\item $\varepsilon^2\langle \alpha+\psi dt, \mathcal D^\varepsilon(\Xi^\varepsilon)(\alpha+\psi dt)\rangle\geq c\|\alpha+\psi dt \|_{1,2,\varepsilon}^2$ for any $1$-form $\alpha(t)+\psi(t) dt\in  d_{A^0}\Omega^0(\Sigma,\mathfrak g_P)\oplus d_{A^0 }^*\Omega^2(\Sigma,\mathfrak g_P)\oplus \Omega^0(\Sigma,\mathfrak g_P)\,dt$;
\item $\mathrm{index}_{E^H}(\Xi^0)=\mathrm{index}_{\mathcal{YM}^{\varepsilon,H}}(\Xi^\varepsilon)$.
\end{enumerate}
\end{theorem}  

\begin{proof}
As we have already mentioned, Weber in \cite{MR1930985} proved that the Morse index of a perturbed geodesic is finite and for a generic Hamiltonian $H_t$ its nullity is zero. We are therefore interested in the behavior of the operator $\mathcal D^\varepsilon(\Xi^\varepsilon)$ respect to $\mathcal D^0(\Xi^0)$ and in order to investigate that we consider the two parts of the orthogonal splitting of the 1-forms
\begin{equation*}
\begin{split}
\Omega^1(\Sigma,\mathfrak g_P)=&\left(d_{A^0}\Omega^0(\Sigma,\mathfrak g_P)\oplus d_{A^0}^*\Omega^2(\Sigma,\mathfrak g_P)\oplus \Omega^0(\Sigma,\mathfrak g_P)\,dt\right)\\
&\oplus H^1_{A^0}(\Sigma,\mathfrak g_P).
\end{split}
\end{equation*}
We also recall that 
$$\|\Xi^\varepsilon-\Xi^0\|_{2,p,\varepsilon}+\varepsilon^{\frac 1p}\|\Xi^\varepsilon-\Xi^0\|_{\infty,\varepsilon}\leq c\varepsilon^{2}, \quad\left\|d_{A^0}(A^\varepsilon-A^0-\alpha_0^\varepsilon)\right\|_{L^2}\leq c\varepsilon^4$$ by the theorem \ref{thm:existence} and the Sobolev estimate (\ref{eq:sobolev1}) provided that $\varepsilon$ is sufficiently small where $\alpha_0^\varepsilon$ is defined in the theorem \ref{thm:existence}. Integrating by parts we obtain for a $\alpha+\psi dt\in  d_{A^0}\Omega^0(S^1,M,\mathfrak g_P)\oplus d_{A^0}^*\Omega^2(S^1,M,\mathfrak g_P)\oplus \Omega^0(S^1,M,\mathfrak g_P)\,dt$:

\begin{equation}\label{crit:mt:eq1}
\begin{split}
\varepsilon^2\langle \alpha&+\psi dt, \mathcal D^\varepsilon(\Xi^\varepsilon)(\alpha+\psi dt)\rangle\\
= &\varepsilon^2\langle \alpha+\psi dt, \mathcal D^\varepsilon(\Xi^0)(\alpha+\psi dt)\rangle\\
&+\varepsilon^2\langle \alpha+\psi dt, \left(\mathcal D^\varepsilon(\Xi^\varepsilon)-\mathcal D^\varepsilon(\Xi^0)\right)(\alpha+\psi dt)\rangle\\
\geq &c\|\alpha+\psi dt \|_{1,2,\varepsilon}^2+\varepsilon^2\langle \alpha+\psi dt, \left(\mathcal D^\varepsilon(\Xi^\varepsilon)-\mathcal D^\varepsilon(\Xi^0)\right)(\alpha+\psi dt)\rangle\\
\geq&c\|\alpha+\psi dt \|_{1,2,\varepsilon}^2- c\varepsilon^{-\frac12} \|\Xi^\varepsilon-\Xi^0 \|_{1,2,\varepsilon}\|\alpha+\psi dt \|_{1,2,\varepsilon}\|\alpha+\psi dt \|_{0,2,\varepsilon }\\
\geq&c\|\alpha+\psi dt \|_{1,2,\varepsilon}^2- c \varepsilon^{3/2} \|\alpha+\psi dt \|_{1,2,\varepsilon}^2\\
\geq&c\|\alpha+\psi dt \|_{1,2,\varepsilon}^2
\end{split}
\end{equation}
where the third step follows by the quadratic estimates of the lemma \ref{lemma:diffoperator} from the Sobolev estimate of lemma \ref{lemma:sobolev} and the last one holds for $\varepsilon$ small enough. We choose now $\alpha(t) \in  H_{A^0}^1(\Sigma,\mathfrak g_P)$ and then we pick $\psi(t) \in \Omega^0(\Sigma,\mathfrak g_P)$, such that 
\begin{equation*}
d_{A^0}^*d_{A^0}\psi=-2*[\alpha\wedge *(\partial_tA^0-d_{A^0}\Psi^0)]
\end{equation*}
Then
\begin{equation*}
\langle \alpha+\psi dt, \mathcal D^\varepsilon(\Xi^\varepsilon)(\alpha+\psi dt)\rangle
= \langle \mathcal D^0(\Xi^0)(\alpha), \alpha\rangle
+\varepsilon^2\|\nabla_t\psi\|_{L^2}^2+Q
\end{equation*}
Where
\begin{equation*}
\begin{split}
Q:=&\frac 1{\varepsilon^2}\left\langle *\left[\alpha\wedge *\left(d_{A^0}(A^\varepsilon-A^0-\alpha_0)+\frac12 \left[(A^\varepsilon-A^0)\wedge(A^\varepsilon-A^0)\right]\right)\right],\alpha\right\rangle\\
&+\frac1{\varepsilon^2}\left\|\left[(A^\varepsilon-A^0)\wedge \alpha\right]\right\|_{L^2}^2
+\frac1{\varepsilon^2}\left\|\left[(A^\varepsilon-A^0)\wedge *\alpha\right]\right|_{L^2}^2\\
&+{\varepsilon^2}\|[(\Psi^\varepsilon-\Psi^0), \psi]\|_{L^2}^2-\langle d*X_t(A^\varepsilon)\alpha-d*X_t(A^0)\alpha,\alpha\rangle\\
&-\left\langle 2 \left[\psi,\left(\nabla_t(A^\varepsilon-A^0)-d_{A^0}(\Psi^\varepsilon-\Psi^0)-\left[(A^\varepsilon-A^0)\wedge (\Psi^\varepsilon-\Psi^0)\right]\right)\right],\alpha\right\rangle\\
&-\left\langle 2 *\left[\alpha\wedge *\left(\nabla_t(A^\varepsilon-A^0)-d_{A^0}(\Psi^\varepsilon-\Psi^0)\right)\right],\psi\right\rangle\\
&+\left\langle 2 *\left[\alpha\wedge *\left[(A^\varepsilon-A^0)\wedge (\Psi^\varepsilon-\Psi^0)\right]\right],\psi\right\rangle\\
&+\|[(A^\varepsilon-A^0),\psi]\|_{L^2}^2+\|[(\Psi^\varepsilon-\Psi^0)\wedge \alpha]\|_{L^2}^2
\end{split}
\end{equation*}
and hence
\begin{equation}\label{crit:mainthm:est:lsdghos}
|Q|\leq c_1 \varepsilon^{1/2}\left(\|\alpha\|_{L^2}^2+\|\nabla_t\alpha\|_{L^2}^2\right)
\end{equation}
for a positive constant $c_1$; in order to compute (\ref{crit:mainthm:est:lsdghos}) we need also to use
$$\|\psi\|_{L^2}\leq \|\alpha\|_{L^2}, \quad \|\alpha\|_{L^4}\leq \|\alpha\|_{L^2}+\|\nabla_t\alpha\|_{L^2}$$
where the first estimate follows from the definition of $\psi$ and the second from the Sobolev inequality. Therefore there is a constant $c>0$ such that if $\alpha$ is an element of the negative eigenspace of $\mathcal D^0(\Xi^0)$, then
\begin{equation}\label{crit:mt:eq2}
\begin{split}
\langle \alpha+\psi dt, &\mathcal D^\varepsilon(\Xi^\varepsilon)(\alpha+\psi dt)\rangle\\
\leq &-c\left(\|\alpha\|_{L^2}+\|\nabla_t\alpha\|_{L^2}\right)^2
+c_1 \varepsilon^{1/2}\left(\|\alpha\|_{L^2}^2+\|\nabla_t\alpha\|_{L^2}^2\right);
\end{split}
\end{equation}
and if $\alpha$ is in the positive eigenspace for $\mathcal D^0(\Xi^0)$, then
\begin{equation}\label{crit:mt:eq3}
\begin{split}
\langle \alpha+\psi dt,& \mathcal D^\varepsilon(\Xi^\varepsilon)(\alpha+\psi dt)\rangle\\
\geq &c\left(\|\alpha\|_{L^2}+\|\nabla_t\alpha\|_{L^2}\right)^2
-c_1 \varepsilon^{1/2}\left(\|\alpha\|_{L^2}^2+\|\nabla_t\alpha\|_{L^2}^2\right).
\end{split}
\end{equation}
Thus, by (\ref{crit:mt:eq1}), (\ref{crit:mt:eq2}) and (\ref{crit:mt:eq3}) the dimensions of the negative eigenspaces of $\mathcal D^0(\Xi^0)$ and $\mathcal D^\varepsilon(\Xi^\varepsilon)$ are equal provided that $\varepsilon$ is small enough and hence we can conclude that the Morse indices are equal.
\end{proof}

\begin{appendix}

\section{Estimates on the surface}

In this section we list some estimates that will be needed all along this exposition. The first two lemmas were proved in \cite{MR1283871} (lemma 7.6 and lemma 8.2) for $p> 2$ and $q=\infty$; the proofs in the case $p=2$ and $2\leq q <\infty$ is similar.

\begin{lemma}\label{lemma76dt94}
We choose $p> 2$ and $q=\infty$ or $p=2$ and $2\leq q <\infty$. Then there exist two positive constants $\delta$ and $c$ such that for every connection 
$A\in\mathcal A(P)$ with
$$\|F_A\|_{L^p(\Sigma) }\leq \delta$$
there are estimates 
$$\|\psi\|_{L^q(\Sigma)}\leq c\|d_A\psi\|_{L^p(\Sigma)},\qquad 
\|d_A\psi\|_{L^q(\Sigma)}\leq c\|d_A*d_A\psi\|_{L^p(\Sigma)},$$
for $\psi\in \Omega^0(\Sigma,\mathfrak g_P)$.
\end{lemma}

\begin{lemma}\label{lemma82dt94}
We choose $p> 2$ and $q=\infty$ or $p=2$ and $2\leq q <\infty$. Then there exist two positive constants $\delta$ and $c$ such that the following holds. 
For every connection $A\in \mathcal A(P)$ with 
$$\|F_A\|_{L^p(\Sigma)}\leq \delta$$
there exists a unique section $\eta \in \Omega^0(\Sigma,\mathfrak g_P)$ such that
$$F_{A+*d_A\eta}=0,\qquad \|d_A\eta\|_{L^q(\Sigma)}\leq c\|F_A\|_{L^p(\Sigma)}.$$
\end{lemma}

The following lemma is a symplified version of the lemma B.2. in \cite{MR1736219} where Salamon allows also to modify the complex structure on $\Sigma$ if it is $C^1$-closed to a fixed one.
\begin{lemma}\label{flow:lemma:lpl2}
Fix a connection $A^0\in \mathcal A_0(P)$. Then, for every $\delta>0$, $C>0$, and $p\geq2$, there exists a constant $c=c(\delta, C,A^0)\geq 1$ such that, if $A\in\mathcal A(P)$ satisfy $\|A-A^0\|_{L^\infty(\Sigma)}\leq C$ then, for every $\psi\in \Omega^0(\Sigma,\mathfrak g_P)$ and every $\alpha\in \Omega^1(\Sigma,\mathfrak g_P)$,
\begin{equation}
\|\psi\|_{L^p(\Sigma)}^p\leq\delta \|d_A\psi\|_{L^p(\Sigma)}^p+c\|\psi\|_{L^2(\Sigma)}^p,
\end{equation}
\begin{equation}
\|\alpha\|_{L^p(\Sigma)}^p\leq \delta\left(\|d_A\alpha\|_{L^p(\Sigma)}^p+\|d_A*\alpha\|_{L^p(\Sigma)}^p\right)+c\|\alpha\|_{L^2(\Sigma)}^p.
\end{equation}

\end{lemma}

\begin{lemma}\label{flow:lemma40} We choose $p\geq2$. There is a positive constant $c$ such that the following holds. For any connection $A\in \mathcal A_0(P)$ and any $\alpha\in \Omega^1(\Sigma,\mathfrak g_P)$
\begin{equation}
\begin{split}
\|\alpha\|_{L^p(\Sigma)}+\|d_A\alpha\|_{L^p(\Sigma)}&+\|d_A^*\alpha\|_{L^p(\Sigma)}+\|d_A^*d_A\alpha\|_{L^p(\Sigma)}+\|d_A^*d_A^*\alpha\|_{L^p(\Sigma)}\\
\leq& c\|(d_Ad_A^*+d_A^*d_A)\alpha\|_{L^p}+\|\pi_A(\alpha)\|_{L^p(\Sigma) }.
\end{split}
\end{equation}
\end{lemma}

\begin{proof}
For any flat connection $A$, the orthogonal splitting of $\Omega^1(\Sigma)=\textrm{im } d_A\oplus \textrm{im } d_A^*\oplus H_A^1(\Sigma,\mathfrak g_P)$ implies that there is a positive constant $c_0$ such that
\begin{equation*}
\|d_Ad_A^*\alpha\|_{L^p(\Sigma)}+\|d_A^*d_A\alpha\|_{L^p(\Sigma)}\leq c_0 \|(d_Ad_A^*+d_A^*d_A)\alpha\|_{L^p(\Sigma)};
\end{equation*}
thus, we can conclude the proof applying the lemma \ref{lemma76dt94}.
\end{proof}

\begin{lemma}\label{lemma:dAalphaest}
We choose $p\geq2$. There is a positive constant $c$ such that the following holds. For any $\delta>0$, any connection $A\in\mathcal A_0(P)$, $\alpha\in\Omega^1(\Sigma,\mathfrak g_P)$ and $\psi \in\Omega^0(\Sigma,\mathfrak g_P)$
\begin{equation}
\begin{split}
\left\| d_A\alpha\right\|_{L^p(\Sigma) }\leq c\left(\delta^{-1}\, \left\| \alpha  \right\|_{L^p(\Sigma)}+ \delta\, \left\| d_A^*d_A\alpha  \right\|_{L^p(\Sigma)}\right),\\
\left\| d_A^*\alpha\right\|_{L^p(\Sigma)}\leq c\left(\delta^{-1}\, \left\| \alpha  \right\|_{L^p(\Sigma)}+ \delta\, \left\| d_Ad_A^*\alpha  \right\|_{L^p(\Sigma)}\right),\\
\left\| d_A\psi\right\|_{L^p(\Sigma) }\leq c\left(\delta^{-1}\, \left\| \psi \right\|_{L^p(\Sigma)}+ \delta\, \left\| d_A^*d_A\psi  \right\|_{L^p(\Sigma)}\right).
\end{split}
\end{equation}
Furthermore, for any $\delta>0$, any connection $A+\Psi dt\in\mathcal A(P\times S^1)$, $\alpha+\psi dt\in\Omega^1(\Sigma\times S^1,\mathfrak g_P)$ 
\begin{equation}
\begin{split}
\varepsilon\left\| \nabla_t\alpha\right\|_{L^p(\Sigma\times S^1) }\leq c\left(\delta^{-1}\, \left\| \alpha  \right\|_{L^p(\Sigma\times S^1)}+ \delta\varepsilon^2 \left\|\nabla_t\nabla_t\alpha  \right\|_{L^p(\Sigma\times S^1)}\right),\\
\varepsilon^2\left\| \nabla_t\psi\right\|_{L^p(\Sigma\times S^1) }\leq c\left(\delta^{-1}\varepsilon\, \left\| \psi \right\|_{L^p(\Sigma\times S^1)}+ \delta\varepsilon^3 \left\| \nabla_t\nabla_t\psi  \right\|_{L^p(\Sigma\times S^1)}\right).
\end{split}
\end{equation}
\end{lemma}

\begin{proof}
The last two estimates follow analogously to the lemma D.4.  in \cite{MR2276534}. The first can be proved as follows. We choose $q$ such that $\frac 1p+\frac 1q=1$ then
\begin{equation*}
\begin{split}
\left\| d_A\alpha\right\|_{L^p(\Sigma)}=& \sup_{\bar \alpha} 
\frac {\langle d_A\alpha, \delta^{-1}\bar \alpha +\delta d_Ad_A^*\bar \alpha\rangle}{\left\| \delta^{-1}\bar\alpha+\delta d_Ad_A^*\bar\alpha\right\|_{L^q(\Sigma)}}\\
\leq & \sup_{\bar \alpha} 
\frac {c \langle\delta^{-1}\alpha+\delta d_A^*d_A\alpha, d_A^*\bar \alpha\rangle}{\delta^{-1}\left\|\bar\alpha\right\|_{L^q(\Sigma)}+\delta \left\|d_Ad_A^*\bar\alpha\right\|_{L^q(\Sigma)}+\left\| d_A^*\bar\alpha\right\|_{L^q}(\Sigma)}\\
\leq&\left(\delta^{-1}\left\|\alpha\right\|_{L^p(\Sigma)}+\delta \left\|d_A^*d_A\alpha\right\|_{L^p(\Sigma)}\right)
 \sup_{\bar \alpha}
 \frac {c \left\| d_A^*\bar \alpha\right\|_{L^q(\Sigma)}}{\left\| d_A^*\bar\alpha\right\|_{L^q(\Sigma)}}\\
 =&c\left(\delta^{-1}\left\|\alpha\right\|_{L^p(\Sigma)}+\delta \left\|d_A^*d_A\alpha\right\|_{L^p(\Sigma)}\right).
\end{split}
\end{equation*}
where the supremum is taken over all non-vanishing $1$-forms $\bar \alpha\in L^q$ with $d_Ad_A^*\bar \alpha\in L^q$. The norm $\left\| \delta^{-1}\bar\alpha+\delta d_Ad_A^*\bar\alpha\right\|_{L^q(\Sigma)}$ is never $0$ because
$$\left\| \delta^{-1}\bar\alpha+\delta d_Ad_A^*\bar\alpha\right\|_{L^2(\Sigma)}^2=\delta^{-2}\left\|\bar\alpha\right\|_{L^2(\Sigma)}^2+\delta^2 \left\|d_Ad_A^*\bar\alpha\right\|_{L^2(\Sigma)}^2+2\left\| d_A^*\bar\alpha\right\|_{L^2(\Sigma)}^2\neq 0,$$
otherwise we would have a contradiction by the H\"older inequality and the operator $\delta^{-1}+\delta d_Ad_A^*$ is surjective. The second and the third estimate of the lemma can be shown exactly in the same way.
\end{proof}

\end{appendix}

\bibliographystyle{amsplain}
\bibliography{bib}
\end{document}